\documentclass[a4paper,10pt]{article}
\usepackage{geometry}
\usepackage[english]{babel}
\usepackage{amsmath,amssymb,stmaryrd,enumerate,latexsym,theorem}

\catcode`\<=\active \def<{
\fontencoding{T1}\selectfont\symbol{60}\fontencoding{\encodingdefault}}
\catcode`\>=\active \def>{
\fontencoding{T1}\selectfont\symbol{62}\fontencoding{\encodingdefault}}
\newcommand{\assign}{:=}
\newcommand{\asterisk}{\mathord{*}}
\newcommand{\comma}{{,}}
\newcommand{\longhookrightarrow}{{\lhook\joinrel\relbar\joinrel\rightarrow}}
\newcommand{\mathd}{\mathrm{d}}
\newcommand{\nobracket}{}
\newcommand{\tmaffiliation}[1]{\\ #1}
\newcommand{\tmem}[1]{{\em #1\/}}
\newcommand{\tmemail}[1]{\\ \textit{Email:} \texttt{#1}}
\newcommand{\tmname}[1]{\textsc{#1}}
\newcommand{\tmop}[1]{\ensuremath{\operatorname{#1}}}
\newcommand{\tmsep}{, }
\newcommand{\tmtextbf}[1]{{\bfseries{#1}}}
\newcommand{\tmtextit}[1]{{\itshape{#1}}}
\newcommand{\nequiv}[0]{\not\equiv}
\newenvironment{enumeratenumeric}{\begin{enumerate}[1.] }{\end{enumerate}}
\newenvironment{itemizedot}{\begin{itemize} }{\end{itemize}}
\newenvironment{proof}{\noindent\textbf{Proof\ }}{\hspace*{\fill}$\Box$\medskip}
\newenvironment{proof*}[1]{\noindent\textbf{#1\ }}{\hspace*{\fill}$\Box$\medskip}
\newtheorem{theorem}{Theorem}
\newtheorem{corollary}[theorem]{Corollary}
\newtheorem{definition}[theorem]{Definition}
\newtheorem{lemma}[theorem]{Lemma}
\newtheorem{proposition}[theorem]{Proposition}
{\theorembodyfont{\rmfamily}\newtheorem{remark}[theorem]{Remark}}
\newcommand{\tmkeywords}{\textbf{Keywords:} }
\newcommand{\tmmsc}{\textbf{A.M.S. subject classification:} }

\newcommand{\descriptionparagraphs}[1]{\begin{description}#1\end{description}}

\begin{document}

\title{ Elliptic stochastic quantization}

\author{
  Sergio Albeverio, Francesco C. De Vecchi and Massimiliano Gubinelli
  \tmaffiliation{Hausdorff Center for Mathematics \& \\
  Institute for Applied Mathematics\\
  University of Bonn, Germany}
  \tmemail{albeverio@iam.uni-bonn.de}
  \tmemail{francesco.devecchi@uni-bonn.de}
  \tmemail{gubinelli@iam.uni-bonn.de}
}

\date{}

\maketitle

\begin{abstract}
  We prove an explicit formula for the law in zero of the solution of a class
  of elliptic (nonlinear) SPDE in $\mathbb{R}^2$. This formula is the simplest
  instance of \tmtextit{dimensional reduction}, discovered in the physics
  literature by Parisi and Sourlas~(1979), which links the law of an elliptic
  SPDE in $d + 2$ dimension with a Gibbs measure in $d$ dimensions. This
  phenomenon is similar to the relation between an $\mathbb{R}^{d + 1}$
  dimensional parabolic SPDE and its $\mathbb{R}^d$ dimensional invariant
  measure. As such, dimensional reduction of elliptic SPDEs can be considered
  a sort of \tmtextit{elliptic stochastic quantization} procedure in the sense
  of Nelson~(1966) and Parisi and Wu~(1981). Our proof uses in a fundamental
  way the representation of the law of the SPDE as a supersymmetric quantum
  field theory. Dimensional reduction for the supersymmetric theory was
  already established by Klein,~Landau~and~Perez~(1984). We fix a subtle gap
  in their proof and also complete the dimensional reduction picture by
  providing the link between the elliptic SPDE and the supersymmetric model.
  Even in our $d = 0$ context the arguments are non-trivial and a
  non-supersymmetric, elementary proof seems only to be available in the
  linear, i.e., Gaussian case.
  
  \ 
\end{abstract}

\tmmsc{60H15}{\tmsep}{81Q60}{\tmsep}{82B44}

\tmkeywords{stochastic quantization}{\tmsep}{elliptic stochastic partial
differential equations}{\tmsep}{dimensional reduction}{\tmsep}{Wiener
space}{\tmsep}{supersymmetry}{\tmsep}{Euclidean quantum field theory }

{\tableofcontents}

\section{Introduction}\label{section_introduction}

\tmtextit{Stochastic
quantization}~{\cite{damgaard_stochastic_1987,damgaard_stochastic_1988,parisi_perturbation_1981}}
broadly refers to the idea of sampling a given probability distribution by
solving a stochastic differential equation (SDE). This idea is both appealing
practically and theoretically since simulating or solving an SDE is sometimes
simpler than sampling or studying a given distribution. If, in finite
dimensions, this boils down mostly to the idea of the Monte Carlo Markov chain
method (which was actually invented before stochastic quantization), it is in
infinite dimensions that the method starts to have a real theoretical appeal.

It was Nelson~{\cite{nelson1966,MR0214150,Nelson1973}} and subsequently
Parisi and Wu~{\cite{parisi_perturbation_1981}} who advocated the constructive
use of stochastic partial differential equations (SPDEs) to realize a given
Gibbs measure for the use of Euclidean quantum field theory (QFT). Indeed the
original (parabolic) stochastic quantization procedure
of~{\cite{parisi_perturbation_1981}} can be understood as the equivalence
\begin{equation}
  \mathbb{E} [F (\varphi (t))] \propto \int F (\phi) e^{- S (\phi)}
  \mathcal{D} \phi \label{eq:dim-red-par} .
\end{equation}
Here $F$ belongs to a suitable space of real-valued test functions,
$\mathcal{D} \phi$ is an heuristic ``Lebesgue measure'' on $\mathcal{S}'
(\mathbb{R}^d)$, while on the left hand side the random field $\varphi$
depends on $(t, x) \in \mathbb{R} \times \mathbb{R}^d$ and is a stationary
solution to the parabolic SPDE
\begin{equation}
  \partial_t \varphi (t, x) + (m^2 - \Delta) \varphi (t, x) + V' (\varphi (t,
  x)) = \xi (t, x) \label{eq:langevin},
\end{equation}
where $\xi$ is a Gaussian white noise in $\mathbb{R}^{d + 1}$, $V : \mathbb{R}
\rightarrow \mathbb{R}$ a generic local potential bounded from below, $m^2$ a
positive parameter, and $\varphi (t)$ is the fixed time marginal of $\varphi$
which has a law independent of $t$ by stationarity and on the right hand side
we have the formal expression for a measure on functions on $\mathbb{R}^d$
with weight factor given by
\begin{equation}
  S (\phi) \assign \int_{\mathbb{R}^d} | \nabla \phi (x) |^2 + m^2 | \phi (x)
  |^2 + V (\phi (x)) \mathd x. \label{eq:action}
\end{equation}
Eq.~{\eqref{eq:dim-red-par}} can be made mathematically precise and rigorous
by tools from the theory of Markov
processes~{\cite{da_prato_stochastic_2008,krylov_stochastic_1999,prato_kolmogorov_2005}},
SDE/SPDEs~{\cite{khasminskii_stochastic_2011,Albeverio2017,MR2768734,Ugolini2011}}
and Dirichlet forms~{\cite{albeverio_stochastic_1991}}, for example when $d =
0$, or when the equation is regularized appropriately and, in certain cases,
for suitable renormalized versions of the
SPDE~{\cite{Albeverio2002,albeverio_invariant_2017,Gallavotti1980,borkar_stochastic_1988,da_prato_strong_2003,GH18,hairer_discretisations_2018,hairer_tightness_2018,jona-lasinio_stochastic_1985,MW17,albeverio_strong_2012,iwata_infinite_1987}}
when $d = 1, 2, 3$. Let us note for example that in the full space it is
easier to make sense of equation~{\eqref{eq:langevin}} than of the formal
Gibbs measure on the right hand side of~{\eqref{eq:dim-red-par}},
see~{\cite{GH18}}.

\

In a slightly different context, and inspired by previous perturbative
computations of Imry and Ma~{\cite{imry_random_field_1975}}, and
Young~{\cite{young_lowering_1977}}, Parisi and
Sourlas~{\cite{parisi_random_1979,parisi_supersymmetric_1982}} considered the
solutions of the elliptic SPDEs
\begin{equation}
  (m^2 - \Delta) \phi + V' (\phi) = \xi \label{eq:SPDE-d+2}
\end{equation}
in $\mathbb{R}^{d + 2}$ where $\xi$ is a Gaussian white noise on
$\mathbb{R}^{d + 2}$ and they discovered that its stationary solutions are
similarly related to the same $d$ dimensional Gibbs measure. If we take $x \in
\mathbb{R}^d$ then, they claimed that, for ``nice'' test functions $F$ (e.g.
correlation functions) we have
\begin{equation}
  \mathbb{E} [F (\phi (0, \cdot))] \propto \int F (\varphi) e^{- 4 \pi S
  (\varphi)} \mathcal{D} \varphi . \label{eq:dim-red}
\end{equation}
More precisely the law of the random field $(\phi (0, y))_{y \in
\mathbb{R}^d}$, obtained by looking at the trace of $\phi$ on the hyperplane
$\{ x = (x_1, \ldots, x_{d + 2}) \in \mathbb{R}^{d + 2} : x_1 = x_2 = 0 \}
\subset \mathbb{R}^{d + 2}$, should be equivalent to that of the Gibbs measure
formally appearing on the right hand side of~{\eqref{eq:dim-red}} and
corresponding to the action functional~{\eqref{eq:action}}. Therefore one can
interpret equation~{\eqref{eq:dim-red}} as an \tmtextit{elliptic stochastic
quantization} prescription in the same spirit of
equation~{\eqref{eq:dim-red-par}}.

\

When $V = 0$ one can directly check that the formula~{\eqref{eq:dim-red}} is
correct. Indeed in this case the unique stationary solution $\phi$ to the
elliptic SPDE~{\eqref{eq:SPDE-d+2}} is given by a Gaussian process with
covariance
\[ \mathbb{E} [\phi (x) \phi (x')] = \int_{\mathbb{R}^{d + 2}} \frac{e^{i k
   \cdot (x - x')}}{(m^2 + | k |^2)^2} \frac{\mathd k}{(2 \pi)^{d + 2}},
   \qquad x, x' \in \mathbb{R}^{d + 2} . \]
Therefore for all $y, y' \in \mathbb{R}^d$ we have
\[ \mathbb{E} [\phi (0, y) \phi (0, y')] = \int_{\mathbb{R}^d} e^{i k \cdot (y
   - y')} \int_{\mathbb{R}^2} \frac{\mathd q}{(| q |^2 + m^2 + | k |^2)^2}
   \frac{\mathd k}{(2 \pi)^{d + 2}} \]
\[ = \int_{\mathbb{R}^2} \frac{\mathd q}{(| q |^2 + 1)^2} \int_{\mathbb{R}^d}
   \frac{e^{i k \cdot (y - y')}}{m^2 + | k |^2} \frac{\mathd k}{(2 \pi)^{d +
   2}} = \frac{1}{4 \pi} \int_{\mathbb{R}^d} \frac{e^{i k \cdot (y - y')}}{m^2
   + | k |^2} \frac{\mathd k}{(2 \pi)^d} \]
where we performed a rescaling of the $q$ integral in order to decouple the
two integrations. The reader can easily check that the expression we obtained
describes the covariance of the Gaussian random field formally corresponding
to the right hand side of~{\eqref{eq:dim-red}} for $V = 0$.

\

While this last argument is almost trivial, a more general justification
outside the Gaussian setting is not so obvious. The
equivalence~{\eqref{eq:dim-red}} was derived
in~{\cite{parisi_random_1979,parisi_supersymmetric_1982}} at the theoretical
physics level of rigor going through a representation of the left hand side
via a supersymmetric quantum field theory (QFT) involving a pair of scalar
fermion fields. This is one of the instances of the \tmtextit{dimensional
reduction} phenomenon which is conjectured in certain random systems where the
randomness effectively decreases the dimension of the space where fluctuations
take place. \ A crucial assumption is that the equation~{\eqref{eq:SPDE-d+2}}
has a unique solution, which is already a non-trivial problem for general $V$.
Parisi and Sourlas~{\cite{parisi_supersymmetric_1982}} observed that
non-uniqueness can lead to a breaking of the supersymmetry, in which case the
relation~{\eqref{eq:dim-red}} could fail. So, part of the task of clarifying
the situation is to determine under which conditions \tmtextit{some} relations
in the spirit of~{\eqref{eq:dim-red}} could anyway be true.

\

The dimensional reduction~{\eqref{eq:dim-red}} of the elliptic
SPDEs~{\eqref{eq:SPDE-d+2}} seems less amenable to standard probabilistic
arguments than its parabolic counterpart~{\eqref{eq:dim-red-par}}. Let us
remark that from the point of view of theoretical physics it is
possible~{\cite{damgaard_stochastic_1988,parisi_supersymmetric_1982}} to
justify also dimensional reduction in the parabolic case~{\eqref{eq:langevin}}
using a supersymmetric argument much like in the elliptic setting.

\

The only attempt we are aware of to a mathematically rigorous understanding
of the relation~{\eqref{eq:dim-red}} is the work of Klein, Landau and
Perez~{\cite{klein_supersymmetry_1985,Klein1984,klein_supersymmetry_1983}}
(see also the related work on the density of states of electronic systems with
random potentials~{\cite{klein_density_1985}}) which however do not fully
prove equation~{\eqref{eq:dim-red}} but only the equivalence between the
intermediate supersymmetric theory in $d + 2$ dimensions and the Gibbs measure
in $d$ dimensions. The reason for this limitation is that the problem of
uniqueness of the elliptic SPDE seems to unnecessarily restrict the class of
potentials for which~{\eqref{eq:dim-red}} can be established and Klein et al.
decided to bypass a detailed analysis of the situation by starting directly
with the supersymmetric formulation. Their rigorous argument requires a
cut-off, both on large momenta in $d$ ``orthogonal'' dimensions and on the
space variable in $d + 2$ dimensions in order to obtain a well defined, finite
volume problem. This regularization breaks the supersymmetry which has to be
recovered by adding a suitable correction term, spoiling the final result (see
Theorem \ref{th:dim-red-qc} and Theorem \ref{th:dim-red-1} below). A subtle
gap in their published proof is pointed out, and closed, in
Section~\ref{sec:super}.

\

Let us remark that, in a different context, dimensional reduction has been
proven and exploited in the remarkable work of Brydges and Imbrie on branched
polymers~{\cite{brydges_branched_2003,brydges_dimensional_2003}} and more
recently by Helmuth~{\cite{helmuth_dimensional_2016}}.

\

In the present work we complete the program of \tmtextit{elliptic stochastic
quantization}, in $d = 0$ case, by proving relation {\eqref{eq:dim-red}}
linking the solution to the ellptic SPDE {\eqref{eq:SPDE-d+2}} with the Gibbs
measure with action {\eqref{eq:action}} and removing the finite volume cut-off
in some cases.

\

Fix $d = 0$ and consider the two dimensional elliptic multidimensional SPDE
\begin{equation}
  (m^2 - \Delta) \phi (x) + f (x) \partial V (\phi (x)) = \xi (x)
  \label{equation2d1} \qquad x \in \mathbb{R}^2
\end{equation}
where $\phi = (\phi^1, \ldots, \phi^n)$ takes values in $\mathbb{R}^n$,
$(\xi^1, \ldots, \xi^n)$ are $n$ independent Gaussian white noises, $V :
\mathbb{R}^n \rightarrow \mathbb{R}$ a smooth potential function, $f (x)
\assign \tilde{f} (| x |^2)$ with $\tilde{f} : \mathbb{R}_+ \rightarrow
\mathbb{R}_+$ a decreasing cut-off function, such that the derivative
$\tilde{f}'$ of the function $r \longmapsto \tilde{f} (r)$ is defined, tending
to $0$ at infinity, and $\partial V = (\partial_i V)_{i = 1, \ldots, n}$
denotes the gradient of $V$. We will denote $f' (x) \assign \tilde{f}' (| x
|^2)$.

\

Eq.~{\eqref{equation2d1}} is the elliptic counterpart of the equilibrium
Langevin reversible dynamics for finite dimensional Gibbs measures. Let us
note that the elliptic dynamics is already described by an SPDE in two
dimensions while in the parabolic setting one would consider a much simpler
Markovian SDE~{\cite{iwata_infinite_1987,albeverio_strong_2012}} (no
renormalization being necessary). The question of uniqueness of solutions is
however quite similar in difficulty, indeed it is non-trivial to establish
uniqueness of stationary solutions to the SDE and much work in the theory of
long time behavior of Markov processes is devoted precisely to this. In the
elliptic context of~{\eqref{equation2d1}} there is no (easy) Markov property
helping and the question of uniqueness of weak stationary solutions seems more
open, even in the presence of the cut-off~$f$.

\

What makes this $d = 0$ problem very interesting, is above all the fact that
while the statements we would like to prove are quite easy to describe (see
below), to our surprise their rigorous justification is already quite involved
and not yet quite complete in full generality.

\

Define the probability measure $\kappa$ on $\mathbb{R}^n$ by
\begin{equation}
  \frac{\mathd \kappa}{\mathd y} \assign Z_{\kappa}^{- 1} \exp \left[ - 4 \pi
  \left( \frac{m^2}{2} | y |^2 + V (y) \right) \right], \label{eq:gibbs-kappa}
\end{equation}
where $y \in \mathbb{R}^n$, $Z_{\kappa} \assign \int_{\mathbb{R}^n} \exp
\left[ - 4 \pi \left( \frac{m^2}{2} | y |^2 + V (y) \right) \right] \mathd y$
($Z_{\kappa}$ is well defined since $V$ is bounded from below).

\

The main result of this paper is the following theorem which states that on
very general conditions on $V$ there is always a weak solution which satisfies
(an approximate) elliptic stochastic quantization relation (of the
form~{\eqref{eq:dim-red}}). By weak solution to the SPDE~{\eqref{equation2d1}}
we mean a probability measure $\nu$ on the space of fields $\phi$ under which
$(m^2 - \Delta) \phi + \partial V (\phi)$ is distributed like Gaussian white
noise on $\mathbb{R}^2$. A strong solution $\phi$ to
equation~{\eqref{equation2d1}} is a measurable map $\xi \mapsto \phi = \phi
(\xi)$ satisfying the equation for almost all realizations of $\xi$. In order
to state precisely our results we need to introduce \ the following
assumptions on~$V$ and on the finite volume cut-off~$f$:

{\descriptionparagraphs{\item[Hypothesis~C. (convexity)] The potential $V :
\mathbb{R}^n \rightarrow \mathbb{R}$ is a positive smooth function such that
\[ y \in \mathbb{R}^n \mapsto V (y) + m^2 | y |^2 \]
is strictly convex and $V$ with its first and second partial derivatives grow
at most exponentially at infinity.

\item[Hypothesis~QC. (quasi convexity)] The potential $V : \mathbb{R}^n
\rightarrow \mathbb{R}$ is a positive smooth function, such that it and its
first and second partial derivatives grow at most exponentially at infinity
and moreover it is such that there exists a function $H : \mathbb{R}^n
\rightarrow \mathbb{R}$ with exponential growth at infinity such that we have
\[ - \langle \hat{n}, \partial V (y + r \hat{n}) \rangle \leqslant H (y),
   \qquad \text{$\hat{n} \in \mathbb{S}^n, y \in \mathbb{R}^n$ and $r \in
   \mathbb{R}_+$}, \]
with $\mathbb{S}^n$ is the $n - 1$ dimensional unit sphere.

\item[Hypothesis~CO. (cut-off) ] The function $f$ is real valued, has at least
$C^2$ smoothness and in addition satisfies $f' \leqslant 0$, it decays
exponentially at infinity and fulfills $\Delta (f) \leqslant b^2 f$ for $b^2 <
4 m^2$ (some examples of such functions are given in~{\cite{Klein1984}} and
the motivations for this hypothesis are explained in
Remark~\ref{remark_hypotheses} below).}}

\begin{theorem}
  \label{th:dim-red-qc}Under the Hypotheses~QC and~CO there exists (at least)
  one weak solution $\tilde{\nu}$ to equation~{\eqref{equation2d1}} such that
  for all measurable bounded functions $h : \mathbb{R}^n \rightarrow
  \mathbb{R}$ we have
  \begin{equation}
    \int_{\widetilde{\mathcal{W}}} h (\phi (0)) \Upsilon_f (\phi)  \tilde{\nu}
    (\mathd \phi) = Z_f \int_{\mathbb{R}^n} h (y) \mathd \kappa (y)
    \label{eq:prec-dim-red}
  \end{equation}
  where $\Upsilon_f (\phi) \assign e^{4 \int_{\mathbb{R}^2} f' (x) V (\phi
  (x)) \mathd x}$ and $Z_f \assign \int_{\mathcal{W}} \Upsilon_f (\phi)
  \tilde{\nu} (\mathd \phi)$. $\widetilde{\mathcal{W}}$ is a suitable Banach
  space of functions from $\mathbb{R}^2$ to $\mathbb{R}^n$ where $\tilde{\nu}$
  is defined (see Section~\ref{sec:solutions}
  equations~{\eqref{eq:Banach1}},~{\eqref{eq:Banch2}}~and~{\eqref{eq:tildenu}}
  for a precise definition of $\widetilde{\mathcal{W}}$ and $\tilde{\nu}$).
\end{theorem}

\begin{remark}
  The following families of functions satisfy Hypothesis~QC:
  \begin{itemizedot}
    \item Smooth convex functions (since they satisfy the stronger
    Hypothesis~C),
    
    \item Smooth bounded functions,
    
    \item Smooth functions having the second derivative semidefinite positive
    outside a compact set,
    
    \item Any positive linear combinations of the previous functions.
  \end{itemizedot}
\end{remark}

The drawback of this result is the lack of constructive determination of the
weak solution $\nu$ for which the dimensional reduction described by
equation~{\eqref{eq:prec-dim-red}} is realized. This is linked with the fact
that Hypothesis~QC does not guarantee uniqueness of strong solutions to
eq.~{\eqref{equation2d1}}. The fact that non-uniqueness is related to a
possible breaking of the supersymmetry associated with~{\eqref{equation2d1}}
was already noted by Parisi and Sourlas~{\cite{parisi_supersymmetric_1982}}.
If we are willing to assume that the potential $V$ is convex we can be more
precise, as the following theorem shows.

\begin{theorem}
  \label{th:dim-red-1}Under Hypotheses~C and~CO there exists an unique strong
  solution $\phi = \phi (\xi)$ of equation~{\eqref{equation2d1}} and for all
  measurable bounded functions $h : \mathbb{R}^n \rightarrow \mathbb{R}$ we
  have
  \begin{equation}
    \mathbb{E} [h (\phi (0)) \Upsilon_f (\phi)] = Z_f \int_{\mathbb{R}^n} h
    (y) \mathd \kappa (y) \label{eq:prec-dim-red-c}
  \end{equation}
  where $\Upsilon_f$ is defined as in Theorem \ref{th:dim-red-qc}, $Z_f
  \assign \mathbb{E} [\Upsilon_f (\phi)]$, and $\mathbb{E}$ denotes
  expectation with respect to the law of $\xi$.
\end{theorem}

Both theorems require the presence of a suitable cut-off $f \nequiv 1$ which
is responsible for the weighting factor $\Upsilon_f (\phi)$ on the left hand
side of the dimensional reduction statements~{\eqref{eq:prec-dim-red}}
and~{\eqref{eq:prec-dim-red-c}}. If we would be allowed to take $f \equiv 1$
then we would have proven the $d = 0$ version of
equation~{\eqref{eq:dim-red}}. However, presently we are not able to do this
for all QC potentials but only for those satisfying Hypothesis~C (see
Section~\ref{sec:super} for the proof). \

\begin{theorem}
  \label{theorem_cutoff1}Suppose that $V$ satisfies Hypothesis~C and let
  $\phi$ be the unique strong solution in $C^0_{\exp \beta} (\mathbb{R}^2 ;
  \mathbb{R}^n)$ (see Section~\ref{sec:removal} for the definition of this
  space) of the equation
  \begin{equation}
    (m^2 - \Delta) \phi + \partial V (\phi) = \xi . \label{equationcutoff1}
  \end{equation}
  Then for any $x \in \mathbb{R}^2$ and any measurable and bounded function
  $h$ defined on $\mathbb{R}^n$ we have
  \begin{equation}
    \mathbb{E} [h (\phi (x))] = \int_{\mathbb{R}^n} h (y) \mathd \kappa (y) .
    \label{equationcutoff3}
  \end{equation}
\end{theorem}

This result is the first rigorous result on elliptic stochastic quantization
without any cut-off. In fact in this case the results of Klein, Landau and
Perez~{\cite{Klein1984}} do not hold, since they use only an integral
representation of the solution to equation~{\eqref{equation2d1}} which has
meaning only when $f \nequiv 1$.

\begin{remark}
  It is easy to generalize Theorems~\ref{th:dim-red-qc},~\ref{th:dim-red-1}
  and~\ref{theorem_cutoff1} to equations of the form
  \begin{equation}
    (m^2 - \Delta) \phi^i (x) + \sum_{r = 1}^n \gamma_r^i \gamma^j_r f (x)
    \partial_{\phi^j} V (\phi (x)) = \gamma^i_r \xi^r (x),
    \label{equation2d1-gamma}
  \end{equation}
  where $f$ is as before, $\Gamma = (\gamma^i_j)_{i, j = 1, \ldots, n}$ is an
  $n \times n$ invertible matrix and the Hypothesis~C and QC are generalized
  accordingly.
\end{remark}

\tmtextbf{Plan.} The paper is organized as follows. In
Section~\ref{sec:solutions} we introduce the notions of strong and weak
solutions to equation~{\eqref{equation2d1}}, and we prove, in
Theorem~\ref{theorem_existence2}, the existence of strong
solutions (and thus also of weak solutions) under Hypothesis~QC. We also provide, in
Theorem~\ref{theorem_weaksolution}, a representation of weak solutions via the
theory of transformation of measures on abstract Wiener spaces developed by
{\"U}st{\"u}nel and Zakai in~{\cite{Ustunel2000}} (whose setting and main
facts needed here are summarized in Appendix~\ref{appendix_wienerspace}).

Section~\ref{sec:dim-red} is devoted to the proof Theorem~\ref{th:dim-red-qc}
and Theorem~\ref{th:dim-red-1} about elliptic stochastic quantization, under
the Hypothesis~QC and~CO and using Theorem~\ref{theorem_main1} and PDE
techniques.

In Section~\ref{sec:super} Theorem~\ref{theorem_main1} is proven, i.e.
dimensional reduction using Hypothesis~$V_{\lambda}$ (see Section
\ref{sec:dim-red}). The proof of Theorem~\ref{theorem_main1} is similar to the
rigorous version of Parisi and Sourlas argument proposed in~{\cite{Klein1984}},
starting from different hypotheses. The proof of  Theorem~\ref{theorem_main1} in Section~\ref{sec:super} is
based on Theorem~\ref{th:pol-eq} stating a relation between the expectation
involving some bosonic and fermionic free fields.

In Section~\ref{section:supersymmetry} we prove Theorem~\ref{th:pol-eq}
exploiting the properties of supersymmetric Gaussian fields. In
Section~\ref{section:supersymmetry} we also propose a brief introduction to
supersymmetry and supersymmetric Gaussian fields.

Section~\ref{sec:removal} discusses the proof of
Theorem~\ref{theorem_cutoff1} on the cut-off removal under Hypothesis C.

Appendix \ref{appendix_wienerspace} is a brief introduction to the theory of
transformations on abstract Wiener spaces used in this paper, and Appendix
\ref{appendix_grasmannian} consists in a discussion of some properties of
fermionic fields.

\

{\noindent}\tmtextbf{Acknowledgments.} The authors would like to thank the
Isaac Newton Institute for Mathematical Sciences for support and hospitality
during the program Scaling limits, rough paths, quantum field theory when work
on this paper was undertaken. This work was supported by EPSRC Grant Number
EP/R014604/1 and by the German Research Foundation (DFG) via CRC 1060.

\section{The elliptic SPDE}\label{sec:solutions}

In order to study equation~{\eqref{equation2d1}} we have to recall some
definitions, notations and conventions. Fix an abstract Wiener space
$(\mathcal{W}, \mathcal{H}, \mu)$ where the noise $\xi$ is defined (for the
concept of abstract Wiener space we refer e.g.
to~{\cite{Gross1967,Nualart2006,Ustunel2000}}). The Cameron-Martin space
$\mathcal{H}$ is the space
\[ \mathcal{H} \assign L^2 (\mathbb{R}^2 ; \mathbb{R}^n), \]
with its natural scalar product and natural norm given by $\langle h, g
\rangle = \sum_{i = 1}^n \int_{\mathbb{R}^2} h^i (x) g^i (x) \mathd x$. Let
$\mathcal{W}$ (in which $\mathcal{H}$ is densely embedded) be the space
\begin{equation}
  \mathcal{W}=\mathcal{W}_{p, \eta} \assign W^{p, - 1 - 2 \epsilon}_{\eta}
  (\mathbb{R}^2 ; \mathbb{R}^n) \cap (1 - \Delta) (C^0_{\eta} (\mathbb{R}^2 ;
  \mathbb{R}^n)), \label{eq:Banach1}
\end{equation}
where $p \geqslant 1$, $\eta > 0$ and $W_{\eta}^{p, - 1 - 2 \epsilon}
(\mathbb{R}^2 ; \mathbb{R}^n)$ is a fractional Sobolev space with norm
\[ \|g\|_{W_{\eta}^{p, - 1 - 2 \epsilon}} \assign \left( \int_{\mathbb{R}^2}
   (1 + | x |)^{- \eta} \left| (1 - \Delta)^{- \frac{1}{2} - \epsilon} (g)
   \right|^p \mathd x \right)^{\frac{1}{p}}, \]
for some $\epsilon > 0$ small enough and $(1 - \Delta) (C^0_{\eta}
(\mathbb{R}^2 ; \mathbb{R}^n))$ is the space of the second order
distributional derivatives of continuous functions on $\mathbb{R}^n$ growing
at infinity at most as $| x |^{\eta}$ with norm
\[ \|g\|_{(- \Delta + 1) (C^0_{\eta})} \assign \| (1 + | x |)^{- \eta} ((1 -
   \Delta)^{- 1} g) (x) \|_{L^{\infty}_x} . \]
Thus $\mathcal{W}_{p, \eta}$ is a Banach space with norm given by the sum of
the norms of $W_{\eta}^{p, - 1 - 2 \epsilon} (\mathbb{R}^2 ; \mathbb{R}^n)$
and of $(1 - \Delta)^{- 1} (C^0_{\eta} (\mathbb{R}^2 ; \mathbb{R}^n))$. In the
following we usually do not specify the indices $\eta$ and $p$ in the
definition of $\mathcal{W}_{p, \eta}$ and we write only $\mathcal{W}$. We also
introduce the notation
\begin{equation}
  \tilde{\mathcal{W}} = (1 - \Delta)^{- 1} (\mathcal{W}) \label{eq:Banch2}
\end{equation}
The Gaussian measure $\mu$ on $\mathcal{W}$ is the standard Gaussian measure
with Fourier transform given by $e^{- \frac{1}{2} \| \cdot
\|^2_{\mathcal{H}}}$. The white noise $\xi$ is then naturally realized on
$(\mathcal{H}, \mathcal{W}, \mu)$, in the sense that $\xi$ is the random
variable $\xi : \mathcal{W} \rightarrow \mathcal{S}' (\mathbb{R}^2 ;
\mathbb{R}^n)$ (where $\mathcal{S}' (\mathbb{R}^2 ; \mathbb{R}^n)$ is the
space of $\mathbb{R}^n$--valued Schwartz distributions on $\mathbb{R}^2$)
defined as $\xi (w) = \tmop{id}_{\mathcal{W}} (w) = w$. In this way the law of
$\xi$ is simply $\mu$ (or, better, it is equal to the pushforward of $\mu$ on
$\mathcal{S}' \assign \mathcal{S}' (\mathbb{R}^2, \mathbb{R}^n)$ with respect
the natural inclusion map of $\mathcal{W}$ in $\mathcal{S}'$).

Sometimes it is also useful to consider the space
$\mathcal{C}^{\alpha}_{\tau}$ of $\alpha$-H{\tmname{}}{\"o}lder continuous
functions such that they and their derivatives (or H{\"o}lder norms) grow at
infinity at most like $| x |^{\tau}$ for a real number $\tau$ (this notation
is used also if $\tau$ is negative in that case the functions decrease at
least like $\frac{1}{| x |^{- \tau}}$). It is important to note that
$\mathcal{C}^{\alpha}_{\eta}$ can be identified with the Besov space
$B^{\alpha}_{\infty, \infty, \eta} (\mathbb{R}^2)$ (where $B^{\alpha}_{\infty,
\infty, \eta} (\mathbb{R}^2)$ is the weighted Besov space $B^{\alpha}_{\infty,
\infty} (\mathbb{R}^2, (1 + | x |)^{\eta})$ \ of~{\cite{Bahouri2011}},~Chapter
2 Section 2.7). It is also important to realize that $(1 - \Delta)^{- 1}
(\mathcal{W}) \subset \mathcal{C}^{\alpha}_{\eta}$ if we choose $p$ big enough
and $\alpha > 0$ small enough.

\

We introduce now two notions of solutions for equation~{\eqref{equation2d1}}.
For later convenience it is better to discuss the equation in term of the
variable $\eta \assign (m^2 - \Delta) \phi$ for which it reads
\begin{equation}
  \eta + f \partial V (\mathcal{I} \eta) = \eta + U (\eta) = \xi,
  \label{eq:transf}
\end{equation}
where
\[ \mathcal{I} \assign (m^2 - \Delta)^{- 1} \]
and where we introduced the map $U : \mathcal{W} \rightarrow \mathcal{H}$
given by
\begin{equation}
  U (w) \assign f \partial V (\mathcal{I}w) \label{equationU}, \qquad w \in
  \mathcal{W}.
\end{equation}
Under the condition of (at most) exponential growth at infinity of $V$,
required by Hypothesis~QC and Hypothesis~C, it is possible to prove, that for
$\eta < 1$ in the definition of $\mathcal{W}$, for each $w \in \mathcal{W}$ we
have $U (w) \in \mathcal{W}$. Indeed we have
\[ \| U (w) \|_{(L^2 (\mathbb{R}^2))^n} \leq \left\| \sqrt{f (x)}
   \right\|_{L^2 (\mathbb{R}^2)} \cdot \left\| \sqrt{f (x)} | \partial V
   (\mathcal{I} w (x)) | \right\|_{\infty} \]
and $\left\| \sqrt{f (x)} | \partial V (\mathcal{I} w (x)) |
\right\|_{\infty}$ is finite since $f$ decreases exponentially at infinity and
$V$ grow at most exponentially at infinity.

Furthermore we introduce the map $T : \mathcal{W} \rightarrow \mathcal{W}$ as
\[ T (w) \assign w + U (w) . \]
It is clear that a map $S : \mathcal{W} \rightarrow \mathcal{W}$ satisfies
equation~{\eqref{eq:transf}}, i.e. $T (S (w)) = \xi (w) = w$, for
($\mu$-)almost all $w \in \mathcal{W}$, if and only if $\mathcal{I}S (w)$
satisfies equation {\eqref{equation2d1}}. The law $\nu$ on $\mathcal{W}$
associated to a solution of equation {\eqref{eq:transf}} must satisfy the
relation $T_{\asterisk}^{} (\nu) = \mu$. For these reasons we introduce the
following definition.

\begin{definition}
  \label{definition_solution}A measurable map $S : \mathcal{W} \rightarrow
  \mathcal{W}$ is a {\tmem{strong solution}} to equation~{\eqref{eq:transf}}
  if $T \circ S = \tmop{Id}_{\mathcal{W}}$ $\mu$-almost surely. A probability
  measure $\nu \in \mathcal{P} (\mathcal{W})$ (where $\mathcal{P}
  (\mathcal{W})$ is the space of probability measures on $\mathcal{W}$) on the
  space $\mathcal{W}$ is a {\tmem{weak solution}} to
  equations{\eqref{eq:transf}} if $T_{\asterisk} (\nu) = \mu$, where
  $T_{\ast}$ is the pushforward related with the map $T$.
\end{definition}

If $\nu$ is a probability measure on the space $\mathcal{W}$, we write
$\tilde{\nu}$ the unique probability measure on $\widetilde{\mathcal{W}}$ such
that
\begin{equation}
  (- \Delta + m^2)^{- 1}_{\asterisk} (\nu) = \tilde{\nu} . \label{eq:tildenu}
\end{equation}

\subsection{Strong solutions}

In order to study the existence of strong solutions to
equation~{\eqref{equation2d1}} we introduce an equivalent version of the same
equation that is simpler to study. Indeed if we write
\[ \bar{\phi} = \phi -\mathcal{I} \xi, \]
and we suppose that $\phi$ satisfies equation {\eqref{equation2d1}}, then the
function $\bar{\phi}$ satisfies the following (random) PDE
\begin{equation}
  (m^2 - \Delta) \bar{\phi} + f \partial V (\bar{\phi} -\mathcal{I} \xi) = 0.
  \label{equation2d2}
\end{equation}
Equation {\eqref{equation2d2}} can be studied pathwise for any realization of
the random field $\mathcal{I} \xi$. Hereafter the symbol $\lesssim$ stands for
inequality with some positive constant standing on the right hand side.

\begin{lemma}
  \label{lemma_bound}Suppose that $V$ satisfies Hypothesis~QC, and let
  $\bar{\phi}$ be a classical $C^2$ solution to the
  equation~{\eqref{equation2d2}}, such that $\lim_{x \rightarrow \infty}
  \bar{\phi} (x) = 0$, then for any $0 < \tau < 2$ and $\eta > 0$ we have
  \begin{equation}
    \begin{array}{ll}
      \| \bar{\phi} \|_{\infty} & \lesssim 1 +\|f e^{\alpha_1 | \mathcal{I}
      \xi |} \|_{\infty} \label{equationexistence2d1}
    \end{array}
  \end{equation}
  \begin{equation}
    \begin{array}{lll}
      \| \bar{\phi} \|_{\mathcal{C}_{- \eta}^{2 - \tau}} & \lesssim & 1 +
      e^{\alpha_1 \| \bar{\phi} \|_{\infty}} \| (| x | + 1)^{\eta} f
      e^{\alpha_1 | \mathcal{I} \xi |} \|_{\infty},
    \end{array} \label{equationexistence2d2}
  \end{equation}
  for some positive constant $\alpha_1$ and where it and the constants
  involved in the symbol $\lesssim$ depend only on the function $H$ in
  Hypothesis~QC.
\end{lemma}

\begin{proof}
  Putting $r_{\bar{\phi}} (x) = \sqrt{\sum_i (\bar{\phi}^i (x))^2} = |
  \bar{\phi} (x) |$, $x \in \mathbb{R}^2$, since the $C^2$ function
  $\bar{\phi}$ converges to zero at infinity, the function $r_{\bar{\phi}}$
  must have a global maximum at some point $\bar{x} \in \mathbb{R}^2$. This
  means that $- \Delta (r^2_{\bar{\phi}}) (\bar{x}) \geqslant 0$. On the other
  hand since $\bar{\phi}$ solves equation {\eqref{equation2d2}} we have
  \begin{eqnarray}
    m^2 r^2_{\bar{\phi}} (\bar{x}) & \leqslant & - \frac{1}{2} \Delta
    (r^2_{\bar{\phi}}) (\bar{x}) + m^2 r^2_{\bar{\phi}} (\bar{x}) \nonumber\\
    & \leqslant & (- \bar{\phi} \cdot \Delta \bar{\phi} - | \nabla \bar{\phi}
    |^2 + m^2 | \bar{\phi} |^2) \nonumber\\
    & \leqslant & - f (\bar{x}) r_{\bar{\phi}} (\bar{x})
    (\hat{n}_{\bar{\phi}} (\bar{x}) \cdot \partial V (\mathcal{I} \xi
    (\bar{x}) + \hat{n}_{\bar{\phi}} (\bar{x}) r_{\bar{\phi}} (\bar{x})))
    \nonumber
  \end{eqnarray}
  where $\hat{n}_{\bar{\phi}} = \frac{\bar{\phi}}{| \bar{\phi} |} \in
  \mathbb{S}^n$ when $r_{\bar{\phi}} \not{=} 0$, and $0$ elsewhere. Using
  Hypothesis~QC we obtain
  \[ \| r_{\bar{\phi}} \|_{\infty} \leqslant \frac{f (\bar{x}) H (\mathcal{I}
     \xi (\bar{x}))}{m^2} \lesssim 1 +\|f e^{\alpha_1 | \mathcal{I} \xi |}
     \|_{\infty}, \]
  since $H$ grows at most exponentially at infinity. This result implies
  inequality {\eqref{equationexistence2d1}}.
  
  The bound~{\eqref{equationexistence2d2}} can be obtained directly using the
  fact $\| \phi \|_{\mathcal{C}^{2 - \tau}} \lesssim \| (- \Delta + m^2)
  (\phi) \|_{\infty}$, where we use the properties of the Besov spaces
  $\mathcal{C}^{\alpha} (\mathbb{R}^2) = B^{\alpha}_{\infty, \infty}
  (\mathbb{R}^2)$ with respect to derivatives (see~{\cite{Trie1983}},
  Chapter~2 Section~2.3.8).
\end{proof}

\begin{remark}
  \label{remark_bound}It is simple to prove that the
  inequalities~{\eqref{equationexistence2d1}}
  and~{\eqref{equationexistence2d2}} can be chosen to be uniform with respect
  to some rescaling of the potential of the form $\lambda V$, or satisfying
  Hypothesis~$V_{\lambda}$ below, where $\lambda \in [0, 1]$.
\end{remark}

In the following we denote by $\mathcal{F} : \mathcal{W} \rightarrow
\mathcal{P} (\mathcal{C}^{2 - \tau} (\mathbb{R}^2 ; \mathbb{R}^n))$ the set
valued function which associates to a given $w \in \mathcal{W}$ the (possible
empty) set of solutions to equation~{\eqref{equation2d2}} in $\mathcal{C}^{2 -
\tau} (\mathbb{R}^2 ; \mathbb{R}^n)$, where $\tau > 0$, when $\mathcal{I} \xi$
is evaluated in $w$.

\begin{theorem}
  \label{theorem_existence1}Let $V$ be a smooth positive function satisfying
  Hypothesis~QC, then for any $w \in \mathcal{W}$ the set $\mathcal{F} (w)$ is
  non-empty and closed. Furthermore $\mathcal{F} (w) \subset C^2 (\mathbb{R}^2
  ; \mathbb{R}^n)$ and if $B \subset \mathcal{W}$ is a bounded set then
  $\mathcal{F} (B) = \bigcup_{w \in B} \mathcal{F} (w)$ is compact in
  $\mathcal{C}_{- \eta}^{2 - \tau} (\mathbb{R}^2 ; \mathbb{R}^n)$ for any
  $\tau > 0$ and $\eta \geqslant 0$.
\end{theorem}

\begin{proof}
  We introduce the map $\mathcal{C}^{2 - \tau}_{- \eta} (\mathbb{R}^2 ;
  \mathbb{R}^n) \times \mathcal{W} \ni (\bar{\phi}, w) \mapsto \mathcal{K}
  (\bar{\phi}, w) \in \mathcal{C}^{2 - \tau'} (\mathbb{R}^2 ; \mathbb{R}^n)$,
  where \ $\tau' < \tau$, given by
  \[ \mathcal{K}^i (\bar{\phi}, w) \assign -\mathcal{I} (f \partial V
     (\bar{\phi} +\mathcal{I} \xi (w))) . \]
  The map $\mathcal{K}$ is continuous with respect to its first argument,
  indeed if $\bar{\phi}, \bar{\phi}_1 \in \mathcal{C}^{2 - \tau} (\mathbb{R}^2
  ; \mathbb{R}^n)$,
  \begin{eqnarray*}
    & \| \mathcal{K}^i (\bar{\phi}, w) - \mathcal{K}^i (\bar{\phi}_1, w)
    \|_{\mathcal{C}^{2 - \tau'}_{- \eta}} & \\
    & \lesssim \| (| x | + 1)^{\eta} f (\partial V (\bar{\phi}, \mathcal{I}
    \xi (w)) - \partial V (\bar{\phi}_1, \mathcal{I} \xi (w))) \|_{\infty} &
    \\
    & \lesssim \left\| \int_0^1 (| x | + 1)^{\eta} f \partial^2 V (\bar{\phi}
    - t (\bar{\phi} - \bar{\phi}_1) +\mathcal{I} \xi (w)) \cdot (\bar{\phi} -
    \bar{\phi}_1) \mathd t \right\|_{\infty} & \\
    & \lesssim \| \bar{\phi} - \bar{\phi}_1 \|_{\infty} \| (| x | + 1)^{\eta}
    \sqrt{f} \|_{\infty} \left( \| \partial^2 V_B \|_{\infty} + e^{\alpha \|
    \bar{\phi} - \bar{\phi}_1 \|_{\infty}} \| \sqrt{f} e^{\alpha | \mathcal{I}
    \xi |} \|_{\infty} \right), & 
  \end{eqnarray*}
  where the positive constant $\alpha$ depends on the exponential growth of
  $\partial^2 V$ at infinity. By a similar reasoning we can prove that
  $\mathcal{K}$ sends bounded sets of $\mathcal{C}^{2 - \tau}_{- \eta}$ into
  bounded sets of $\mathcal{C}^{2 - \tau'}_{- \eta'}$, where $\tau' < \tau$
  and $\eta' > \eta$. Since the immersion $\mathcal{C}^{2 - \tau'}_{- \eta'}
  \longhookrightarrow \mathcal{C}^{2 - \tau}_{- \eta}$ is compact we have that
  $\mathcal{K}$ is a compact map.
  
  Since $\mathcal{I} \xi \in \mathcal{C}^{1 -}_{\alpha}$ and $\bar{\phi} \in
  \mathcal{C}^{2 - \tau}_{- \eta} (\mathbb{R}^2 ; \mathbb{R}^n)$ we have $(-
  \Delta + m^2) \mathcal{K}^i (\bar{\phi}, w) \in \mathcal{C}^{1 -}
  (\mathbb{R}^2 ; \mathbb{R}^n)$. This implies, using the regularity results
  for Poisson equations (see Theorem~4.3 in~{\cite{Gilbarg2001}}) and a
  bootstrap argument, that if $\bar{\phi} = \mathcal{K} (\bar{\phi}, w)$ then
  $\bar{\phi} \in C^2 (\mathbb{R}^2)$. From this fact it follows that, using
  inequalities~{\eqref{equationexistence2d1}}
  and~{\eqref{equationexistence2d2}} of Lemma~\ref{lemma_bound} and
  Remark~\ref{remark_bound}, the solutions to the equation $\bar{\phi} =
  \lambda \mathcal{K} (\bar{\phi}, w) \mathcal{}$ are uniformly bounded for
  $\lambda \in [0, 1]$. Thanks to these properties of the map $\mathcal{K}$ we
  can use Schaefer's fixed-point theorem (see~{\cite{Evans1998}} Theorem~4
  Section 9.2 Chapter 9) to prove the existence of at least one solution to
  equation~{\eqref{equation2d2}}. Finally using again Lemma~\ref{lemma_bound}
  we have that $\mathcal{F} (B)$ is compact for any bounded set $B \subset
  \mathcal{W}$.
\end{proof}

\begin{theorem}
  \label{theorem_existence2}Under Hypothesis~QC on $V$ there exists a strong
  solution to equation~{\eqref{equation2d1}} (or equivalently to
  equation~{\eqref{eq:transf}}).
\end{theorem}

\begin{proof}
  For proving the existence of a strong solution to the equation
  {\eqref{eq:transf}} (in the sense of Definition~\ref{definition_solution})
  it is sufficient to prove that we can choose the solutions to
  equation~{\eqref{equation2d2}}, whose existence for any $w \in \mathcal{W}$
  is guaranteed by Theorem~\ref{theorem_existence1}, in a measurable way with
  respect $w \in \mathcal{W}$. More precisely we have to prove that there
  exists a measurable selection for the function set map $\mathcal{F}$, i.e.
  there exists a map $\bar{S} : \mathcal{W} \rightarrow \mathcal{C}^{2 -
  \tau}_{- \eta}$ such that $\bar{S} (w) \in \mathcal{F} (w)$.
  
  Fix a sequence of balls $B_1, \ldots, B_n, \ldots \subset \mathcal{W}$ of
  increasing radius and such that $\lim_{n \rightarrow + \infty} B_n
  =\mathcal{W}$, then, by Theorem~\ref{theorem_existence1}, the map
  $\nobracket \mathcal{F} |_{B_n \backslash B_{n - 1}}$ takes values in a
  compact set. As proven in Theorem~\ref{theorem_existence1} the map
  $\mathcal{K}$ is continuous in $\bar{\phi}$ and measurable in $w$ and
  therefore a Carath{\'e}odory map. As a consequence, by Filippov's implicit
  function theorem~(see Theorem~18.17 in~{\cite{Aliprantis2006}}), there
  exists a (Borel) measurable function $\bar{S}_n$ defined on $B_n \backslash
  B_{n - 1}$ such that $\bar{S}_n (w) \in \mathcal{F} (w)$. The map $\bar{S}$
  defined on $B_n \backslash B_{n - 1}$ by $\nobracket \bar{S} |_{B_n
  \backslash B_{n - 1}} = \bar{S}_n$ is the measurable selection that we need
  (since $B_n \backslash B_{n - 1}$ is measurable).
  
  A strong solution $S$ to equation~{\eqref{eq:transf}} is then given by $S
  (w) \assign w + (m^2 - \Delta) \bar{S} (w)$, $w \in \mathcal{W}$. \ 
\end{proof}

\begin{corollary}
  \label{corollary_uniqueness1}Under the Hypothesis~C there exists only one
  strong solution to equation~{\eqref{eq:transf}}.
\end{corollary}

\begin{proof}
  Suppose that $S_1, S_2$ are two strong solutions to
  equation~{\eqref{eq:transf}} then, letting $\phi_j (x, w) =\mathcal{I} (S_j
  (w (x)))$, $j = 1, 2$, writing $\delta \phi (x, w) = \phi_1 (x, w) - \phi_2
  (x, w)$ and $\delta r (x, w) = \sqrt{\sum_{i = 1}^n (\delta \phi^i (x,
  w))^2}$, we obtain
  \[ (m^2 - \Delta) (\delta r^2) + 2 \sum_i (| \nabla \delta \phi^i |^2) + f
     \delta r [\hat{n}_{\delta \phi} \cdot (\partial V (\phi_1) - \partial V
     (\phi_2))] = 0. \]
  By Lagrange's theorem there exists a function $g (x)$, $x \in \mathbb{R}^2$,
  taking values in the segment $[\phi_1 (x), \phi_2 (x)] \subset \mathbb{R}^n$
  such that $\hat{n}_{\delta \phi} \cdot (\partial V (\phi_1) - \partial V
  (\phi_2)) = \delta r \partial^2 V (g) (\hat{n}_{\delta \phi},
  \hat{n}_{\delta \phi})$. From this fact we obtain
  \[ (m^2 - \Delta) (\delta r^2) + f (\partial^2 V (g) (\hat{n}_{\delta \phi},
     \hat{n}_{\delta \phi})) \delta r^2 \leqslant 0. \]
  Since $m^2 + \partial^2 V (g) (\hat{n}_{\delta \phi}, \hat{n}_{\delta \phi})
  \geqslant \varepsilon > 0$ , $y \mapsto V (y) + m^2 | y |^2$ being strictly
  convex by our Hypothesis~C, and $\delta r^2 (x)$ is positive and goes to
  zero as $x \rightarrow + \infty$, we have $\phi_1 = \phi_2$ and therefore
  $S_1 (w) = S_2 (w)$.
\end{proof}

\subsection{Weak solutions}

First of all we prove that the map $U$, given by {\eqref{equationU}}, is a $H
- C^1$ function (in the sense of~{\cite{Ustunel2000}}, see
Appendix~\ref{appendix_wienerspace}) for the abstract Wiener space
$(\mathcal{W}, \mathcal{H}, \mu)$.

\begin{proposition}
  \label{proposition_C1H}If $V$ and its derivatives grow at most exponentially
  at infinity, then the map $U$ is a $H - C^1$ function, on the abstract
  Wiener space $(\mathcal{W}, \mathcal{H}, \mu)$ and we have
  \[ \nabla U^i (w) [h] = f (x) \partial_{\phi^i \phi^j}^2 V (\mathcal{I}w)
     \cdot \mathcal{I} (h^j) . \]
  Furthermore $U$ is $C^2$ Fr{\'e}chet differentiable as a map from
  $\mathcal{W}$ into $H$.
\end{proposition}

\begin{proof}
  The proof is essentially based on the fundamental theorem of calculus and
  the use of the Fourier transform. In order to give an idea of the proof we
  only \ prove the most difficult part, namely that $\nabla U$ is continuous
  with respect to translations by elements of $\mathcal{H}$, where continuity
  is understood with respect to the Hilbert-Schmidt norm for operators acting
  on $\mathcal{H}$.
  
  For fixed $w \in \mathcal{W}$, $h, h' \in \mathcal{H}$ we have, for $i = 1,
  \ldots, n$:
    \begin{multline}
      \nabla U^i (w + h') [h] - \nabla U^i (w) [h] =\\
      = f (x) \int_0^1 \partial_{\phi^i \phi^j \phi^r}^3 V ((m^2 - \Delta)^{-
      1} (w + t h')) \cdot \mathcal{I} (h^j) \cdot \mathcal{I} (h^{\prime r})
      \mathd t,
\label{equationc1h}    
\end{multline}

  where the sum over $j, r = 1, \ldots, n$ is implied. We recall that the
  Hilbert-Schmidt norm of an integral kernel is the integral of the square of
  the absolute value of the kernel. In our case the Fourier transform of the
  integral kernel representing the difference {\eqref{equationc1h}} is given
  by
  \[ \hat{K}^i_j (k, k') = \sum_{r = 1}^n \int_{\mathbb{R}^4} \int^1_0
     \frac{\hat{V}_{t, j r, f}^i (k - k_1)}{(| k_1 - k_2 |^2 + m^2)} \cdot
     \frac{\hat{h}^{\prime r} (k_1 - k_2)}{(| k_2 - k' |^2 + m^2)}
     \frac{\mathd k_1 \mathd k_2}{(2 \pi)^4}, \]
  where $\hat{V}_{t, j k, f}^i (k, l)$ is the Fourier transform of $f
  \partial_{\phi^i \phi^j \phi^k}^3 V (\mathcal{I} (w + t h'))$, $t \in [0,
  1]$. It is simple to prove that
  \[ \| \nabla U (w + h') [\cdot] - \nabla U (w) [\cdot] \|_2^2 \lesssim
     \int_{\mathbb{R}^4} | \hat{K}^i_r (k, k') \hat{K}^r_i (k', k) | \mathd k
     \mathd k'  \]
  \[ \lesssim \| \sqrt{f} e^{\alpha | \mathcal{I}w | + \alpha | \mathcal{I}h'
     |} \|^2_{\infty} \| \sqrt{f} \|^2_{L^2} \|h' \|^2_{\mathcal{H}}, \]
  where $\alpha$ depends on the exponential growth of $\partial^3 V$. Since
  $\| \sqrt{f} e^{\alpha | \mathcal{I}w | + \alpha | \mathcal{I}h' |}
  \|_{\infty}$ is always finite in $\mathcal{W}$ (for $\eta$ positive and
  small enough) we have proved the continuity of the map $h' \longmapsto
  \nabla U (w + h')$ with respect to the Hilbert-Schmidt norm.
\end{proof}

By the notation $\deg_2 (I_{\mathcal{H}} + K)$ we denote the regularized
Fredholm determinant (see Appendix~\ref{appendix_wienerspace} and
also~{\cite{Simon2005}}, Chapter~9) which is well defined when $K$ is a
Hilbert-Schmidt operator. The function $\det_2 (I_{\mathcal{H}} + \cdot)$ is a
smooth functional from the space of Hilbert-Schmidt operators (with its
natural norm) to $\mathbb{R}$ (see~{\cite{Simon2005}} Theorem~9.2 for the
proof of this fact).

We define the measurable map $N : \mathcal{W} \rightarrow \mathbb{N} \cup \{ +
\infty \}$
\[ N (w) \assign \left( \text{number of solutions $y \in \mathcal{W}$ to the
   equation $T (y) = w$} \right), \]
moreover let $M \subset \mathcal{W}$ be the set of zeros of the continuous
function $w \in \mathcal{W} \longmapsto \det_2 (I_{\mathcal{H}} + \nabla U
(w))$.

\begin{theorem}
  \label{theorem_weaksolution2}The function $N$ is greater or equal to $1$ and
  it is $\mu$-almost surely finite. Furthermore the map $T$ is proper.
\end{theorem}

\begin{proof}
  We define $\mathcal{T} (\hat{\phi}, w) = \hat{\phi} + U (\hat{\phi} + w)$.
  Obviously we have that $z$ is a solution to the equation $T (z) = w$ if and
  only if $\hat{\phi} = z - w$ is a solution to the equation $\mathcal{T}
  (\hat{\phi}, w) = 0$. On the other hand $\hat{\phi}$ is solution to the
  equation $\mathcal{T} (\hat{\phi}, w) = 0$ if and only if $\bar{\phi}
  =\mathcal{I} (\hat{\phi})$ is a solution to equation~{\eqref{equation2d2}}.
  By Theorem~\ref{theorem_existence1}, equation~{\eqref{equation2d2}} has at
  least one solution for any $w \in \mathcal{W}$ and so $N (w) \geqslant 1$
  for any $w \in \mathcal{W}$.
  
  Let $K$ be a compact set in $\mathcal{W}$ we have that $T^{- 1} (K) \subset
  K + (m^2 - \Delta) (\mathcal{F} (K))$ (where $\mathcal{F}$ is the set valued
  map introduced in Theorem~\ref{theorem_existence1}). Since $K$ is compact,
  by Theorem~\ref{theorem_existence1}, $\mathcal{F} (K)$ is compact in
  $\mathcal{C}^{2 -}_{- \eta}$ which implies that $(m^2 - \Delta) (\mathcal{F}
  (K))$ is compact in $\mathcal{C}^{0 -}_{- \eta}$. Since the immersion
  $\mathcal{C}^{0 -}_{- \eta} \longhookrightarrow \mathcal{W}$ is compact and
  the sum of two compact sets is compact, we obtain that $T$ is a proper map.
  
  Since by Proposition~\ref{proposition_Sard}, $\mu (T (M)) = 0$, for proving
  the theorem it is sufficient to prove that $N (w) < + \infty$ for $w
  \not{\in} T (M)$. If $w \not{\in} T (M)$ then $\tmop{id}_H + \nobracket
  \nabla U (w) |_{\mathcal{H}}$ is a linear invertible operator on
  $\mathcal{H}$ and so $\tmop{id}_{\mathcal{W}} + \nabla U (w)$ is a linear
  invertible operator on $\mathcal{W}$. By the implicit function theorem, we
  have that $T$ is a $C^1$ diffeomorphism between a neighborhood $B_w$ of $w$
  onto $T (B_w)$. This implies that the set $T^{- 1} (w)$ consists of isolated
  points. Since the map $T$ is proper, this means that $T^{- 1} (w)$ is a
  compact set made only by isolated points which implies that $T^{- 1} (w)$ is
  a finite set. \ 
\end{proof}

If $K : \mathcal{W} \rightarrow \mathcal{H}$ is an $H - C^1$ function we
denote by $\delta (K) : \mathcal{W} \rightarrow \mathbb{R}$ the well defined
Skorokhod integral of the map $K$ (see Appendix~\ref{appendix_wienerspace} for
an informal introduction of the concept, Appendix~B of~{\cite{Ustunel2000}}
for a more detailed treatment and Proposition~3.4.1 of~{\cite{Ustunel2000}}
for the proof of the fact that the Skorokhod integral of an $H - C^1$ function
is well defined).

\begin{theorem}
  \label{theorem_weaksolution}A probability measure $\nu$ is a weak solution
  to equation~{\eqref{eq:transf}} if and only if it is absolutely continuous
  with respect to $\mu$ and there exists a non-negative function $A \in
  L^{\infty} (\mu)$ such that $\sum_{y \in T^{- 1} (w)} A (y) = 1$ for
  $\mu$-almost all $w \in \mathcal{W}$ and $\frac{\mathd \nu}{\mathd \mu} = A
  | \Lambda_U |$ with
  \[ \Lambda_U (w) \assign \det_2 (I + \nabla U (w)) \exp \left( - \delta (U)
     (w) - \frac{1}{2} \|U (w) \|_{\mathcal{H}}^2 \right) . \]
\end{theorem}

\begin{proof}
  Recall that, by Proposition~\ref{proposition_Sard}, $\mu (T (M)) = 0$. This
  implies that for any weak solution $\nu$ we have $\nu (T^{- 1} (T (M))) =
  0$. Letting $\mathbb{W}^n \assign T^{- 1} (N = n) \cap T^{- 1} (T (M))$ we
  deduce that $\nu (\cup_n \mathbb{W}^n) = \sum_n \nu (\mathbb{W}^n) = 1$ and
  if we prove that $\nu$ is absolutely continuous with respect to $\mu$ on
  each $\mathbb{W}^n$ we have proved that $\nu$ is absolutely continuous with
  respect to $\mu$.
  
  Using $n$~times iteratively the Kuratowski-Ryll-Nardzewski selection
  theorem~(see Theorem~18.13 in~{\cite{Aliprantis2006}}) due to the fact that
  $T^{- 1} (x) \cap \mathbb{W}^n$ is composed by zero or $n$ elements, we can
  decompose the set $\mathbb{W}^n$ into $n$ measurable subsets
  $\mathbb{W}^n_1, \ldots, \mathbb{W}^n_n$ where the map $\nobracket T
  |_{\mathbb{W}^n_i}$ is invertible. This means that if $\Omega \subset
  \mathbb{W}^n$ we have $\nu (\Omega \cap \mathbb{W}^n_i) \leqslant \mu (T
  (\Omega))$. On the other hand we have that $\mu (T (\Omega)) = \int_{\Omega
  \cap \mathbb{W}^n_i} | \Lambda_U | \mathd \mu$. This implies that if $\mu
  (\Omega) = 0$ then $\nu (\Omega \cap \mathbb{W}^n_i) \leqslant \mu (T
  (\Omega)) = \int_{\Omega \cap \mathbb{W}^n_i} | \Lambda_U | \mathd \mu = 0$.
  As a consequence $\nu (\Omega) = \sum_i \nu (\Omega \cap \mathbb{W}^n_i) =
  0$ and $\nu$ is absolutely continuous with respect to $\mu$.
  
  Theorem~\ref{theorem_Lambda} below implies that for any measurable positive
  functions $f, A$ we have
  \begin{equation}
    \int f \circ T (w) A (w) | \Lambda_U (w) | \mathd \mu = \int f (w) \left(
    \sum_{y \in T^{- 1} (w)} A (y) \right) \mathd \mu . \label{eq:th50}
  \end{equation}
  Taking $f =\mathbb{I}_{T (M)}$ and $A = 1$ we deduce that $\int_{T^{- 1} (T
  (M))} | \Lambda_U | \mathd \mu = \mu (T (M)) = 0$. Therefore we can suppose
  that there exists a specific non-negative function $A$ such that $\mathd \nu
  = A | \Lambda_U | \mathd \mu$ and since $T_{\asterisk} (\nu) = \mu^{}$ we
  must have
  \[ \int f (w) \mathd \mu = \int f \circ T (w) \mathd \nu = \int f \circ T
     (w) A (w) | \Lambda_U (w) | \mathd \mu, \]
  for any bounded measurable function $f$. Comparing this
  with~{\eqref{eq:th50}} we deduce that $\sum_{y \in T^{- 1} (w)} A (y) = 1$
  for ($\mu$-)almost all $w \in \mathcal{W}.$
  
  On the other hand, using again Theorem~\ref{theorem_Lambda} it is simple to
  prove that if $\mathd \nu = A | \Lambda_U | \mathd \mu$ and $\sum_{y \in
  T^{- 1} (w)} A (y) = 1$ then $\nu$ is a weak solution to
  equation~{\eqref{eq:transf}}.
\end{proof}

\begin{remark}
  \label{remark_weaksolution1}If $S$ is any strong solution to
  equation~{\eqref{eq:transf}} then $\nu = S_{\asterisk} \mu$ is a weak
  solution. Furthermore it is simple to prove that the weak solutions of the
  form $S_{\asterisk} \mu$, where $S$ is some strong solution to
  {\eqref{equation2d1}}, are the extremes of the convex set $\mathfrak{W}
  \assign \left\{ \text{$\nu$ satisfying $T_{\asterisk} \nu = \mu$} \right\}$.
  Using a lemma (precisely Lemma~\ref{lemma_reduction1}) that we shall prove
  below, it follows from this that $\mathfrak{W}$ is weakly compact and thus,
  by Krein--Milman theorem (see Theorem~3.21 in~{\cite{Rudin1973}}), any
  measure $\nu \in \mathfrak{W}$ can be written as convex combination of
  measures induced by strong solutions. \ 
\end{remark}

\begin{corollary}
  \label{corollary_uniqueness2}If $V$ satisfies Hypothesis~C there exists only
  one weak solution $\nu$ to equation~{\eqref{eq:transf}} and we have that
  $\frac{\mathd \nu}{\mathd \mu} = | \Lambda_U |$ and $\nu = S_{\asterisk}
  \mu$ (where $S$ is the only strong solution to equation~{\eqref{eq:transf}}
  and $\Lambda_U$ is as in Theorem~\ref{theorem_weaksolution}).
\end{corollary}

\begin{proof}
  If $V$ satisfies Hypothesis~C, by Corollary~\ref{corollary_uniqueness1}, $T$
  is invertible and by Theorem~\ref{theorem_weaksolution} we have that $\nu$
  is unique and $\frac{\mathd \nu}{\mathd \mu} = | \Lambda_U | .$ By
  Remark~\ref{remark_weaksolution1} we have that $S_{\asterisk} \mu$, where
  $S$ is the unique strong solution of~{\eqref{eq:transf}}, is the unique weak
  solution to the same equation.
\end{proof}

\section{Elliptic stochastic quantization}\label{sec:dim-red}

In this section we want to prove the dimensional reduction of
equation~{\eqref{equation2d1}}, namely that the law in 0 of at least a (weak)
solution to equation~{\eqref{eq:transf}}, has an explicit expression in terms
of the potential $V$.

\

The original idea of Parisi and Sourlas~{\cite{parisi_random_1979}} for
proving this relations was to transform expectations involving the solution
$\phi$ to equation~{\eqref{equation2d1}} (taken at the origin) into an
integral of the form
\begin{equation}
  \mathbb{E} [h (\phi (0))] = \int h (\mathcal{I}w (0)) \det (I + \nabla U
  (\mathcal{I}w)) e^{- \langle U (\mathcal{I}w), \mathcal{I}w \rangle -
  \frac{1}{2} \| U (\mathcal{I}w) \|^2_{\mathcal{H}}} \mathd \mu (w),
  \label{equationdimensional1}
\end{equation}
where $U$ is defined in equation {\eqref{equationU}}. Then one can express the
weight on the right hand side of~{\eqref{equationdimensional1}} as the
exponential $e^{\int V (\Phi) \mathd x \mathd \theta \mathd \bar{\theta}}$
involving the superfield
\[ \Phi (x, \theta, \bar{\theta}) = \varphi (x) + \psi (x) \theta + \bar{\psi}
   (x) \bar{\theta} + \omega (x) \theta \bar{\theta}, \]
(see Section~\ref{sec:super} and Section~\ref{section:supersymmetry} for a
more precise description) constructed from the real Gaussian free field
$\varphi$ over $\mathbb{R}^2$, two additional fermionic (i.e. anticommuting)
fields $\psi, \bar{\psi}$ and the complex Gaussian field $\omega$. Introducing
these new anticommuting fields it can be argued that the
integral~{\eqref{equationdimensional1}} admits an invariance property with
respect to supersymmetric transformations. This implies the dimensional
reduction, i.e.
\begin{equation}
  \eqref{equationdimensional1} = \int h (\varphi (0)) e^{- \int V
  (\Phi) \mathd x \mathd \theta \mathd \bar{\theta}} \mathcal{D} \Phi =
  \int_{\mathbb{R}^n} h (y) \mathd \kappa (y) . \label{equationdimensional2}
\end{equation}
Unfortunately this reasoning is only heuristic since the integral on the right
hand side of~{\eqref{equationdimensional1}} is not well defined without a
spatial cut-off, given that both the determinant and the exponential are
infinite.

For polynomial potentials $V$, a rigorous version of this reasoning was
proposed by Klein~et~al.~{\cite{Klein1984}}. More precisely Klein~et~al. give
a rigorous proof of the relationship~{\eqref{equationdimensional2}}
introducing a suitable modification due to the presence of the spatial cut-off
$f$, but they do not discuss the relationship between
equation~{\eqref{equation2d1}} and the
reduction~{\eqref{equationdimensional1}}.

\

In this section we do not want to propose a rigorous version of the previous
reasoning which will be given in Section \ref{sec:super}. Here we only assume
that the conclusion of Parisi and Sourlas' formal argument holds for a general
enough class of potentials. More precisely we assume Theorem
\ref{theorem_main1} below.

For technical reasons, which will become clear in the following (see
Remark~\ref{remark_hypotheses} below), in order to state Theorem
\ref{theorem_main1}, we need first to introduce an additional class of
potentials.

{\descriptionparagraphs{\item[Hypothesis~$V_{\lambda}$.] We have the
decomposition
\[ V = V_B + \lambda V_U, \qquad V_U (y) = \sum_{i = 1}^n (y^i)^4, \qquad y =
   (y^1, \ldots, y^n) \in \mathbb{R}^n, \]
with $\lambda > 0$ and $V_B$ a bounded function with all bounded derivatives
on $\mathbb{R}^n$. }}

In Section~\ref{sec:super} below we will exploit a supersymmetric argument,
described briefly at the beginning of this section, for the family of
potentials $V$ satisfying the more restrictive Hypothesis~$V_{\lambda}$ to
prove that in this case a cut-off version of
equation~{\eqref{equationdimensional2}}.

\begin{theorem}
  \label{theorem_main1}Under the Hypotheses~CO and~$V_{\lambda}$ if $h$ is any
  real measurable bounded function defined on $\mathbb{R}^n$ then we have
  \[ \int_{\mathcal{W}} h (\mathcal{I}w (0)) \Lambda_U (w) \Upsilon_f
     (\mathcal{I}w) \mathd \mu (w) = Z_f \int_{\mathbb{R}^n} h (y) \mathd
     \kappa (y), \]
  where $Z_f = \int_{\mathcal{W}} \Lambda_U (w) \Upsilon_f (\mathcal{I}w)
  \mathd \mu (w) > 0$.
\end{theorem}

\begin{proof}
  The proof is given in Section~\ref{sec:super} below.
\end{proof}

In the rest of this section we want to show how to derive from
Theorem~\ref{theorem_main1} the dimensional reduction result for the solution
to the elliptic SPDE. More precisely the goal of the rest of this section is
to prove the following theorem.

\begin{theorem}
  \label{theorem_reduction2}Under the Hypotheses~CO and~QC there exists (at
  least) one weak solution $\nu$ to equation~{\eqref{equation2d1}} such that
  for any measurable bounded function $h$ defined on $\mathbb{R}^n$ we have
  \begin{equation}
    \begin{array}{lll}
      \int_{\mathcal{W}} h (\mathcal{I}w (0)) \Upsilon_f (\mathcal{I}w) \mathd
      \nu (w) & = & \int_{\mathcal{W}} h (\mathcal{I}w (0)) \Upsilon_f
      (\mathcal{I}w) \Lambda_U (w) \mathd \mu (w)\\
      & = & Z_f \int_{\mathbb{R}^n} h (y) \mathd \kappa (y)
      \label{equationgaussian1}
    \end{array}
  \end{equation}
  where $Z_f = \int_{\mathcal{W}} \Upsilon_f (\mathcal{I}w) \mathd \nu (w) >
  0$.
\end{theorem}

This result is very important since it implies Theorem~\ref{th:dim-red-qc} and
Theorem~\ref{th:dim-red-1}.

\begin{proof*}{Proof of Theorem~\ref{th:dim-red-qc} and
Theorem~\ref{th:dim-red-1}}
  The relation {\eqref{equationgaussian1}} can be expressed in the following
  more probabilistic way. Suppose that on a given probability space
  $(\Omega^{\nu}, \mathbb{P}^{\nu})$, the map $\phi : \mathbb{R}^2 \times
  \Omega^{\nu} \rightarrow \mathbb{R}^n$ gives the weak solution $\nu$ of
  Theorem~\ref{theorem_reduction2}, namely that the law of the
  $\mathcal{W}$-random variable $(m^2 - \Delta) \phi (\cdot, \omega)$ is the
  measure $\nu$. Then we have that, for any real measurable bounded function
  defined on $\mathbb{R}^n$,
  \[ \mathbb{E}_{\mathbb{P}^{\nu}} \left[ h (\phi (0)) \frac{\Upsilon_f
     (\phi)}{Z_f} \right] = \int_{\mathcal{W}} h (y) \mathd \kappa (y), \]
  namely we have proven Theorem~\ref{th:dim-red-qc}. If we assume Hypothesis~C
  then by Corollary~\ref{corollary_uniqueness1},
  Corollary~\ref{corollary_uniqueness2} and Theorem~\ref{theorem_reduction2}
  there exists a unique strong solution satisfying~{\eqref{equationgaussian1}}
  and we have proven as a consequence Theorem~\ref{th:dim-red-1}.
\end{proof*}

The proof of Theorem~\ref{theorem_reduction2} will be given in several steps
of wider degree of generality with respect to the hypothesis on the potential
$V$. Before we prove an auxiliary result.

\begin{lemma}
  \label{theorem_lambda2}Under the Hypothesis~$V_{\lambda}$ we have that
  \begin{equation}
    \int_{\mathcal{W}} g \circ T (w) \Lambda_U (w) \mathd \mu (w) =
    \int_{\mathcal{W}} g (w) \mathd \mu (w) . \label{equationreduction5}
  \end{equation}
  where $g$ is any bounded measurable function defined on $\mathcal{W}$.
\end{lemma}

\begin{proof}
  Using the methods of Section~\ref{sec:solutions} we can prove that the map
  $T$ satisfies Hypotheses DEG1, DEG2, DEG3 of
  Appendix~\ref{appendix_wienerspace}. The claim then follows from
  Theorem~\ref{theorem_wienerspace3} and Theorem~\ref{theorem_wienerspace4}
  below, where we can choose the function $g$ to be any bounded continuous
  function since $\Lambda_U \in L^1 (\mu)$ under Hypothesis~$V_{\lambda}$.
\end{proof}

\begin{proposition}
  \label{proposition_reduction1}Under the Hypotheses~CO and~$V_{\lambda}$
  there exists at least one weak solution $\nu$ to
  equation~{\eqref{eq:transf}} satisfying~{\eqref{equationgaussian1}}.
\end{proposition}

\begin{proof*}{Proof}
  Let $\mathcal{V} \subset L^1 (| \Lambda_U | \mathd \mu)$ be the span of the
  two linear spaces $\mathcal{V}_1, \mathcal{V}_2 \subset L^1 (| \Lambda_U |
  \mathd \mu)$ where $\mathcal{V}_1$ is composed by the functions of the form
  $g \circ T$, where $g$ is a measurable function defined on $\mathcal{W}$
  such that $g \circ T \in L^1 (| \Lambda_U | \mathd \mu)$, and
  $\mathcal{V}_2$ is formed by the functions of the form $h (\mathcal{I}w (0))
  \Upsilon_f (\mathcal{I}w)$, where $h$ is a measurable function defined on
  $\mathbb{R}^n$ such that $h (\mathcal{I}w (0)) \Upsilon_f (\mathcal{I}w) \in
  L^1 (| \Lambda_U | \mathd \mu)$. Note that $\mathcal{V}_1$ and
  $\mathcal{V}_2$, and so $\mathcal{V} = \tmop{span} \{ \mathcal{V}_1,
  \mathcal{V}_2 \}$, are non-void since, under the Hypotheses $V_{\lambda}$
  and CO (see Lemma~\ref{lemma_Lp} below), $\Lambda_U \in L^p (\mu)$ and so $g
  \circ T, h (\mathcal{I}w (0)) \Upsilon_f (\mathcal{I}w) \in L^1 (\mu)$
  whenever $g, h$ are bounded. Define a positive functional $\hat{L} :
  \mathcal{V} \rightarrow \mathbb{R}$ by extending via linearity the relations
  \begin{eqnarray}
    \hat{L} (h (\mathcal{I}w (0)) \Upsilon_f (\mathcal{I}w)) & \assign & \int
    h (\mathcal{I}w (0)) \Upsilon_f (\mathcal{I}w) \Lambda_U (w) \mathd \mu
    (w) \label{equationreduction2} \\
    \hat{L} (g \circ T) & \assign & \int g (w) \mathd \mu (w) .
    \label{equationreduction1} 
  \end{eqnarray}
  to the whole $\mathcal{V}$. We have to verify that $\hat{L}$ is well defined
  and positive on $\mathcal{V} .$ Suppose that there exist functions $g$ and
  $h$ such that $g \circ T = h (\mathcal{I}w (0)) \Upsilon_f (\mathcal{I}w)$
  then, by Lemma~\ref{theorem_lambda2}, we have
  \begin{equation}
    \int_{\mathcal{W}} g \mathd \mu = \int_{\mathcal{W}} g \circ T \Lambda_U
    \mathd \mu = \int_{\mathcal{W}} h (\mathcal{I}w (0)) \Upsilon_f
    (\mathcal{I}w) \Lambda_U \mathd \mu . \label{equationreduction4}
  \end{equation}
  This implies that $\hat{L}$ is well defined on $\mathcal{V}_1 \cap
  \mathcal{V}_2$ and so on $\mathcal{V}$. Obviously $\hat{L}$ is positive on
  $\mathcal{V}_2$, and, by Theorem~\ref{theorem_main1} we have
  \begin{equation}
    \begin{array}{c}
      \hat{L} (h (\mathcal{I}w (0)) \Upsilon_f (\mathcal{I}w)) =
      \int_{\mathcal{W}} h (\mathcal{I}w (0)) \Upsilon_f (\mathcal{I}w)
      \Lambda_U \mathd \mu =\\
      = Z_f \int_{\mathbb{R}^n} h (y) \mathd \kappa (y) \geqslant 0
    \end{array} \label{equationreduction3}
  \end{equation}
  whenever $h$, and so $h (\mathcal{I}w (0)) \Upsilon_f (\mathcal{I}w)$, is
  positive. This means that $\hat{L}$ is positive.
  
  For any $f = g \circ T \in \mathcal{V}_1$, by Theorem~\ref{theorem_Lambda}
  and Theorem~\ref{theorem_weaksolution2}, we have
  
  \begin{multline*}
    | \hat{L} (f) | = \left| \int_{\mathcal{W}} g (w) \mathd \mu (w) \right|
    \leqslant \int_{\mathcal{W}} | g (w) | N (w) \mathd \mu (w) =
    \int_{\mathcal{W}} | g \circ T (w) \Lambda_U (w) | \mathd \mu (w) = \| f
    \Lambda_U \|_1 .
  \end{multline*}
  
  On the other hand, if $f \in \mathcal{V}_2$, by
  relation~{\eqref{equationreduction2}}, $\hat{L} (f) \leqslant \| f \Lambda_U
  \|_1$. These two inequalities and the positivity of $\hat{L}$ imply, by
  Theorem~8.31 of {\cite{Aliprantis2006}} on the extension of positive
  functionals on Riesz spaces, that there exists at least one positive
  continuous linear functional $L$ on $L^1 (| \Lambda_U | \mathd \mu)$, such
  that $L (f) = \hat{L} (f)$ for any $f \in \mathcal{V}$. The functional $L$
  defines the weak solution to equation {\eqref{eq:transf}} we are looking
  for. Indeed, since $L$ is a continuous positive functional on $L^1 (|
  \Lambda_U | \mathd \mu)$ there exists a measurable positive function $B \in
  L^{\infty} (| \Lambda_U | \mathd \mu) \subset L^{\infty} (\mathd \mu)$ such
  that $L (f) = \int_{\mathcal{W}} f (w) B (w) | \Lambda_U (w) | \mathd \mu
  (w)$. Since $\Lambda_U \in L^p$ by Lemma~\ref{lemma_Lp} below, we have $1
  \in \mathcal{V}_1$ and so $L (1) = \int_{\mathcal{W}} 1 \mathd \mu (w) = 1$.
  This implies, since the function $B$ is positive, that the $\sigma$-finite
  measure $\mathd \nu = B | \Lambda_U | \mathd \mu$ is a probability measure.
  Furthermore, since $\mathcal{V}_1$ contains all the functions $g \circ T$,
  where $g$ is measurable and bounded, equality {\eqref{equationreduction1}}
  implies that $T_{\asterisk} (\nu) = \mu$. This means that $\nu$ is a weak
  solution to equation~{\eqref{eq:transf}}. Finally since $\mathcal{V}_2$
  contains all the functions of the form $h (\mathcal{I}w (0)) \Upsilon_f
  (\mathcal{I}w)$ where $h$ is measurable and bounded on $\mathbb{R}^n$ the
  measure $\nu$ satisfies the thesis of the theorem.
\end{proof*}

Unfortunately we cannot repeat this reasoning for general potentials
satisfying the weaker Hypothesis~QC since both Theorem~\ref{theorem_main1} and
Proposition~\ref{proposition_reduction1} exploit an $L^p$ bound on $\Lambda_U$
(see Lemma~\ref{lemma_Lp} below) that cannot be obtained for more general
potentials. Thus the idea is to generalize
equation~{\eqref{equationgaussian1}} without passing from
equation~{\eqref{equationdimensional2}}. Indeed it is possible to approximate
any potential $V$ satisfying Hypothesis~QC by a sequence of potentials
$(V_i)_i$ satisfying Hypothesis~$V_{\lambda}$ in such a way that the sequence
of weak solutions $(\nu_i)_i$ associated with $(V_i)_i$ converges (weakly) to
a weak solution associated with the potential $V$ (see
Lemma~\ref{lemma_reduction1}, Lemma~\ref{lemma_reduction2} and
Lemma~\ref{lemma_reduction3} below). Since
equation~{\eqref{equationgaussian1}} involves only integrals with respect to a
weak solution to equation~{\eqref{equation2d1}}, we are able to prove that
equation~{\eqref{equationgaussian1}} holds for any potential $V$ approximating
its weak solution $\nu$ by the sequence $(\nu_i)_i$ satisfying
equation~{\eqref{equationgaussian1}}. \

Let us now set up the approximation argument, starting with a series of
lemmas about convergence of weak solutions.

\begin{lemma}
  \label{lemma_reduction1}Let $\{ T_i \}_{i \in \mathbb{N}}$ be a sequence of
  continuous maps on $\mathcal{W}$ such that for any compact $K \subset
  \mathcal{W}$ we have that $\bigcup_{i \in \mathbb{N}} T_i^{- 1} (K)$ is
  pre-compact and there exists a continuous map $T$ such that $T_i \rightarrow
  T$ uniformly on the compact subsets of \ $\mathcal{W}$. Let $\mathbb{M}_i$
  be a set of probability measures on $\mathcal{W}$ defined as follows
  \[ \mathbb{M}_i \assign \left\{ \text{$\nu$ probability measure on
     $\mathcal{W}$ such that $T_{j, \asterisk} (\nu) = \mu$ for some $j
     \geqslant i$} \right\} . \]
  Then $\mathbb{M} \assign \bigcap_{i \in \mathbb{N}} \bar{\mathbb{M}}_i$,
  where the closure is taken with respect to the weak topology on the set of
  probability measures on $\mathcal{W}$, is non-void and
  \[ \mathbb{M} \subset \left\{ \text{$\nu$ probability measure on
     $\mathcal{W}$ such that $T_{\asterisk} (\nu) = \mu$} \right\} . \]
\end{lemma}

\begin{proof}
  First of all we prove that $\mathbb{M}_i$ is pre-compact for any $i \in
  \mathbb{N}$. This is equivalent to proving that the measures in
  $\mathbb{M}_i$ are tight. Let $\tilde{K}$ be a compact set such that $\mu
  (\tilde{K}) \geqslant 1 - \epsilon$ for a fixed $0 < \epsilon < 1$, then $K
  \assign \overline{\bigcup_{i \in \mathbb{N}} T^{- 1}_i (\tilde{K})}$ is a
  compact set in $\mathcal{W}$. Consider $\nu \in \mathbb{M}_j$ then there
  exists $T_k$ such that $T_{k, \asterisk} \nu = \mu$. This implies
  \[ \nu (K) \geqslant \nu \left( \bigcup_i T^{- 1}_i (\tilde{K}) \right)
     \geqslant \nu (T^{- 1}_k (\tilde{K})) \geqslant \mu (\tilde{K}) \geqslant
     1 - \epsilon, \]
  for any $k \in \mathbb{N}$. Since $\mathbb{M}_i$ are pre-compact,
  $\bar{\mathbb{M}}_i$ are compact and $\bar{\mathbb{M}}_i \subset
  \bar{\mathbb{M}}_j$ if $i \geqslant j$. This implies that $\mathbb{M}$ is
  non-void. If we consider a $\nu \in \mathbb{M}$ there exists a sequence
  $\nu_k$ weakly converging to $\nu$, for $k \rightarrow + \infty$, such that
  $T_{i_k, \asterisk} (\nu_k) = \mu$ and $i_k \rightarrow + \infty$. Proving
  that $T_{\asterisk} (\nu) = \mu$ is equivalent to prove that for any $C^1$
  bounded function $g$ with bounded derivatives defined on $\mathcal{W}$
  taking values in $\mathbb{R}$ we have $\int g \circ T \mathd \nu = \int g
  \mathd \mu$. Let $K$ the compact set defined before, then there exists a $k
  \in \mathbb{N}$ such that $\sup_{w \in K} \| T_{i_k} (w) - T (w) \|
  \leqslant \epsilon$ and that $\left| \int_{\mathcal{W}} g \circ T \mathd \nu
  - \int_{\mathcal{W}} g \circ T \mathd \nu_k \right| \leqslant \epsilon$, for
  the arbitrary $0 < \epsilon < 1$. This implies that
  \begin{eqnarray}
    \left| \int_{\mathcal{W}} g \circ T \mathd \nu - \int_{\mathcal{W}} g
    \mathd \mu \right| & \leqslant & \left| \int_{\mathcal{W}} g \circ T
    \mathd \nu - \int_{\mathcal{W}} g \circ T \mathd \nu_i \right| +
    \nonumber\\
    &  & + \left| \int_K (g \circ T - g \circ T_{i_k}) \mathd \nu_k \right| +
    \nonumber\\
    &  & + \| g \|_{\infty} \epsilon + \left| \int_{\mathcal{W}} g \circ
    T_{i_k} \mathd \nu_k - \int_{\mathcal{W}} g \mathd \mu \right| \nonumber\\
    & \leqslant & \epsilon + \| \nabla g \|_{\infty} \epsilon + \| g
    \|_{\infty} \epsilon . \nonumber
  \end{eqnarray}
  Since $\epsilon$ is arbitrary, from this it follows that
  $\int_{\mathcal{W}} g \circ T \mathd \nu = \int_{\mathcal{W}} g \mathd \mu$.
  
\end{proof}

\begin{remark}
  \label{remark_reduction1}The proof of Lemma~\ref{lemma_reduction1} proves
  also that given any sequence of $\nu_i \in \mathbb{M}_i$ there exists a
  subsequence converging weakly to $\nu \in \mathbb{M}$.
\end{remark}

\begin{remark}
  \label{remark_Vi}In the following we consider a sequence of functions $V_i$
  satisfying Hypothesis~QC. To each function $V_i$ of the sequence it is
  possible to associate a map $U_i : \mathcal{W} \rightarrow \mathcal{H}$
  defined by $U_i (w) \assign f \partial V_i (\mathcal{I} w)$ and the
  corresponding map $T_i : \mathcal{W} \rightarrow \mathcal{W}$ defined by
  $T_i (w) = w + U_i (w)$.
\end{remark}

\begin{lemma}
  \label{lemma_reduction2}Let $\{ V_i \}_{i \in \mathbb{N}}$ be a sequence of
  potentials satisfying the Hypothesis~QC and converging to the potential $V$,
  and such that $\partial V_i$ converges uniformly to $\partial V$ on compact
  subsets of $\mathbb{R}^n$; moreover we assume that $V_i$, $V$, $\partial
  V_i$ and $\partial V$ are uniformly exponentially bounded and there exists a
  common function $H$ entering Hypothesis~QC for $\{ V_i \}_{i \in
  \mathbb{N}}$ and $V$. Let $T_i$, $T$ be the maps on $\mathcal{W}$ associated
  with $V_i$ and $V$ respectively as in Remark \ref{remark_Vi}. Then the
  sequence $\{ T_i \}_{i \in \mathbb{N}}$ satisfies the hypothesis of
  Lemma~\ref{lemma_reduction1}. \ 
\end{lemma}

\begin{proof}
  Note that the a priori estimates~{\eqref{equationexistence2d1}}
  and~{\eqref{equationexistence2d2}} in Lemma~\ref{lemma_bound} are uniform in
  $i \in \mathbb{N}$ since they depend only on the function $H$ and the
  exponential growth of $V_i, V, \partial V_i, \partial V$. From this we can
  deduce the pre-compactness of the set $K = \bigcup_{i \in \mathbb{N}} T_i^{-
  1} (\tilde{K})$ for any compact set $\tilde{K} \subset \mathcal{W}$ using a
  reasoning similar to the one proposed in Theorem~\ref{theorem_existence1}
  and Theorem~\ref{theorem_weaksolution2}.
  
  Proving that $T_i$ converges to $T$ uniformly on the compact sets is
  equivalent to prove that the map $U_i (w) (x) = f (x) \partial V_i
  (\mathcal{I}w (x))$ converges to $U (w) (x) = f (x) \partial V (\mathcal{I}w
  (x))$ in $L^2$ uniformly on the compact subsets of $\mathcal{W}$. Let $K$ be
  a compact set of $\mathcal{W}$, then there exists an $M > 0$ such that $|
  \mathcal{I}w (x) | \leqslant M (1 + | x |^{\eta})$ (where we suppose without
  loss of generality that $\eta < 1$). By hypotheses we have that there exist
  two constants $\alpha, \beta > 0$ such that $| \partial V_i (y) |, |
  \partial V (y) | \leqslant e^{\alpha | y | + \beta}$, thus there exists a
  compact subset $\mathfrak{K}$ of $\mathbb{R}^2$ such that
  $\int_{\mathfrak{K}^c} (f (x))^2 \exp (2 \alpha M (1 + | x |^{\eta}) + 2
  \beta) \mathd x \leqslant \epsilon$, for some $\epsilon \in (0, 1)$. Denote
  by $B_{\epsilon}$ the ball of radius $\sup_{x \in \mathfrak{K}} M (1 + | x
  |^{\eta})$ then we have
  \begin{eqnarray}
    \sup_{w \in K} \| U_i (w) - U (w) \|^2_{\mathcal{H}} & \leqslant & 2
    \left| \int_{\mathfrak{K}^c} (f (x))^2 e^{2 \alpha M (1 + | x |^{\eta}) +
    2 \beta} \mathd x \right| \nonumber\\
    &  & + \sup_{w \in K} \left| \int_{\mathfrak{K}} (f (x))^2 | \partial V
    (\mathcal{I}w) - \partial V_i (\mathcal{I}w) |^2 \mathd x \right|
    \nonumber\\
    & \leqslant & 2 \epsilon + (\sup_{y \in B_{\epsilon}} | \partial V (y) -
    \partial V_i (y) |)^2 \int_{\mathfrak{K}} (f (x))^2 \mathd x \nonumber\\
    & \rightarrow & 2 \epsilon, \nonumber
  \end{eqnarray}
  as $i \rightarrow + \infty .$ This means that $\lim_{i \rightarrow + \infty}
  (\sup_{w \in K} \| U_i (w) - U (w) \|^2_{\mathcal{H}}) \leqslant 2
  \epsilon$, and since $\epsilon$ is arbitrary in $(0, 1)$ the theorem is
  proved. 
\end{proof}

\begin{lemma}
  \label{lemma_reduction3}Let $V$ be a potential satisfying Hypothesis~QC,
  then there exists a sequence $\{ V_i \}_{i \in \mathbb{N}}$ of bounded
  smooth potentials converging to $V$ and satisfying the hypothesis of
  Lemma~\ref{lemma_reduction2}. 
\end{lemma}

\begin{proof}
  Let $V$ be a potential satisfying the Hypothesis~QC and let $\tilde{H}$ the
  function whose existence is guaranteed by Hypothesis~QC. Let, for any $N \in
  \mathbb{N}$, $v_N \assign \sup_{y \in B (0, N)} | V (y) |$ and let
  $\tilde{V}^N \assign G_{v_N} \circ V$ where
  \[ G_k (z) \assign \left\{\begin{array}{ll}
       z \quad & \text{if $| z | \leqslant k$},\\
       k & \text{if $| z | > k$} .
     \end{array}\right. \]
  Let $\rho$ be a smooth compactly supported mollifier and denote by
  $\rho_{\epsilon}$ the function $\rho_{\epsilon} (y) \assign \epsilon^{- n}
  \rho \left( \frac{y}{\epsilon} \right) .$ We want to prove that $V^N =
  \tilde{V}^N \asterisk \rho_{\epsilon_N}$, for a suitable sequence
  $\epsilon_N \in \mathbb{R}_+$, is the approximation requested by the lemma.
  Without loss of generality we can suppose that $\tilde{H}$ is a positive
  function depending only on the radius $| y |$ and increasing as $| y |
  \rightarrow + \infty$. Under these conditions, Hypothesis~QC is equivalent
  to say that for any unit vector $\hat{n} \in \mathbb{S}^n$ we have that for
  any $y \in \mathbb{R}^n$
  \[ \max (- \hat{n} \cdot \partial V (y + r \hat{n}), 0) \leqslant \tilde{H}
     (y) . \]
  We want to prove that $H (| y |) = \tilde{H} (| y | + \sup_N (\epsilon_N))$
  is the function requested by the lemma.
  
  Since for any unit vector $\hat{n} \in \mathbb{S}^n$ we have $| \hat{n}
  \cdot \partial \tilde{V}^N | \leqslant | \hat{n} \cdot \partial V|$ and
  since $\tilde{V}^N$ is absolutely continuous we obtain
  \[ - \hat{n} \cdot \partial V^N (y + r \hat{n}) = ((- \hat{n} \cdot \partial
     \tilde{V}^N) \asterisk \rho_{\epsilon_N}) (y + r \hat{n}) \]
  \[ \leqslant (\max (- \hat{n} \cdot \partial V (\cdot + r \hat{n}), 0)
     \asterisk \rho_{\epsilon_N}) (y) \leqslant \tilde{H} \asterisk
     \rho_{\epsilon_N} (y) . \]
  Furthermore we have that $\tilde{V}^N = V$ on $B (0, N - 1)$ and so there
  exists a sequence $\{ \epsilon_N \}_N$ such that $\epsilon_N \rightarrow 0$
  and $\sup_{x \in B (0, N - 1)} | \partial V^N (x) - \partial V (x) |
  \leqslant \frac{1}{N}$. Since $V^N$ is smooth and bounded and
  \[ \tilde{H} \asterisk \rho_{\epsilon_N} (y) \leqslant \tilde{H} (| y | +
     \sup_N (\epsilon_N)) = H (y), \]
  we conclude the claim.
\end{proof}

Finally we are able to prove~{\eqref{equationgaussian1}} for all QC
potentials, which will conclude this section.

\begin{proof*}{Proof of Theorem~\ref{theorem_reduction2}}
  By Proposition~\ref{proposition_reduction1} the
  equality~{\eqref{equationgaussian1}} holds when $V$ satisfies the
  Hypothesis~$V_{\lambda}$ for some $\lambda > 0$, i.e. if $V (y) =
  V_{\lambda, V_B} (y) = V_B (y) + \lambda \sum_{k = 1}^n (y^k)^4$ for some
  bounded potential $V_B$. It is clear that if $\lambda_i \rightarrow 0$ the
  potentials $V_{\lambda_i, V_B}$ converge to the potential $V_B$ and the
  hypothesis of Lemma~\ref{lemma_reduction2} hold. This means that if
  $\hat{\nu}_i$ is a sequence of probability measures such that $\hat{\nu}_i$
  is a weak solution to the equation associated with $V_{\lambda_i, V_B}$
  satisfying the thesis of Proposition~\ref{proposition_reduction1}, by
  Remark~\ref{remark_reduction1} and Lemma~\ref{lemma_reduction1}, there
  exists a probability measure $\hat{\nu}$, that is a weak solution to the
  equation associated with $V_B$, such that $\hat{\nu}_i \rightarrow
  \hat{\nu}$ in the weak sense, as $i \rightarrow \infty$ and $\lambda_i
  \rightarrow 0$.
  
  We want to prove that $\hat{\nu}$ is a weak solution to the equation
  associated with $V_B$ satisfying equation {\eqref{equationgaussian1}}. The
  previous claim is equivalent to proving that
    \begin{multline}
      \int_{\mathcal{W}} g (\mathcal{I}w (0)) e^{4 \int f' (x) V_{\lambda_i,
      B} (\mathcal{I}w (x)) \mathd x} \mathd \hat{\nu}^i (w) \longrightarrow \int_{\mathcal{W}} g (\mathcal{I}w (0)) e^{4 \int f' (x)
      V_B (\mathcal{I}w (x)) \mathd x} \mathd \hat{\nu} (w),
\label{equationreduction7}    
\end{multline}

  as $\lambda \rightarrow 0$, for any continuous bounded function $g$, and
  that $\kappa_{\lambda_i} \rightarrow \kappa_B$ weakly, where $\mathd
  \kappa_{\lambda_i} = \exp (- 4 \pi V_{\lambda_i, B}) \mathd x /
  Z_{\lambda_i}$ and $\mathd \kappa_B = \mathd \kappa_{\lambda_i} = \exp (- 4
  \pi V_B) \mathd x / Z_B$. Proving relation~{\eqref{equationreduction7}} is
  equivalent to prove that
  \[ \int f' (x) V_{\lambda_i, B} (\mathcal{I}w (x)) \mathd x \rightarrow \int
     f' (x) V_B (\mathcal{I}w (x)) \mathd x \]
  uniformly on compact sets of $\mathcal{W}$. This assertion can be easily
  proved using the methods of Lemma~\ref{lemma_reduction2}. Indeed for any $w$
  in the compact set $K \subset \mathcal{W}$, using the same notations of the
  proof of Lemma~\ref{lemma_reduction2}, we have
  
  \begin{multline*}
    \left| \int f' V_{\lambda_i, B} (\mathcal{I}w) \mathd x - \int f' V_B
    (\mathcal{I}w) \mathd x \right| \lesssim \lambda_i \int_{\mathfrak{}} | f' (x) | (M (1 + | x |^{\eta}))^4
    \mathd x = C_K \lambda_i \rightarrow 0.
  \end{multline*}
  
  The weak convergence of $\kappa_{\lambda_i}$ to $\kappa_B$ easily follows
  from Lebesgue's dominated convergence theorem.
  
  The previous reasoning proves the theorem for any bounded potential $V_B$.
  Using Lemma~\ref{lemma_reduction3} we can approximate any potential $V$
  satisfying Hypothesis~QC by a sequence of bounded potentials $V_{B, i}$.
  Using Lemma~\ref{lemma_reduction2}, Remark~\ref{remark_reduction1},
  Lemma~\ref{lemma_reduction1} and a reasoning similar to the one exploited in
  the first part of the proof we obtain the thesis of the theorem for a
  general potential satisfying Hypothesis~QC. 
\end{proof*}

\section{Dimensional reduction}\label{sec:super}

Define
\begin{equation}
  \Xi (h) \assign \int_{\mathcal{W}} h (\mathcal{I}w (0)) \Lambda_U (w)
  \frac{\Upsilon_f (\mathcal{I}w)}{Z_f} \mathd \mu (w),
  \label{equationintegral1}
\end{equation}
with the notations as in Section~\ref{sec:solutions}
(Theorem~\ref{theorem_weaksolution}) and Section~\ref{sec:dim-red}
(Theorem~\ref{theorem_reduction2}). In this section we prove
Theorem~\ref{theorem_main1}, i.e. the identity
\begin{equation}
  \Xi (h) = \int_{\mathbb{R}^n} h (y) \mathd \kappa (y) . \label{eq:main}
\end{equation}
It is important to note that $\Lambda_U$ appears without the modulus
in~{\eqref{equationintegral1}}.

\

Let us start by unfolding the definition of $\Lambda_U$ and $\Upsilon_f
(\mathcal{I}w)$ in~{\eqref{equationintegral1}} to get the expression

\begin{multline*}
  Z_f \Xi (h) = \int h (\mathcal{I}w (0)) \det_2 (I_{\mathcal{H}} + \nabla U
  (w)) \times\\
  \times \exp \left( - \delta (U) - \frac{1}{2} \|U\|^2_{\mathcal{H}} + 4 \int_{\mathbb{R}^2}
  V (\mathcal{I}w (x)) f' (x) \mathd x \right) \mathd \mu (w) .
\end{multline*}

In order to manipulate the regularized Fredholm determinant we approximate the
right hand side by

\begin{multline*}
  Z_f^{\chi} \Xi_{\chi} (h) \assign \int h (\mathcal{J}_{\chi} w (0)) \det_2
  (I_{\mathcal{H}} + \nabla U^{\chi}) \times\\
  \times \exp \left( - \delta (U^{\chi}) - \frac{1}{2} \|U^{\chi}
  \|^2_{\mathcal{H}} + 4 \int_{\mathbb{R}^2} V (\mathcal{J}_{\chi} w (x)) f' (x) \mathd x
  \right) \mathd \mu (w) .
\end{multline*}

where $\chi > 0$ is a regularization parameter, $\mathcal{J}_{\chi} \assign
\mathcal{I}^{1 + \chi} = (m^2 - \Delta)^{- 1 - \chi}$, $Z^{\chi}_f$ is the
normalization constant such that $\Xi_{\chi} (h) = 1$ and
\begin{equation}
  U^{\chi} (w) \assign \frac{1}{1 + 2 \chi} \mathcal{I}^{\chi} \partial V
  (\mathcal{J}_{\chi} w) . \label{eq:regularization}
\end{equation}
We will prove below that $\lim_{\chi \rightarrow 0} \Xi_{\chi} (h) = \Xi (h)$.
When $\chi > 0$, $\nabla U^{\chi} (w) = \frac{1}{1 + 2 \chi}
\mathcal{I}^{\chi} \partial V (\mathcal{J}_{\chi} w) \mathcal{J}_{\chi}$ is
almost surely a trace class operator and $U^{\chi} \in
\mathcal{W}^{\asterisk}$. This means that we can rewrite the regularized
Fredholm determinant $\det_2$ in term of the unregularized one (denoted by
$\text{det}$) (see equation~{\eqref{equationwienerspace3}} and the discussion
in Appendix~\ref{appendix_wienerspace}) obtaining
\begin{equation}
  \begin{array}{lll}
    Z_f^{\chi} \Xi_{\chi} (h) & = & \int h (\mathcal{J}_{\chi} w (0)) \det
    (I_{\mathcal{H}} + \nabla U^{\chi}) \times\\
    &  & \times \exp \left( - \langle U^{\chi}, w \rangle - \frac{1}{2}
    \|U^{\chi} \|^2_{\mathcal{H}} + 4 \int_{\mathbb{R}^2} V (\mathcal{J}_{\chi} w (x)) f' (x)
    \mathd x \right) \mathd \mu (w) .
  \end{array} \label{eq:heu1}
\end{equation}
The determinant is invariant with respect to conjugation and so we can
multiply $\nabla U^{\chi}$ by $(- \Delta + m^2)^{\chi}$ at the left
hand side and by $(- \Delta + m^2)^{-\chi}$ at the right hand side (this
last multiplication can be done since $\mathcal{I}^{\chi}=(-\Delta+m^2)^{-\chi}$ is a bounded
operator from $L^2(\mathbb{R}^2)$ into the Sobolev space $W^{2 \chi, 2} (\mathbb{R}^2)$ and $(- \Delta + m^2)^{\chi}$ is a bounded operator from  $W^{2 \chi, 2} (\mathbb{R}^2)$ into
$L^2 (\mathbb{R}^2)$). In this way we obtain
\begin{multline*}
  \det (I_{\mathcal{H}} + \nabla U^{\chi}) = \det (I_{\mathcal{H}} + \varpi
  \mathcal{I}^{\chi} f \partial^2 V (\mathcal{J}_{\chi} w) \mathcal{J}_{\chi})
  = \det (I_{\mathcal{H}} + \varpi f \partial^2 V (\mathcal{J}_{\chi} w)
  \mathcal{I}^{1 + 2 \chi}),
\end{multline*}
where $\varpi = \frac{1}{1 + 2 \chi}$, and featuring the operator $\varpi f
\partial^2 V (\mathcal{J}_{\chi} w) \mathcal{I}^{1 + 2 \chi}$. Let $\gamma$ be
the Gaussian measure given by the law of $\varphi =\mathcal{J}_{\chi} w \in
\widetilde{\mathcal{W}}$ under $\mu$. In other words the Gaussian measure
$\gamma$ is the one whose Fourier transform is
\[ \int_{\widetilde{\mathcal{W}}} \exp \left( i \int_{\mathbb{R}^2} k (x) \varphi (x) \mathd
   x \right) \mathd \gamma (\varphi) = \exp \left( - \frac{1}{2} \|
   \mathcal{J}_{\chi} (k) \|_{\mathcal{H}}^2 \right) . \]
The expression~{\eqref{eq:heu1}} is then equivalent to
\[ \int h (\varphi (0)) \det (I_{\mathcal{H}} + \varpi f \partial^2 V
   (\mathcal{J}_{\chi} w) \mathcal{I}^{1 + 2 \chi}) \exp (- \langle \varpi f
   \partial V (\varphi), (m^2 - \Delta) \varphi \rangle) \times \]
\[ \times \exp \left( - \frac{\varpi^2}{2} \|\mathcal{I}^{\chi} f \partial V
   (\varphi) \|^2_{\mathcal{H}} + 4 \int V (\varphi (x)) f' (x) \mathd x
   \right) \gamma (\mathd \varphi) . \]
At this point we introduce an auxiliary Gaussian field $\eta$ distributed as
the Gaussian white noise $\mu$ to write
\[ \exp \left( - \frac{\varpi^2}{2} \|\mathcal{I}^{\chi} f \partial V (\phi)
   \|^2_{\mathcal{H}} \right) = \int \exp (- i \varpi \langle f \partial V
   (\phi), \mathcal{I}^{\chi} \eta \rangle) \mu (\mathd \eta) . \]
We also introduce two fermionic fields $\psi, \bar{\psi}$ realized as bounded
operators on a suitable Hilbert space $\mathfrak{H}_{\psi, \bar{\psi}}$ with a
state $\tmop{Tr} (\rho \cdot) = \langle \cdot \rangle_{\psi \comma
\bar{\psi}}$ for which
\[ \{ \psi (x), \psi (x') \} = \{ \bar{\psi} (x), \bar{\psi} (x') \} = \{ \psi
   (x), \bar{\psi} (x') \} = 0 \]
\[ \langle \bar{\psi} (x) \bar{\psi} (x') \rangle_{\psi \comma \bar{\psi}} =
   \langle \psi (x) \psi (x') \rangle_{\psi \comma \bar{\psi}} = 0, \qquad
   \langle \psi (x) \bar{\psi} (x') \rangle_{\psi \comma \bar{\psi}} = \varpi
   \mathcal{G}_{1 + 2 \chi} (x - x'), \]
where $\{\cdot, \cdot\}$ is the anticommutator between bounded operators, i.e.
$\{K_1, K_2 \} = K_1 K_2 + K_2 K_1$ for any bounded operators defined on
$\mathfrak{H}_{\psi, \bar{\psi}}$, and $\mathcal{G}_{\alpha}$ is the kernel of
the operator $\mathcal{I}^{\alpha}$ (see Appendix~\ref{appendix_grasmannian}
for the definition of fermionic fields and Theorem~\ref{theorem_existence12}
for the existence of such fields). By Theorem~\ref{theorem_representation} and
Remark~\ref{remark_representation}, these additional fields allow to represent
the determinant as
\begin{multline*}
  \det (I_{\mathcal{H}} + \varpi f \partial^2 V (\mathcal{J}_{\chi} w)
  \mathcal{I}^{1 + 2 \chi}) = \left\langle \exp \left( \int \psi^i (x) f (x) \partial^2_{\phi^i \phi^j}
  V (\varphi (x)) \bar{\psi}^j (x) \mathd x \right) \right\rangle_{\psi,
  \bar{\psi}} .
\end{multline*}
By tensorizing the fermionic Hilbert space $\mathfrak{H}_{\psi, \bar{\psi}}$
with the $L^2$ space of the product measure $\gamma \otimes \mu$ one can
realize the fermionic fields $\psi, \bar{\psi}$ and the Gaussian fields
$\varphi, \eta$ as operators on the same Hilbert space $\mathfrak{H}$ with a
state which we denote by $\langle \cdot \rangle_{\chi}$ when this does not
cause any ambiguity. As a consequence, we have
\begin{equation}
  Z_f^{\chi} \Xi_{\chi} (h) = \langle h (\varphi (0)) \exp (Q_{\chi} (V, f))
  \rangle_{\chi}, \label{eq:exp-V-general}
\end{equation}
with an operator $Q_{\chi} (V, f)$ given by
\begin{eqnarray*}
  Q_{\chi} (V, f) & \assign & \int \psi (x) f (x) \partial^2 V (\varphi (x))
  \bar{\psi} (x) \mathd x +\\
  &  & - \varpi \langle f \partial V (\varphi), (m^2 - \Delta) \varphi +
  i\mathcal{I}^{\chi} \eta \rangle + 4 \int V (\varphi (x)) f' (x) \mathd x.
\end{eqnarray*}
The operator $Q$ satisfies the following important theorem.

\begin{theorem}
  \label{th:pol-eq}For all polynomials $p, P : \mathbb{R}^n \rightarrow
  \mathbb{R}$ and all $n \geqslant 0$ and all $\chi > 0$ we have
  \begin{equation}
    \langle p (\varphi (0)) (Q_{\chi} (P, f))^n \rangle_{\chi} = \langle p
    (\varphi (0)) (- 4 \pi P (\varphi (0)))^n \rangle_{\chi} .
    \label{eq:pol-eq}
  \end{equation}
\end{theorem}

This theorem is the key to our results and will be proved with the aid of
supersymmetry in Section \ref{section:supersymmetry}. Going back to
equation~{\eqref{eq:exp-V-general}} a possible strategy would be to expand the
exponential getting
\begin{equation}
  \begin{array}{lll}
    \langle h (\varphi (0)) \exp (Q_{\chi} (V, f)) \rangle_{\chi} & = &
    \sum_{n \geqslant 0} \frac{1}{n!} \langle h (\varphi (0)) (Q_{\chi} (V,
    f))^n \rangle_{\chi}
  \end{array} \label{eq:expansion}
\end{equation}
and then to use Theorem~\ref{th:pol-eq} to prove that each average on the
right hand side is equal to
\[ \langle h (\varphi (0)) (- 4 \pi V (\varphi (0)))^n \rangle_{\chi} . \]
Since
\[ \langle h (\varphi (0)) (- 4 \pi V (\varphi (0)))^n \rangle_{\chi} =
   Z_f^{\chi} \int_{\mathbb{R}^n} h (y) \mathd \kappa (y), \]
the equality~{\eqref{eq:main}} would be proved by taking the limit $\chi
\rightarrow 0$. Unfortunately equation~{\eqref{eq:expansion}} is not easy to
prove since the series on the right hand side of~{\eqref{eq:expansion}} does
not converge absolutely for a general $V$. For this reason we present below an
indirect proof of~{\eqref{eq:main}}. Given Theorem~\ref{th:pol-eq} we will
deduce \ Theorem~\ref{theorem_main1} from it via a sequence of successive
generalizations.
\begin{enumerate}
  \item First we consider potentials $V$ bounded and such that $\| \partial^2
  V \|_{\infty} < m^2 / 2$;
  
  \item then the class of $V$ satisfying Hypothesis~$V_{\lambda}$ and $C$;
  
  \item finally those $V$ satisfying only~$V_{\lambda}$.
\end{enumerate}

\subsection{Bounded potentials}

\begin{proposition}
  \label{th:equiv-pol}For all $V : \mathbb{R}^n \rightarrow \mathbb{R}$
  bounded such that $\| \partial^2 V \|_{\infty} < m^2 / 2$ and $h :
  \mathbb{R}^n \rightarrow \mathbb{R}$ bounded and measurable we have
  \begin{equation}
    \langle h (\varphi (0)) \exp (Q_{\chi} (V, f)) \rangle_{\chi} = \langle h
    (\varphi (0)) \exp (- 4 \pi V (\varphi (0))) \rangle_{\chi}
    \label{eq:exp-V-bounded}
  \end{equation}
  for $\chi > 0$ small enough. 
\end{proposition}

Let us introduce
\[ G_{\chi} (t) \assign \langle h (\varphi (0)) \exp (t Q_{\chi} (V, f))
   \rangle_{\chi}, \]
\[ H_{\chi} (t) \assign \langle h (\varphi (0)) \exp (- t 4 \pi V (\varphi
   (0))) \rangle_{\chi} \]
for $t \in [0, 1]$.

\begin{proof*}{Proof of Proposition~\ref{th:equiv-pol}}
  It is clear that $H_{\chi}$ is real analytic in $t \in [0, 1]$. By
  Lemma~\ref{lemma_chi2} below the function $G_{\chi} (t)$ is real analytic in
  $[- 1, 1]$. It is enough then to prove $\partial^n_t G_{\chi} (0) =
  \partial^n_t H_{\chi} (0)$ for any $n \in \mathbb{N}$. Now
  \[ \partial^n_t G_{\chi} (0) = \langle h (\varphi (0)) (Q_{\chi} (V, f))^n
     \rangle_{\chi}, \]
  \[ \partial^n_t H_{\chi} (0) = \langle h (\varphi (0)) (- 4 \pi V (\varphi
     (0))^n) \rangle_{\chi} . \]
  By the density of polynomials in the space of two-times differentiable
  functions with respect to the Malliavin derivative (see {\cite{Nualart2006}}
  Corollary~1.5.1) we can approximate both $\partial^n_t G_{\chi} (0)$ and
  $\partial^n_t H_{\chi} (0)$ with expressions of the form $\langle p (\varphi
  (0)) (Q_{\chi} (P, f))^n \rangle_{\chi}$ and $\langle p (\varphi (0)) (- 4
  \pi P (\varphi (0)))^n \rangle_{\chi}$ where $p, P$ are polynomials and
  therefore conclude from~{\eqref{eq:pol-eq}} that $\partial^n_t G_{\chi} (0)
  = \partial^n_t H_{\chi} (0)$ for all $n \geqslant 0$.
\end{proof*}

The following two lemmas prove the claimed analyticity of $G_{\chi}$.

\begin{lemma}
  \label{lemma_chi1}If $V$ is a bounded potential satisfying the
  Hypothesis~$C$, then $\exp (- t \delta (U^{\chi})) \in L^1 (\mu)$ for any $|
  t | \leqslant 1$ and $\chi=0$ and for $\chi > 0 $ small enough. Furthermore the integral $\int \exp
  (- t \delta (U^{\chi})) \mathd \mu$ is uniformly bounded for $\chi=0$ and for $\chi >0$ small enough, and $t$ in the compact subsets of $[- 1, 1]$.
\end{lemma}

\begin{proof}
  Under the Hypothesis~of the lemma we have that
  \[ \| \nabla U^{\chi} \| \leqslant \frac{\| \partial^2 V \|_{\infty}}{m^{2
     (1 + \chi)}}, \]
  where $\| \cdot \|$ is the usual operator norm on $\mathcal{L} (H)$.
  Proposition~B.8.1 of {\cite{Ustunel2000}} states that
  \[ \mathbb{E} \left[ \exp \left( - \frac{1}{2} \delta (K) \right) \right]
     \leqslant (\mathbb{E} [\exp (\| K \|^2_{\mathcal{H}})])^{\frac{1}{4}}
     \cdot \left( \mathbb{E} \left[ \exp \left( \frac{\| \nabla K \|^2_2}{(1 -
     \| \nabla K \|_{\mathcal{H}})} \right) \right] \right)^{\frac{1}{4}} \]
  whenever $K$ is a $H - C^1$ map such that $\| \nabla K \| < 1$. Taking $K =
  2 t U^{\chi}$ in the previous inequality we obtain the thesis.
\end{proof}

\begin{lemma}
  \label{lemma_chi2}The function $G_{\chi} (t)$ is real analytic in $[- 1, 1]$ for $\chi=0$ and for $\chi>0$ small enough.
\end{lemma}

\begin{proof}
  First of all we have that for any $t \in \mathbb{R}$ the map $r \rightarrow
  \det_2 (I + (t + r) \nabla U^{\chi}) = : D_t (r)$ is holomorphic in $r$
  (see~{\cite{Simon2005}} Theorem~9.3). By Cauchy theorem this means that
  \[ | \partial^n_t (\det_2 (I + t \nabla U^{\chi})) | \leqslant \frac{n!
     \sup_{\theta \in \mathbb{S}^1} | D_t (R e^{i \theta}) |}{R^n} . \]
  On the other hand we have for any $\chi \in [0, 1]$
  \[ | D_t (r) | \leqslant \exp \left( \frac{1}{2} \| (t + r) \nabla U^{\chi}
     \|^2_2 \right) \leqslant \exp (C (t^2 + | r |^2) \| \partial^2 V
     \|^2_{\infty}), \]
  where $C \in \mathbb{R}_+$ is some positive constant depending on $f$ but
  not on $V$. Thus we obtain
  \[ | \partial^n_t (\det_2 (I + t \nabla U^{\chi})) | \leqslant \frac{n! \exp
     (C (t^2 + | R |^2) \| \partial^2 V \|^2_{\infty})}{R^n} . \]
  With a similar reasoning we obtain a uniform bound of this kind for
  $\partial^n_t \exp \left( - \frac{1}{2} | t U^{\chi} |^2 \right) .$ Finally
  we note that
  \[ \mathbb{E} [\exp (- \delta ((t + r) U^{\chi}))] = \sum \frac{(- 1)^n
     r^n}{n!} \mathbb{E} [\exp (- \delta (t U^{\chi})) (\delta (U^{\chi}))^n]
     . \]
  By Lemma~\ref{lemma_chi1}, we note that
  \begin{eqnarray}
    | \mathbb{E} [\partial^n_t e^{- \delta ((t + r) U^{\chi})}] | & = & |
    \mathbb{E} [e^{- \delta ((t + r) U^{\chi})} (\delta (U^{\chi}))^n] |
    \nonumber\\
    & \leqslant & \frac{1}{\epsilon^n} \mathbb{E} [e^{- \delta ((t +
    \epsilon) U^{\chi})} e^{- \delta ((t - \epsilon) U^{\chi})}] < + \infty
    \nonumber
  \end{eqnarray}
  for any $| t | \leqslant 1$ and $0 < \epsilon < \frac{m^{2+2\chi}}{2 \| \partial^2 V
  \|_{\infty}} - | t | .$ Using the previous inequality it follows that
  $G_{\chi} (t)$ is real analytic in the required interval. 
\end{proof}

\begin{proposition}
  \label{proposition_chi}We have that $G_0 (t) = H_0 (t)$ for $t \in [- 1, 1]
  .$
\end{proposition}

\begin{proof}
  By Proposition \ref{th:equiv-pol}, we need only to prove that $G_{\chi} (t)
  \rightarrow G_0 (t)$ as $\chi \rightarrow 0$. Since $\det_2$, $\delta$, $|
  \cdot |_{\mathcal{H}}$ are continuous with respect to the natural norm of
  $\mathcal{H}$ and the Hilbert-Schmidt norm on $\mathcal{H} \otimes
  \mathcal{H}$ (see~{\cite{Simon2005}} Theorem~9.2 for the continuity of
  $\det_2$ and~{\cite{Nualart2006}} Proposition~1.5.4 for the continuity of
  $\delta$), and since $\exp (- \delta (t U^{\chi}))$ is bounded uniformly in
  $L^p$ (for $p$ small enough) we only \ have to prove that, for $\chi
  \rightarrow 0$, $U^{\chi} (w) \rightarrow U (w)$ in $\mathcal{H}$ and
  $\nabla U^{\chi} (w) \rightarrow \nabla U (w)$ in $\mathcal{H} \otimes
  \mathcal{H}$ for almost every $w \in \mathcal{W}$. We present only the proof
  of the second convergence, the proof of the first one being simpler and
  similar.
  
  We have that
  \[ \nabla U^{\chi} (w) [h] =\mathcal{I}^{\chi} (f \partial^2 V (\mathcal{J}_{\chi} w) \cdot
     \mathcal{J}_{\chi} h), \]
  thus proving the convergence of $\nabla U^{\chi} (w)$ in $\mathcal{H}
  \otimes \mathcal{H}$ is equivalent to proving the convergence of $(m^2 -
  \Delta)^{- 1 - \chi}$ to $(m^2 - \Delta)^{- 1}$ in $\mathcal{H} \otimes
  \mathcal{H}$ and the convergence of $f \partial^2 V (\mathcal{J}_{\chi} w)$
  to $f \partial^2 V (\mathcal{I} w)$ in $C^0 (\mathbb{R}^2)$. The first
  convergence follows from a direct computation using the Fourier transform of
  this operators. The second convergence follows from the fact that $V$ is
  smooth with bounded derivatives, $f$ decays exponentially at infinity and
  $\mathcal{J}_{\chi} w$ converges to $\mathcal{I} w$ pointwise and uniformly
  on compact sets since $(m^2 - \Delta)^{- \chi} \rightarrow \tmop{id}_{L^2}$,
  weakly as bounded operator on $L^2 (\mathbb{R}^2)$ and $(m^2 - \Delta)^{-
  1}$ is a compact operator from $L^2 (\mathbb{R}^2)$ into $C^0_{\tmop{loc}}
  (\mathbb{R}^2)$.
\end{proof}

\subsection{Potentials satisfying Hypothesis $V_{\lambda}$ and $C$}

\

Let $V_B$ denote a bounded smooth potential with all its derivatives bounded.
Introduce the following equation for $\phi_t = \bar{\phi}_t +\mathcal{I} \xi$:
\begin{equation}
  (m^2 - \Delta) \bar{\phi}_t + t f \partial V_B (\bar{\phi}_t +\mathcal{I}
  \xi) = 0 \label{equationanalytic1} .
\end{equation}
Denote by $\lambda_-$ the infimum on $y \in \mathbb{R}^n$ over the eigenvalues
of the $y$ dependent matrix $(\partial^2 V_B (y))$, and with $\lambda_+$ the
supremum on $y \in \mathbb{R}^n$ over the eigenvalues of the same matrix.

For $t \in \left( - \frac{m^2}{| \lambda_+ \wedge 0 |}, \frac{m^2}{|
\lambda_- \wedge 0 |} \right) $we have that equation
{\eqref{equationanalytic1}} has an unique solution that, by the Implicit
Function Theorem, is infinitely differentiable with respect to $t$ when $V_B
\in C^{\infty} (\mathbb{R}^n)$. Define the formal series
\begin{equation}
  S_t (r) \assign \sum_{k \geqslant 1} \frac{\sup_{x \in \mathbb{R}^2}  |
  \partial^k_t \bar{\phi}_t (x) |}{k!} r^k . \label{equationserie1}
\end{equation}

\begin{lemma}
  \label{lemma_serie1}Suppose that $V_B$ is a bounded real valued function with all
  derivatives bounded such that
  \[ \| \partial^k V_B \|_{\infty} \leqslant C^k k!, \]
  where the norm is the one induced by the identification of $\partial^n V_B$
  as a multilinear operator and for some $C \in \mathbb{R}_+$, then the $r$
  power series $S_t (r)$ is holomorphic for any $t \in \left( - \frac{m^2}{|
  \lambda_+ \wedge 0 |}, \frac{m^2}{| \lambda_- \wedge 0 |} \right) .$
  Furthermore the radius of convergence of $S_t (r)$ can be chosen uniformly
  for $t$ in compact subsets of $\left( - \frac{m^2}{| \lambda_+ \wedge 0 |},
  \frac{m^2}{| \lambda_- \wedge 0 |} \right) .$
\end{lemma}

\begin{proof}
  We define the following functions
  \[ \bar{V}^1 (r) \assign \sum_{k \geqslant 0} \frac{\| \partial^{k + 1} V_B
     \|_{\infty}}{k!} r^k, \qquad \bar{V}^2 (r) \assign \sum_{k \geqslant 0}
     \frac{\| \partial^{k+ 2} V_B \|_{\infty}}{k!} r^k . \]
  We have that the partial derivative $\partial^k_t \bar{\phi}_t$ solves the
  following equation
  \begin{eqnarray}
    (m^2 - \Delta) \partial^k_t \bar{\phi} + t \partial^2 V_B (\bar{\phi}_t)
    \cdot \partial_t^k \bar{\phi}_t & = & - \partial^{k - 1}_t (\partial V_B
    (\bar{\phi}_t) + t \partial^2 V_B (\bar{\phi}_t) \cdot \partial_t
    \bar{\phi}_t) + \nonumber\\
    &  & + t \partial^2 V_B (\bar{\phi}_t) \cdot \partial_t^k \bar{\phi}_t
    \nonumber
  \end{eqnarray}
  Using a reasoning similar to the one of Lemma~\ref{lemma_bound}, it is easy
  to prove that
  \[ \| \partial^k_t \bar{\phi}_t \|_{\infty} \leqslant \frac{\| -
     \partial^{k - 1}_t (\partial V_B (\bar{\phi}_t) + t \partial^2 V_B
     (\bar{\phi}_t) \cdot \partial_t \bar{\phi}_t) + t \partial^2 V_B
     (\bar{\phi}_t) \cdot \partial_t^k \bar{\phi}_t \|_{\infty}}{m^2 - | t |
     (\lambda_{\tmop{sign} (t)} \wedge 0)}, \]
  where it is important to note that the right hand side of the previous
  inequality depends only on the derivatives of order at most $k - 1$. The
  previous inequality and the method of majorants (see
  {\cite{Van2003majorants}}) of holomorphic functions permit to get the
  following differential inequality for $S_t (r)$
  \begin{equation}
    (m^2 - | t | (\lambda_{\tmop{sign} (t)} \wedge 0) - r \bar{V}^2 (S_t (r)))
    \partial_r (S_t) (r) \leqslant \bar{V}^1 (S_t (r)) .
    \label{equationmajorant}
  \end{equation}
  From the previous inequality we obtain that $S_t (r)$ is majorized by the
  holomorphic function $F_t (r)$ that is the solution of the differential
  equation {\eqref{equationmajorant}} (where the symbol $\leqslant$ is
  replaced by $=$) depending parametrically on $t$ with initial condition $F_t
  (0) = 0$. Since $F_t (r)$ is majorized by $F_k (r)$ or by $F_{- k} (r)$ if
  $| t | \leqslant k$ the thesis follows. 
\end{proof}

\begin{remark}
  \label{remark_serie1}An example of potential satisfying the hypotheses of
  Lemma~\ref{lemma_serie1} is given by the family of trigonometric polynomials
  in $\mathbb{R}^n$.
\end{remark}

\begin{lemma}
  \label{lemma_serie2}Under the hypotheses of Lemma~\ref{lemma_serie1} with $V
  = V_B$ and assuming that $h$ is an entire function we have that $G_0 (t) =
  H_0 (t)$ for any $t \in \left( - \frac{m^2}{| \lambda_+ \wedge 0 |},
  \frac{m^2}{| \lambda_- \wedge 0 |} \right)$. In other words the thesis of
  Theorem~\ref{theorem_main1} holds if $\lambda = 0$, $V_B$ satisfies
  Hypothesis~C as well as the hypotheses of Lemma~\ref{lemma_serie1}.
\end{lemma}

\begin{proof}
  By Proposition~\ref{proposition_chi} we need only to prove that $G_0$ is
  real analytic in the required set. By Corollary~\ref{corollary_uniqueness2}
  we have that
  \[ G_0 (t) =\mathbb{E} \left[ h (\mathcal{I} \xi (0) + \bar{\phi}_t (0))
     e^{4 \int t V_B (\mathcal{I} \xi (x) + \bar{\phi}_t (x)) f' (x) \mathd x}
     \right] . \]
  Then the thesis follows from Lemma~\ref{lemma_serie1} and the analyticity of
  $h$ and of the exponential.
\end{proof}

Let $V$ be a potential satisfying the Hypothesis~$V_{\lambda}$ then there
exist $V_B$ such that $V = V_B + \lambda V_U$ and we define
\[ V_{t, \lambda} = t V_B + \lambda V_U, \]
for any $t \in \mathbb{R}$. Denote by $U_{t, \lambda}$ the corresponding map
from $\mathcal{W}$ into $\mathcal{H}$. Let $h : \mathbb{R} \rightarrow
\mathbb{R}$ be a continuous bounded function. We write
\[ G_{0, \lambda} (t) \assign \int_{\mathcal{W}} h (\mathcal{I} w (0)) \det_2
   (I_{\mathcal{H}} + \nabla U_{t, \lambda}) \times \]
\[ \times \exp \left( - \delta (U_{t, \lambda}) - \frac{1}{2} \|U_{t, \lambda}
   \|^2_{\mathcal{H}} + 4 \int_{\mathbb{R}^2} V_{t, \lambda} (\mathcal{I} w
   (x)) f' (x) \mathd x \right) \mathd \mu \]
and
\[ H_{0, \lambda} (t) \assign Z_f \int_{\mathbb{R}^n} h (y) \exp \left( - 4
   \pi \left( m^2 \frac{| y |^2}{2} + t V_B (y) + \lambda V_U (y) \right)
   \right) \mathd y. \]
It is evident that the thesis of Theorem~\ref{theorem_main1} is equivalent to
prove that
\[ G_{0, \lambda} (t) = H_{0, \lambda} (t) \]
for any bounded potential $V_B$, any $h$ continuous and bounded and any $t \in
\left( - \frac{m^2}{| \lambda_+ \wedge 0 |}, \frac{m^2}{| \lambda_- \wedge 0
|} \right)$. This fact is the result of the next proposition.

\begin{proposition}
  \label{proposition_serie3}Under Hypothesis~$V_{\lambda}$ we have that $G_{0,
  \lambda} (t) = H_{0, \lambda} (t)$ for any $t \in \left( - \frac{m^2}{|
  \lambda_+ \wedge 0 |}, \frac{m^2}{| \lambda_- \wedge 0 |} \right)$. In other
  words the thesis of Theorem~\ref{theorem_main1} holds if $V$ satisfies also
  Hypothesis~C.
\end{proposition}

\begin{proof}
  By Lemma~\ref{lemma_serie2} we know that Theorem~\ref{theorem_main1} holds
  for any $\lambda = 0$ and for any bounded potential satisfying Hypothesis~C
  and the hypothesis of Lemma~\ref{lemma_serie1}. Thus if we are able to
  approximate any potential $V$ satisfying Hypothesis~$V_{\lambda}$ and
  Hypothesis~C by potentials of the form requested by Lemma~\ref{lemma_serie2}
  the thesis is proved.
  
  We can use the methods of the proof of Lemma~\ref{lemma_reduction3} for
  approximating a potential $V$ satisfying Hypothesis~$V_{\lambda}$ by a
  sequence of potentials $V_{B, N}$ satisfying the hypothesis of
  Lemma~\ref{lemma_serie1}. More in detail, using the notations of
  Lemma~\ref{lemma_reduction3}, we have that the sequence of functions $V^N$
  is composed by smooth, bounded functions and, if $V$ satisfies
  Hypothesis~$V_{\lambda}$, they are identically equal to $N$ outside a
  growing sequence of squares $Q_N \subset \mathbb{R}^2$. This means that
  $V^{N, p}$, which is the periodic extension of $V^N$ outside the square
  $Q_N$, is a smooth function for any $N \in \mathbb{N}$. Since $V^{N, p}$ is
  periodic it can be approximated with any precision we want by a
  trigonometric polynomial $P^N$. Furthermore since $V$ satisfies
  Hypothesis~C, also $V^{N, p}$ satisfies Hypothesis~C and we can choose the
  trigonometric polynomial $P^N$ satisfying Hypothesis~C too. In this way we
  construct a sequence of potentials $V_{B, N} = P^N$ satisfying the
  hypotheses of Lemma~\ref{lemma_serie1} and converging to $V$ uniformly on
  compact sets. Thus the thesis follows from Lemma~\ref{lemma_reduction1},
  Lemma~\ref{lemma_reduction2}, Corollary~\ref{corollary_uniqueness2} and the
  fact that the functions of the form $L (\mathcal{I} \xi (0) + \bar{\phi}_t
  (0))$, where $L$ is an entire function, are dense in the set of measurable
  functions in $\mathcal{I} \xi (0) + \bar{\phi}_t (0)$ with respect to the
  $L^p (\mu)$ norm.
\end{proof}

\subsection{Potentials satisfying only Hypothesis $V_{\lambda}$}

\begin{lemma}
  \label{theorem_det}Under the Hypothesis~$V_{\lambda}$ we have $\det_2 (I +
  \nabla U (w)) \in L^{\infty} (\mu)$.
\end{lemma}

\begin{proof}
  We follow the same reasoning proposed in~{\cite{Klein1984}} for polynomials.
  First of all, by the invariance property of the determinant with respect to
  conjugation, we have that
  \[ \det_2 (I + \nabla U (w)) = \det_2 (I + O (w)) \]
  where $O (w)$ is the selfadjoint operator given by
  \begin{equation}
    O^i_j (w) [h] = (m^2 - \Delta)^{- \frac{1}{2}} (f \partial_{\phi^i
    \phi^j}^2 V (\mathcal{I}w) \cdot (m^2 - \Delta)^{- \frac{1}{2}} h) .
    \label{equationdet1}
  \end{equation}
  Since $V$ satisfies the Hypothesis~QC the eigenvalues of the symmetric
  matrix $\partial^2 V (y)$ (where $y \in \mathbb{R}^n$) are bounded from
  below. Furthermore we can write the matrix $\partial^2 V (y)$ as the
  difference of two commuting matrices $\partial^2 V (y) = V_+ (y) - V_- (y)$
  where $V_+ (y), V_- (y)$ are symmetric, they have only eigenvalues greater
  or equal to zero and $\ker V_+ (y) \cap \ker V_- (y) = \{ 0 \}$. We denote
  by $O^+, O^-$ the two operators defined as $O$ in equation
  {\eqref{equationdet1}} replacing $\partial^2 V$ by $V_+$ and $V_-$
  respectively. Obviously $O^+$ and $O^-$ are positive definite and $O = O^+ -
  O^-$. By Lemma~3.3 {\cite{Klein1984}} we have that
  \[ | \det_2 (I + O (w)) | \leqslant \exp (2 \| O^- (w) \|^2_2) . \]
  Using a reasoning similar to the one of Proposition~\ref{proposition_C1H}
  and the fact that, under the Hypothesis~$V_{\lambda}$, the minimum
  eigenvalue $\lambda (y)$ of $\partial^2 V (y)$ has a finite infimum
  $\lambda_{_-}$ that is the same as the one for $V_-$ we obtain
  \[ | \det_2 (I + \nabla U (w)) | = | \det_2 (I + O (w)) | \leqslant \exp (C
     \lambda_0 \| f \|_2^2) \]
  for some positive constant $C$. In particular we have $\det_2 (I + \nabla U
  (w)) \in L^{\infty} .$
\end{proof}

In order to prove that $\exp (- \delta (U)) \in L^p$ we split $U$ into two
pieces. First of all if $\lambda (y)$ is the minimum eigenvalue of $\partial^2
V (y)$ we recall that $\lambda_- = \inf_{y \in \mathbb{R}^n} \lambda (y)$.
Moreover we shall set
\[ \bar{U} \assign U - (\lambda_- \wedge 0) f\mathcal{I} (w), \]
and $\hat{U} \assign U - \bar{U}$. We also set $W : = V + \frac{\lambda_-}{2}
| y |^2$. We introduce a useful approximation of $\bar{U} (w)$ for proving
Theorem~\ref{lemma_Lp}. Let $P_n$ the projection of an $L^2 (\mathbb{R}^2)$
function on the momenta less then $n$, i.e.
\[ P_n (h) = \int_{| k | < n} e^{i k \cdot x} \hat{h} (k) \mathd k, \]
where $\hat{h}$ is the Fourier transform of $h$ defined on $\mathbb{R}^2$. We
can uniquely extend the operator $P_n$ to all tempered distributions. In this
way we define $U_n (w)$ as
\begin{equation}
  U_n (w) \assign P_n [f \partial V (\mathcal{I}P_n w)]  \label{equationLp1}
\end{equation}
We shall denote by $\bar{U}_n$ the expression corresponding to
{\eqref{equationLp1}} where $V$ is replaced by $W$.

\begin{lemma}
  \label{lemma_inequality2}Under the Hypothesis~$V_{\lambda}$ there exist two
  positive constants $C, \alpha$ independent on $p \geqslant 2$ and $n \in
  \mathbb{N}$ such that
  \begin{equation}
    \mathbb{E} [| \delta (\bar{U}_n - \bar{U}) |^p] \leqslant C (p - 1)^{2 p}
    n^{- \alpha} . \label{equationLp2}
  \end{equation}
  Furthermore a similar bound holds also for $\mathbb{E} [| \| \nabla U_n
  \|^2_2 -\| \nabla U\|^2_2 |^p]$ and $\mathbb{E} [| \|
  \mathcal{I}w\|^2_{\mathcal{H}} -\|P_n (\mathcal{I}w) \|^2_{\mathcal{H}}
  |^p]$.
\end{lemma}

\begin{proof}
  First of all we write $\bar{U} = U_B + \bar{U}_U$ where $U_B = f \partial
  V_B (\mathcal{I}w)$, and we consider the corresponding decomposition for
  $\bar{U}_n$. If we prove that an inequality analogous to
  {\eqref{equationLp2}} holds for $U_B - U_{B, n}$ and $\bar{U}_U -
  \bar{U}_{U, n}$ separately then the inequality {\eqref{equationLp2}} holds.
  
  In order to prove the lemma we use the following inequality (proven in
  {\cite{Ustunel2000}} Proposition~B.8.1)
  \begin{equation}
    \begin{aligned}
      \mathbb{E} \left[ \cosh \left( \frac{\sqrt{\rho}}{2 \sqrt{2}} \delta (K)
      \right) \right] \leqslant & (\mathbb{E} [\exp (\rho \| K
      \|^2_{\mathcal{H}})])^{\frac{1}{4}} \times\\
      & \times \left( \mathbb{E} \left[ \exp \left( \frac{\rho}{1 - \rho c}
      \| \nabla K \|^2_2 \right) \right] \right)^{\frac{1}{4}}
    \end{aligned}
    \label{equationinequalitynew}
  \end{equation}
  that holds when $\| \nabla K \|^2_2 \in L^{\infty}$, $\| \nabla K \|
  \leqslant c < 1$ and $0 \leqslant \rho < \frac{1}{2 c^2}$. Putting $K =
  \bar{\epsilon} (U_B - U_{B, n})$ for $\bar{\epsilon}$ small enough, since
  $\| \nabla (U_{B, n} - U_B) \|^2_2, \| \nabla (U_{B, n} - U_B) \| \in
  L^{\infty}$ with a bound uniform in $n$, we have that
  \begin{equation}
    \begin{aligned}
      \mathbb{E} [\cosh (\epsilon \delta (U_B - U_{B, n}))] \leqslant &
      (\mathbb{E} [\exp (\epsilon' \| U_B - U_{B, n}
      \|^2_{\mathcal{H}})])^{\frac{1}{4}} \times\\
      & \times (\mathbb{E} [\exp (\epsilon' \| \nabla (U_B - U_{B, n})
      \|^2_2)]),
    \end{aligned}
    \label{equationLp3}
  \end{equation}
  for suitable $\epsilon, \epsilon' > 0$ and for all $n \in \mathbb{N}$. First
  of all we want to give a bound for the right hand side of
  {\eqref{equationLp3}} providing a precise convergence rate to the constant 1
  of the upper bound for the right hand side as $n \rightarrow + \infty$. We
  first note that
  \begin{equation}
    \mathbb{E} [\exp (\epsilon' \| U_B - U_{B, n} \|^2_{\mathcal{H}})] =
    \sum_{k = 1}^{\infty} \frac{\epsilon^{\prime k}}{k!} \mathbb{E} [\| U_B -
    U_{B, n} \|^{2 k}_{\mathcal{H}}] . \label{equationLp4}
  \end{equation}
  Using a reasoning like the one in the proof of
  Proposition~\ref{proposition_C1H} we have that
  \[ \| U_B - U_{B, n} \|^2_{\mathcal{H}} \lesssim \| \partial V_B
     \|^2_{\infty} \| Q_n (f) \|_{\mathcal{H}}^2 + \| \partial^2 V_B
     \|^2_{\infty} \int_{\mathbb{R}^2} (f (x) Q_n (\mathcal{I}w) (x))^2 \mathd
     x, \]
  where $Q_n = I - P_n$. From the previous inequality and the
  hypercontractivity of Gaussian random fields we obtain that
  \begin{eqnarray*}
    \mathbb{E} [\| U_B - U_{B, n} \|^{2 k}_{\mathcal{H}}] & \lesssim & k
    \left( \| Q_n (f) \|^{2 k}_{\mathcal{H}} + \int_{\mathbb{R}^2} f (x)^k
    \mathbb{E} [(Q_n (\mathcal{I}w) (x))^2] \mathd x \right)\\
    & \lesssim & k \| Q_n (f) \|^{2 k}_{\mathcal{H}} + k (2 k - 1)^k \| f^k
    \|_1 \mathbb{E} [(Q_n (\mathcal{I}w) (x))^2]^k,
  \end{eqnarray*}
  where the constants implied by the symbol $\lesssim$ do not depend on $k$.
  The right hand side converges then for $n \rightarrow + \infty$ to 1 as we
  have announced. Using the Fourier transform, the fact that $f$ is a Schwartz
  function, and the fact that $\mathcal{I}w$ is equivalent to a white noise
  transformed by the operator $(m^2 - \Delta)$ it is simple to obtain that $\|
  Q_n (f) \|^2, \mathbb{E} [(Q_n (\mathcal{I}w) (x))^2] \lesssim \frac{1}{n^2}
  .$ Then using the fact that $(2 k - 1)^k \lesssim C_1^k k!$ and inserting
  the previous inequality in equation {\eqref{equationLp4}} we obtain
  \[ \mathbb{E} [\exp (\epsilon' \| U_B - U_{B, n} \|^2)] \leqslant 1 + C_3
     \frac{\frac{\epsilon'}{n^2}}{\left( 1 - \frac{C_2 \epsilon'}{n^2}
     \right)^2}, \]
  that holds when $\epsilon' > 0$ is small enough and for two positive
  constants $C_2, C_3$. Using similar methods it is possible to prove a
  similar estimate for $\mathbb{E} [\exp (\epsilon' \| \nabla (U_B - U_{B, n})
  \|^2_2)]$. Inserting these estimates in the inequality
  {\eqref{equationLp3}}, we obtain
  \begin{equation}
    \mathbb{E} [\cosh (\epsilon \delta (U_B - U_{B, n}))] - 1 \lesssim
    \frac{\epsilon'}{n^2}, \label{equationLp5}
  \end{equation}
  where the constants implied by the symbol $\lesssim$ do not depend on $n$
  and on $\epsilon'$, when $\epsilon'$ is smaller than a suitable $\epsilon'_0
  > 0$. Using the inequality {\eqref{equationLp5}} we obtain that
  \begin{multline*}
    \sum_{k, n = 1}^{+ \infty} \frac{n^{1 / 2} \epsilon^{2 k}}{(2 k) !}
    \mathbb{E} [(\delta (U_B - U_{B, n}))^{2 k}] = \sum_{n = 1}^{+ \infty} n^{\frac{1}{2}} (\mathbb{E} [\cosh (\epsilon
    \delta (U_B - U_{B, n}))] - 1) \lesssim \sum_{n = 1}^{\infty}
    \frac{\epsilon'}{n^{\frac{3}{2}}} < + \infty .
  \end{multline*}
  Since the terms of an absolutely convergent series are bounded we obtain
  \[ \mathbb{E} [(\delta (U_B - U_{B, n}))^{2 k}] \lesssim \frac{(2 k)
     !}{\epsilon^{2 k} n^{\frac{1}{2}}} \lesssim (2 k - 1)^{4 k} n^{-
     \frac{1}{2}} . \]
  Using Young inequality we obtain that the inequality {\eqref{equationLp2}}
  holds for any $p \geqslant 2$. The estimate for $\delta (\bar{U}_U -
  \bar{U}_{U, n})$ follows from the fact that $\bar{U}_U$ is a polynomial of
  at most third degree and from hypercontractivity estimates for polynomial
  expressions of Gaussian random fields.
  
  The result for $\| \nabla U \|^2_2 - \| \nabla U_n \|^2_2$ can be proved
  using the same decomposition of $U$ and $U_n$ and following a similar
  reasoning. The result for $\mathbb{E} [| \| f\mathcal{I}w\|^2_{\mathcal{H}}
  -\|f P_n (\mathcal{I}w) \|^2_{\mathcal{H}} |^p]$ can be proved using
  hypercontractivity for polynomial expressions of Gaussian random fields. \ 
\end{proof}

In the following we write $c_n = \tmop{Tr} (P_n \circ \mathcal{I})$. It is
important to note that
\[ c_n = \int_{| x | < n} \frac{1}{| x |^2 + m^2} \mathd x \lesssim \log (n),
\]
where the integral is taken on the ball $| x | < n$ on $\mathbb{R}^2$.

\begin{lemma}
  \label{lemma_inequality3}There exists a $\lambda_0 > 0$ depending only on
  $f$ and $m^2$ such that for any $0 < \lambda < \lambda_0$ and $V$ satisfying
  the Hypothesis~$V_{\lambda}$ there exist some constants $\alpha, C_1, C_2 >
  0$ such that
  \[ \delta (\bar{U}_n) - R \int_{\mathbb{R}^2} f (P_n \mathcal{I}w)^2 \mathd
     x -\| \nabla U_n \|^2_2 \geqslant - C_1 - C_2 c_n^{\alpha} \]
  for any $R \in \mathbb{R}_+$.
\end{lemma}

\begin{proof}
  If $\tmop{Tr} (| \nabla K |) < + \infty$ and $K \in \mathcal{W}$ we have
  that $\delta (K) = \langle K, w \rangle_{\mathcal{H}} - \tmop{Tr} (\nabla
  K)$. Using this relation we obtain that
  
  \begin{multline*}
    \delta (\bar{U}_k) = \sum_{i = 1}^n \int_{\mathbb{R}^2} P_k (f
    \partial_{\phi^i} W (P_k \mathcal{I}w)) (x) w^i (x) \mathd x - \sum_{i = 1}^n \tmop{Tr}_{L^2} (P_k (f \partial_{\phi^i \phi^i}^2 W (P_k
    \mathcal{I}w) \cdot P_k (m^2 - \Delta)))
  \end{multline*}
  
  From this we obtain the lower bound
  \begin{eqnarray}
    \int_{\mathbb{R}^2} P_k (f \partial_{\phi^i} W (P_k \mathcal{I}w)) w^i
    \mathd x & = & \int_{\mathbb{R}^2} f \partial_{\phi^i} W (P_k
    \mathcal{I}w) (m^2 - \Delta) (P_k \mathcal{I}w^i) \mathd x \nonumber\\
    & = & \int_{\mathbb{R}^2} f \partial_{\phi^i} W (\mathcal{I}w_k) (m^2 -
    \Delta) (\mathcal{I}w_k^i) \mathd x \nonumber\\
    & = & \int_{\mathbb{R}^2} f \partial_{\phi^i \phi^r} W (\mathcal{I}w_k)
    \nabla \mathcal{I}w^i_k \cdot \nabla \mathcal{I}w^r_k \mathd x +
    \nonumber\\
    &  & + m^2 \int_{\mathbb{R}^2} f\mathcal{I}w^i_k \partial_{\phi^i} W
    (\mathcal{I}w_k) \mathd x + \nonumber\\
    &  & - \int_{\mathbb{R}^2} (\Delta f) W (\mathcal{I}w_k) \mathd x
    \nonumber\\
    & \geqslant & \int_{\mathbb{R}^2} f (m^2 \mathcal{I}w^i_k
    \partial_{\phi^i} W (\mathcal{I}w_k) - b^2 W (\mathcal{I}w_k)) \mathd x
    \nonumber
  \end{eqnarray}
  On the other hand we have
  \[ \tmop{Tr}_{L^2} (P_k (f \partial_{\phi^i \phi^i}^2 W (\mathcal{I}w_k)
     \cdot P_k (m^2 - \Delta))) = c_n \int_{\mathbb{R}^2} \partial_{\phi^i
     \phi^i}^2 W (\mathcal{I}w_k) f \mathd x \]
  \[ \leqslant \frac{c_n^p}{p} + \frac{1}{q} \int_{\mathbb{R}^2}
     (\partial_{\phi^i \phi^i}^2 W (\mathcal{I}w_k) (\mathcal{I}w_k))^q f
     \mathd x, \]
  where $\frac{1}{q} + \frac{1}{p} = 1$ and $q < 2$. Furthermore we have that
  
  \begin{multline*}
    \| \nabla U_k \|^2_2 \leqslant \int_{\mathbb{R}^2} \frac{1}{(| x |^2 +
    m^2)^2} \mathd x \int_{\mathbb{R}^2} (\partial_{\phi^i \phi^i}^2 V
    (\mathcal{I}w_k))^2 f \mathd x = \ell \int_{\mathbb{R}^2} (\partial_{\phi^i \phi^i}^2 V
    (\mathcal{I}w_k))^2 f \mathd x,
  \end{multline*}
  
  where $\ell = \int_{\mathbb{R}^2} \frac{1}{(| x |^2 + m^2)^2} \mathd x$.
  Using the previous inequality we obtain that
  \begin{eqnarray}
    &  & \delta (\bar{U}_n) - R \int_{\mathbb{R}^2} f | \mathcal{I}w_k |^2
    \mathd x - \| \nabla U_n \|^2_2 \nonumber\\
    &  & \geqslant - \frac{c_n^p}{p} + \int_{\mathbb{R}^2} f (m^2
    \mathcal{I}w^i_k \partial_{\phi^i} W (\mathcal{I}w_k) - b^2 W
    (\mathcal{I}w_k)) \mathd x + \nonumber\\
    &  & \qquad - \int_{\mathbb{R}^2} f \left( \frac{(\partial_{\phi^i
    \phi^i}^2 W (\mathcal{I}w_k))^q}{q} + \ell (\partial_{\phi^i \phi^i}^2 (V)
    (\mathcal{I}w_k))^2 + R | \mathcal{I}w_k |^2 \right) \mathd x \nonumber
  \end{eqnarray}
  It is simple to see that there exists a $\lambda_0 > 0$ (depending only on
  $b^2$ and $m^2$) such that for any potential $V$ satisfying the
  Hypothesis~$V_{\lambda}$ with $\lambda < \lambda_0$ the expression
  \begin{equation}
    m^2 y^i_k \partial_{\phi^i} W (y) - b^2 W (y) - \frac{(\partial_{\phi^i
    \phi^i}^2 W (y))^q}{q} - \ell (\partial_{\phi^i \phi^i}^2 V
    (\mathcal{I}w_k))^2 - R | y |^2 \label{eq:boundedbelow}
  \end{equation}
  is bounded from below and thus the thesis of the lemma holds.
\end{proof}

\begin{remark}
  \label{remark_hypotheses}Lemma and \ref{lemma_inequality2} Lemma
  \ref{lemma_inequality3} are the only places where Hypothesis~CO and
  Hypothesis~$V_{\lambda}$ are used in an essential way.
  
  Indeed we are able to obtain the estimate {\eqref{equationLp2}}, using the
  technique of the proof of Lemma \ref{lemma_inequality2}, only if $V$ is a
  sum of a bounded function and a polynomial. Furthermore we can obtain that
  the expression {\eqref{eq:boundedbelow}} is bounded from below, for
  $\lambda$ small enough and for any $R > 0$, only if the expression $y^i_k
  \partial_{\phi^i} W (y)$ is positive at infinity and it is able to
  compensate the growth of all the other terms in expression
  {\eqref{eq:boundedbelow}}.
  
  The previous conditions are satisfied only if $b^2 < 4 m^2$ and $V$ is a sum
  of a bounded function and a polynomial of fourth degree (not less because of
  the presence of $- R | y |^2$, and no more since in the other cases the
  growth at infinity of $\ell (\partial_{\phi^i \phi^i}^2 V
  (\mathcal{I}w_k))^2$ would have been strictly stronger than the growth at
  infinity of $y^i_k \partial_{\phi^i} W (y)$). This is the main reason for
  the restriction on $b^2$ in Hypothesis~CO and for the special form of $V$
  required by Hypothesis~$V_{\lambda}$. \ 
\end{remark}

\begin{lemma}
  \label{lemma_inequality4}Given a $p \in [1, + \infty) $there is a $R > 0$
  big enough such that
  \[ \exp \left( - \delta (\hat{U}) - R \int_{\mathbb{R}^2} f (x) |
     \mathcal{I}w (x) |^2 \mathd x \right) \in L^p (\mu) . \]
\end{lemma}

\begin{proof}
  This lemma is proven in {\cite{Klein1984}} Lemma~3.2. 
\end{proof}

\begin{lemma}
  \label{lemma_Lp}Suppose that $f$ satisfies the Hypotheses CO, then there
  exists $\lambda_0 > 0$ depending only on $f$ and $m^2$ such that for any
  $\lambda < \lambda_0$ and any $V$ satisfying the Hypothesis~$V_{\lambda}$ we
  have that
  \[ \exp (- \delta (U) + (1 + \| \nabla U \|^2_2)) \in L^p (\mu) \]
  for any $p \in [1, + \infty) .$
\end{lemma}

\begin{proof}
  The thesis follows from Lemma~\ref{theorem_det},
  Lemma~\ref{lemma_inequality2}, Lemma~\ref{lemma_inequality3} and
  Lemma~\ref{lemma_inequality4} using a standard reasoning due to Nelson (see
  Lemma~V.5 of {\cite{Simon1974}} or {\cite{Glimm1987}}) due to the fact that
  from the previous results it follows that there exist two constants $\alpha,
  \beta > 0$ independent on $N$ such that
  \[ \mu (\{ w \in \mathcal{W}| \delta (U^N) (w) \geqslant \beta (\log (N))
     \}) \leqslant e^{- N^{\alpha}} . \]
\end{proof}

\begin{proof*}{Proof of Theorem~\ref{theorem_main1}}
  By Proposition~\ref{proposition_serie3} in order to prove the theorem it
  remains only to prove that $G_{0, \lambda} (t)$ is real analytic for any $t
  \in \mathbb{R}$. The proof of this fact easily follows from
  Lemma~\ref{lemma_Lp} exploiting a reasoning similar to the one used in
  Lemma~\ref{lemma_chi2}.
\end{proof*}

\section{Supersymmetry}\label{section:supersymmetry}

At this point our main result is reduced to check the claim of
Theorem~\ref{th:pol-eq}, namely that for all polynomials $p, P : \mathbb{R}^n
\rightarrow \mathbb{R}$ and all $n \geqslant 0$ and all $\chi > 0$ we have the
equivalence
\begin{equation}
  \langle p (\varphi (0)) (Q_{\chi} (P, f))^n \rangle_{\chi} = \langle p
  (\varphi (0)) (- 4 \pi P (\varphi (0)))^n \rangle_{\chi} .
  \label{eq:pol-eq-bis}
\end{equation}
Since the expressions in the expectations are polynomials in the fields
$\varphi, \omega, \psi \comma \bar{\psi}$ which are ``free'', namely satisfy
either the bosonic or fermionic version of Wick's theorem (see, e.g.,
{\cite{Fetter2012}} Chapter~3 Section~8) the claim could be checked by
explicit computations. However this is still not trivial and a better
understanding of the structure of the required computations can be obtained
introducing a supersymmetric formulation involving the \tmtextit{superspace}
$\mathfrak{S}$ and the \tmtextit{superfield} $\Phi$. This new formulation
exposes a symmetry of the problem which is not obvious from the expressions we
obtained so far.

\

For an introduction to the mathematical formalism of supersymmetry see
e.g.~{\cite{dewitt_supermanifolds_1992,arai_supersymmetric_1993,rogers_supermanifolds_2007,de_goursac_noncommutative_2015}}.
The details of the rigorous implementation of the ideas exposed here is the
main goal of the paper of Klein et al.~{\cite{Klein1984}} and of the
modifications we implement here in order to overcome a gap in their proof.

\subsection{The superspace}

Formally the superspace $\mathfrak{S}$ can be thought as the set of points
$(x, \theta, \bar{\theta})$ where $x \in \mathbb{R}^2$ and $\theta,
\bar{\theta}$ are two additional anticommuting coordinates. A more concrete
construction is to understand $\mathfrak{S}$ via the algebra of smooth
functions on it.

Let $\mathfrak{G} (\theta_1, \ldots, \theta_n)$ be the (real) Grassmann
algebra generated by the symbols $\theta_1, \ldots, \theta_n$, i.e.
$\mathfrak{G} (\theta_1, \ldots, \theta_n) = \text{span} (1, \theta_i,
\theta_i \theta_j, \theta_i \theta_j \theta_k, \ldots, \theta_1 \cdots
\theta_n)$ with the relations $\theta_i \theta_j = - \theta_j \theta_i$.

A $C^{\infty}$ function $F : \mathbb{R}^2 \rightarrow \mathfrak{G} (\theta,
\bar{\theta})$ is just a quadruplet $(f_{\emptyset}, f_{\theta},
f_{\bar{\theta}}, f_{\theta \bar{\theta}}) \in (C^{\infty} (\mathbb{R}^2))^4$,
via the identification
\begin{equation}
  F (x) = f_{\emptyset} (x) + f_{\theta} (x) \theta + f_{\bar{\theta}} (x)
  \bar{\theta} + f_{\theta \bar{\theta}} (x) \theta \bar{\theta} .
  \label{eq:theta}
\end{equation}
The function $F$ can be considered as a function $F : \mathfrak{S} \rightarrow
\mathbb{R}$ by formally writing
\[ F (x, \theta, \bar{\theta}) = F (x) . \]
In particular we identify $C^{\infty} (\mathfrak{S})$ with $C^{\infty}
(\mathbb{R}^2 ; \mathfrak{G} (\theta, \bar{\theta}))$. $C^{\infty}
(\mathfrak{S})$ is a non-commutative algebra on which we can introduce a
linear functional defined by
\[ F \mapsto \int F (x, \theta, \bar{\theta}) \mathd x \mathd \theta \mathd
   \bar{\theta} \assign - \int_{\mathbb{R}^2} f_{\theta \bar{\theta}} (x)
   \mathd x, \]
where $f_{\theta \bar{\theta}} (x)$ as in equation {\eqref{eq:theta}}, induced
by the standard Berezin integral on $\mathfrak{S}$ satisfying
\[ \int \mathd \theta \mathd \bar{\theta} = \int \theta \mathd \theta \mathd
   \bar{\theta} = \int \bar{\theta} \mathd \theta \mathd \bar{\theta} = 0,
   \qquad \int \theta \bar{\theta} \mathd \theta \mathd \bar{\theta} = - 1. \]
\begin{remark}
  A norm on $C^{\infty} (\mathfrak{S})$ can be defined by
  \[ \| F \|_{C (\mathfrak{G})} = \sup_{x \in \mathbb{R}^2} (| f_{\emptyset}
     (x) | + | f_{\theta} (x) | + | f_{\bar{\theta}} (x) | + | f_{\theta
     \bar{\theta}} (x) |), \]
  and an involution by
  \[ \bar{F} (x, \theta, \bar{\theta}) = \overline{f_{\emptyset}} (x) +
     \overline{f_{\theta}} (x) \theta + \overline{f_{\bar{\theta}}} (x)
     \bar{\theta} + \overline{f_{\theta \bar{\theta}}} (x) \theta
     \bar{\theta}, \]
  where the bar on the right hand side denotes complex conjugation.
\end{remark}

Given $r \in C^1 (\mathbb{R}; \mathbb{R})$ we define the composition $r \circ
F : \mathfrak{S} \rightarrow \mathbb{R}$ by
\[ r (F (x, \theta, \bar{\theta})) \assign r (f_{\emptyset} (x)) + r'
   (f_{\emptyset} (x)) f_{\theta} (x) \theta + r' (f_{\emptyset} (x))
   f_{\bar{\theta}} (x) \bar{\theta} + r' (f_{\emptyset} (x)) f_{\theta
   \bar{\theta}} (x) \theta \bar{\theta}, \]
in accordance with the same expression one would get if $r$ were a monomial.
Moreover we can define similarly the space of Schwartz superfunctions
$\mathcal{S} (\mathfrak{S})$ and the Schwartz superdistributions $\mathcal{S}'
(\mathfrak{S}) =\mathcal{S}' (\mathbb{R}^2 ; \mathfrak{G} (\theta,
\bar{\theta}))$ where $T \in \mathcal{S}' (\mathfrak{S})$ can be written $T =
T_{\emptyset} + T_{\theta} \theta + T_{\bar{\theta}} \bar{\theta} + T_{\theta
\bar{\theta}} \theta \bar{\theta}$ with $T_{\emptyset}, T_{\theta},
T_{\bar{\theta}}, T_{\theta \bar{\theta}} \in \mathcal{S}' (\mathbb{R}^2)$ and
duality pairing
\[ T (f) = - T_{\emptyset} (f_{\theta \bar{\theta}}) + T_{\theta}
   (f_{\bar{\theta}}) - T_{\bar{\theta}} (f_{\theta}) - T_{\theta
   \bar{\theta}} (f_{\emptyset}), \qquad f_{\emptyset}, f_{\theta},
   f_{\bar{\theta}}, f_{\theta \bar{\theta}} \in \mathcal{S} (\mathbb{R}^2) .
\]

\subsection{The superfield}

We take the generators $\theta, \bar{\theta}$ to anticommute with the the
fermionic fields $\psi, \bar{\psi}$, and introduce the complex Gaussian field
\[ \omega \assign - \varpi ((m^2 - \Delta) \varphi + i \mathcal{I}^{\chi}
   \eta) \]
and put all our fields together in a single object defining the
\tmtextit{superfield}
\[ \Phi (x, \theta, \bar{\theta}) \assign \varphi (x) + \bar{\psi} (x) \theta
   + \psi (x) \bar{\theta} + \omega (x) \theta \bar{\theta}, \]
where $x \in \mathbb{R}^2$. We also define
\begin{eqnarray*}
  V (\Phi (x, \theta, \bar{\theta})) & = & V (\varphi (x)) + \partial V
  (\varphi (x)) (\bar{\psi} (x) \theta + \psi (x) \bar{\theta}) +\\
  &  & + [\partial V (\varphi (x)) \omega (x) + \partial^2 V (\varphi (x))
  \psi (x) \bar{\psi} (x)] \theta \bar{\theta}
\end{eqnarray*}
and since
\[ \tilde{f} (| x |^2 + 4 \theta \bar{\theta}) = \tilde{f} (| x |^2) + 4
   \tilde{f}' (| x |^2) \theta \bar{\theta}, \]
where $\tilde{f} : \mathbb{R}_+ \rightarrow \mathbb{R}$ is the smooth function
such that $f (x) = \tilde{f} (| x |^2)$ and $f' (x) = \tilde{f}' (| x |^2)$
(see Section \ref{section_introduction}), we observe that
\[ - \int V (\Phi (x, \theta, \bar{\theta})) \tilde{f} (| x |^2 + 4 \theta
   \bar{\theta}) \mathd x \mathd \theta \mathd \bar{\theta} = \int f (x)
   \partial V (\varphi (x)) \omega (x) \mathd x + \]
\[ + \int [f (x) \partial^2 V (\varphi (x)) \psi (x) \bar{\psi} (x) + 4 V
   (\varphi (x)) f' (x)] \mathd x = Q_{\chi} (V, f) . \]
By introducing the superspace distribution $\theta \bar{\theta} \delta_0
(\mathd x)$ we have also, by similar computations:
\[ p (\varphi (0)) = - \int p (\Phi (x, \theta, \bar{\theta})) \theta
   \bar{\theta} \delta_0 (\mathd x) \mathd \theta \mathd \bar{\theta} . \]
As a consequence we can rewrite $\langle p (\varphi (0)) (Q_{\chi} (P, f))^n
\rangle_{\chi}$ as an average over the superfield $\Phi$:
\begin{equation}
  \begin{aligned}
    \Xi_{\chi} (p) \assign & \langle p (\varphi (0)) (Q_{\chi} (P, f))^n
    \rangle_{\chi} =\\
    = & \left\langle \left( - \int p (\Phi (x, \theta, \bar{\theta})) \theta
    \bar{\theta} \delta_0 (\mathd x) \mathd \theta \mathd \bar{\theta} \right)
    \right.\left. \left( - \int P (\Phi (x, \theta, \bar{\theta}))
    \tilde{f} (| x |^2 + 4 \theta \bar{\theta}) \mathd x \mathd \theta \mathd
    \bar{\theta} \right)^n \right\rangle_{\chi}
  \end{aligned}
  \label{eq:super-form}
\end{equation}
While all these rewritings are essentially algebraic, the supersymmetric
formulation~{\eqref{eq:super-form}} makes appear a symmetry of the expression
for $\Xi_{\chi} (p)$ which was not clear from the original formulation. In
some sense the reader can think of the superspace $(x, \theta, \bar{\theta})$
and of the superfield $\Phi (x, \theta, \bar{\theta})$ as a convenient
bookkeeping procedure for a series of relations between the quantities one is
manipulating.

\

A crucial observation is that the superfield $\Phi$ is a free field with mean
zero, namely all its correlation functions can be expressed in terms of the
two-point function $\langle \Phi (x, \theta, \bar{\theta}) \Phi (x, \theta',
\bar{\theta}') \rangle_{\chi}$ via Wick's theorem. A direct computation of
this two point function gives:
\begin{align*}
  \langle \Phi (x, \theta, \bar{\theta}) \Phi (x, \theta', \bar{\theta}')
  \rangle_{\chi} = & \langle \varphi (x) \varphi (x') \rangle_{\chi} - \langle
  \bar{\psi} (x) \psi (x') \rangle_{\chi} \theta \bar{\theta}' - \langle \psi
  (x) \bar{\psi} (x') \rangle_{\chi} \bar{\theta} \theta'\\
  & + \langle \varphi (x) \omega (x') \rangle_{\chi} \theta' \bar{\theta}' +
  \langle \omega (x) \varphi (x') \rangle_{\chi} \theta \bar{\theta} +\\
  & + \langle \omega (x) \omega (x') \rangle_{\chi} \theta \bar{\theta}
  \theta' \bar{\theta}'\\
  = & \mathcal{G}_{2 + 2 \chi} (x - x') + \varpi \mathcal{G}_{1 + 2 \chi} (x -
  x') (\theta \bar{\theta}' - \bar{\theta} \theta') +\\
  & - \varpi (m^2 - \Delta) \mathcal{G}_{2 + 2 \chi} (x - x') (\theta'
  \bar{\theta}' + \theta \bar{\theta}) +\\
  & + ((m^2 - \Delta)^2 \mathcal{G}_{2 + 2 \chi} (x - x') -\mathcal{G}_{2
  \chi} (x - x')) \theta \bar{\theta} \theta' \bar{\theta}' .
\end{align*}
Upon observing that $(m^2 - \Delta) \mathcal{G}_{2 + 2 \chi} =\mathcal{G}_{1 +
2 \chi}$, $(m^2 - \Delta)^2 \mathcal{G}_{2 + 2 \chi} =\mathcal{G}_{2 \chi}$
and that \ $- \theta \bar{\theta}' + \bar{\theta} \theta' + \theta'
\bar{\theta}' + \theta \bar{\theta} = (\theta - \theta') (\bar{\theta} -
\bar{\theta}')$ we conclude
\begin{equation}
  \langle \Phi (x, \theta, \bar{\theta}) \Phi (x, \theta', \bar{\theta}')
  \rangle = C_{\Phi} (x - x', \theta - \theta', \bar{\theta} - \bar{\theta}')
  \label{eq:correlation}
\end{equation}
where
\[ C_{\Phi} (x, \theta, \bar{\theta}) \assign \mathcal{G}_{2 + 2 \chi} (x) -
   \varpi \mathcal{G}_{1 + 2 \chi} (x) \theta \bar{\theta} . \]
\begin{remark}
  Note that when $\chi = 0$, the superfield $\Phi$ corresponds to the formal
  functional integral
  \[ e^{- \frac{1}{2} \int [\Phi (m^2 - \Delta_S) \Phi] \mathd x \mathd \theta
     \mathd \bar{\theta}} \mathcal{D} \Phi \]
  where $\mathcal{D} \Phi =\mathcal{D} \psi \mathcal{D} \bar{\psi} \mathcal{D}
  \varphi \mathcal{D} \eta$ and where $\Delta_S = \Delta + \partial_{\theta}
  \partial_{\bar{\theta}}$ is the superlaplacian, where $\partial_{\theta},
  \partial_{\bar{\theta}}$ are the Grassmannian derivative such that
  $\partial_{\theta} (\theta) = \partial_{\bar{\theta}} (\bar{\theta}) = - 1$,
  $\partial_{\theta} (\bar{\theta}) = \partial_{\bar{\theta}} (\theta) = 0$,
  $\partial_{\theta} (\bar{\theta} \theta) = - \bar{\theta}$ and
  $\partial_{\bar{\theta}} (\bar{\theta} \theta) = \theta$ (see, e.g,
  {\cite{Wegner2016}} Chapter 20 or {\cite{Zinn1993}} Section 16.8.4).
  
  Then
  \[ \frac{1}{2} \int [\Phi (m^2 - \Delta_S) \Phi] \mathd x \mathd \theta
     \mathd \bar{\theta} = \frac{1}{2} \int [- 2 \bar{\psi} (m^2 - \Delta)
     \psi - \omega \omega + 2 \omega (m^2 - \Delta) \varphi] \mathd x \]
  
  \[ = \frac{1}{2} \int [- 2 \psi (m^2 - \Delta) \bar{\psi} + ((m^2 - \Delta)
     \varphi)^2 + \eta^2] \mathd x \]
  and this indeed corresponds to the action functional appearing in the formal
  functional integral for $(\psi, \bar{\psi}, \varphi, \eta)$. This is in
  agreement with the fact that the two point function satisfies the equation
  \[ (m^2 - \Delta_S) C_{\Phi} (x, \theta, \bar{\theta}) = \delta_0 (x) \delta
     (\theta) \delta (\bar{\theta}), \]
  where $\delta (x) \delta (\theta) \delta (\bar{\theta})$ is the distribution
  such that
  \[ \int F (x, \theta, \bar{\theta}) \delta_0 (x) \delta (\theta) \delta
     (\bar{\theta}) \mathd x \mathd \theta \mathd \bar{\theta} = f_{\emptyset}
     (0), \]
  namely, $C_{\Phi}$ is the Green's function for $(m^2 - \Delta_S)$.
\end{remark}

\subsection{The supersymmetry}

On $C^{\infty} (\mathfrak{S})$ one can introduce the (graded) derivations
\[ Q \assign 2 \theta \nabla + x \partial_{\bar{\theta}}, \qquad \bar{Q}
   \assign 2 \bar{\theta} \nabla - x \partial_{\theta}, \]
where $x \in \mathbb{R}^2$, $\nabla$ (and in the following also $\Delta =
\tmop{div} (\nabla \cdot)$) acts only on the space variables $x \in
\mathbb{R}^2$,which are such that
\[ Q (| x |^2 + 4 \theta \bar{\theta}) = \bar{Q} (| x |^2 + 4 \theta
   \bar{\theta}) = 0, \]
namely they annihilate the quadratic form $| x |^2 + 4 \theta \bar{\theta}$.
Moreover if $Q F = \bar{Q} F = 0$, for $F$ as in equation {\eqref{eq:theta}},
then we must have
\[ 0 = Q F (x, \theta, \bar{\theta}) = 2 \theta \nabla f_{\emptyset} (x) + x
   f_{\bar{\theta}} (x) + 2 \nabla f_{\bar{\theta}} (x) \theta \bar{\theta} -
   x f_{\theta \bar{\theta}} (x) \theta \]
\[ 0 = \bar{Q} F (x, \theta, \bar{\theta}) = 2 \bar{\theta} \nabla
   f_{\emptyset} (x) + x f_{\theta} (x) - 2 \nabla f_{\theta} (x) \theta
   \bar{\theta} + x f_{\theta \bar{\theta}} (x) \bar{\theta} \]
and therefore
\[ \nabla f_{\emptyset} (x) = \frac{x}{2} f_{\theta \bar{\theta}} (x) \qquad
   \text{and} \qquad f_{\theta} (x) = f_{\bar{\theta}} (x) = 0. \]
If we also request that $F$ is invariant with respect to $\mathbb{R}^2$
rotations in space, then there exists an $f$ such that $f (| x |^2) =
f_{\emptyset} (x)$ from which we deduce that $2 x f' (| x |^2) = \nabla f (| x
|^2) = \nabla f_{\emptyset} (x) = \frac{x}{2} f_{\theta \bar{\theta}} (x)$
which implies
\[ f (| x |^2 + 4 \theta \bar{\theta}) = f (| x |^2) + 4 f' (| x |^2) \theta
   \bar{\theta} = f_{\emptyset} (x) + f_{\theta \bar{\theta}} (x) \theta
   \bar{\theta} = F (x, \theta, \bar{\theta}) . \]
Namely any function satisfying these two equations can be written in the form
\[ F (x, \theta, \bar{\theta}) = f (| x |^2 + 4 \theta \bar{\theta}) . \]
Observe that if we introduce the linear transformations
\[ \tau (b, \bar{b}) \left(\begin{array}{c}
     x\\
     \theta\\
     \bar{\theta}
   \end{array}\right) = \left(\begin{array}{c}
     x + 2 \bar{b} \theta \rho + 2 b \bar{\theta} \rho\\
     \theta - (x \cdot b) \rho\\
     \bar{\theta} + (x \cdot \bar{b}) \rho
   \end{array}\right) \in \mathfrak{G} (\theta, \bar{\theta}, \rho) \]
for $b, \bar{b} \in \mathbb{R}^2$ and where $\rho$ is a new odd variable
anticommuting with $\theta, \bar{\theta}$ and itself, then we have
\[ \left. \frac{\mathd}{\mathd t} \right|_{t = 0} \tau (t b, t \bar{b}) F (x,
   \theta, \bar{\theta}) = \left. \frac{\mathd}{\mathd t} \right|_{t = 0} F
   (\tau (t b, t \bar{b}) (x, \theta, \bar{\theta})) = (b \cdot \bar{Q} +
   \bar{b} \cdot Q) F (x, \theta, \bar{\theta}) \]
so $\tau (b, \bar{b}) = \exp (b \cdot \bar{Q} + \bar{b} \cdot Q)$ and $\tau (t
b, t \bar{b}) \tau (s b, s \bar{b}) = \tau ((t + s) b, (t + s) \bar{b})$. In
particular $F \in C^{\infty} (\mathfrak{S})$ is supersymmetric if and only if
$F$ is invariant with respect to rotations in space and for any $b, \bar{b}
\in \mathbb{R}^2$ we have $\tau (b, \bar{b}) (F) = F$.

By duality the operators $Q, \bar{Q}$ and $\tau (b, \bar{b})$ also act  on the
space $\mathcal{S}' (\mathfrak{S})$ and we say that the distribution $T \in
\mathcal{S}' (\mathfrak{S})$ is supersymmetric if it is invariant with respect
to rotations in space and $Q (T) = \bar{Q} (T) = 0$. For supersymmetric
functions and distribution the following fundamental theorem holds.

\begin{theorem}
  \label{theorem_supersymmetry1}Let $F \in \mathcal{S} (\mathfrak{S})$ and $T
  \in \mathcal{S}' (\mathfrak{S})$ such that $T_0$ is a continuous function.
  If both $F$ and $T$ are supersymmetric, then we have the reduction formula
  \begin{equation}
    \int T (x, \theta, \bar{\theta}) \cdot F (x, \theta, \bar{\theta}) \mathd
    x \mathd \theta \mathd \bar{\theta} = 4 \pi T_{\emptyset} (0)
    F_{\emptyset} (0) . \label{eq:key}
  \end{equation}
\end{theorem}

\begin{proof}
  The proof can be found in~{\cite{Klein1984}}, Lemma~4.5 (see
  also~{\cite{Zaboronsky1997}} for a general proof on supermanifolds).
\end{proof}

Let us note that
\[ Q C_{\Phi} (x, \theta, \bar{\theta}) = \bar{Q} C_{\Phi} (x, \theta,
   \bar{\theta}) = 0, \]
indeed we can check that
\begin{align*}
  \nabla \mathcal{G}_{2 + 2 \chi} (x) = & \int_{\mathbb{R}^2} \frac{\mathd
  k}{(2 \pi)^2}  \frac{(i k) e^{i k \cdot x}}{(m^2 + | k |^2)^{2 + 2 \chi}}\\
  = & - \frac{i}{2 (1 + 2 \chi)} \int_{\mathbb{R}^2} \frac{\mathd k}{(2
  \pi)^2} e^{i k \cdot x} \nabla_k  \frac{1}{(m^2 + | k |^2)^{1 + 2 \chi}}\\
  = & \frac{i}{2 (1 + 2 \chi)} \int_{\mathbb{R}^2} \frac{\mathd k}{(2 \pi)^2} 
  \frac{(i x) e^{i k \cdot x}}{(m^2 + | k |^2)^{1 + \delta}} = - \frac{x}{2 (1
  + 2 \chi)} \mathcal{G}_{1 + 2 \chi} (x)\\
  = & - \frac{x \varpi}{2} \mathcal{G}_{1 + 2 \chi} (x)
\end{align*}
As a consequence expectation values of polynomials over the superfield $\Phi$
are invariant under the supersymmetry generated by any linear combinations of
$Q, \bar{Q}$.

\begin{remark}
  \label{remark_supersymmetry}The previous discussion implies that
  \begin{equation}
    \tau (b, \bar{b}) C_{\Phi} (x, \theta, \bar{\theta}) = C_{\Phi} (x,
    \theta, \bar{\theta}) . \label{eq:invariance}
  \end{equation}
  As a consequence, the superfield $\Phi' \assign \tau (b, \bar{b}) \Phi$ is a
  Gaussian free field and has the same correlation function $C_{\Phi'}$ as
  $\Phi$ given by equation~{\eqref{eq:correlation}}. However it is important
  to stress that this does not imply that $\Phi'$ has the same ``law'' as
  $\Phi$, namely that $\langle F (\Phi') \rangle = \langle F (\Phi) \rangle$
  for nice arbitrary functions. Indeed the correlation function given in
  equations~{\eqref{eq:correlation}} involves only the product $\langle \Phi
  (x, \theta, \bar{\theta}) \Phi (x, \theta', \bar{\theta}') \rangle$ of the
  complex superfield $\Phi$ and not also the product $\langle \Phi (x, \theta,
  \bar{\theta}) \bar{\Phi} (x, \theta', \bar{\theta}') \rangle$ of $\Phi$ with
  its complex conjugate $\bar{\Phi}$. The law of $\Phi$ would have been
  invariant with respect super transformations if and if only $\langle \Phi
  (x, \theta, \bar{\theta}) \Phi (x, \theta', \bar{\theta}') \rangle$ and
  $\langle \Phi (x, \theta, \bar{\theta}) \bar{\Phi} (x, \theta',
  \bar{\theta}') \rangle$ had been both supersymmetric. Unfortunately the
  function $\langle \Phi (x, \theta, \bar{\theta}) \bar{\Phi} (x, \theta',
  \bar{\theta}') \rangle$ is not invariant with respect to super
  transformations.
\end{remark}

\subsection{Expectation of supersymmetric polynomials}

As explained in Remark~\ref{remark_supersymmetry}, the law of $\Phi$ is not
supersymmetric. Nevertheless we can deduce important consequences from the
supersymmetry of the correlation function $C_{\Phi}$. More precisely, since
$\Phi$ is a free field Wick's theorem (see, e.g.,~{\cite{Fetter2012}}
Chapter~3 Section~8) hold and
\begin{eqnarray}
  & \begin{array}{c}
    \left\langle \prod_{i = 1}^{2 n} \Phi (x_i, \theta_i, \bar{\theta}_i)
    \right\rangle_{\chi} =\\
    = \sum_{\{ (i_k, j_k) \} _k} \prod_{k = 1}^n C_{\Phi} (x_{i_k} -
    x_{j_k}^{}, \theta_{i_k} - \theta_{j_k}, \bar{\theta}_{i_k} -
    \bar{\theta}_{j_k}),
  \end{array} &  \label{equationsupersymmetry3}\\
  & \left\langle \prod_{i = 1}^{2 n + 1} \Phi (x_i, \theta_i, \bar{\theta}_i)
  \right\rangle_{\chi} = 0. & 
\end{eqnarray}
By the supersymmetry of $C_{\Phi} (x - x', \theta - \bar{\theta}, \theta -
\bar{\theta}')$ and of its products, we obtain that
\[ \left\langle \prod_{i = 1}^{2 n} \tau (b, \bar{b}) (\Phi) (x_i, \theta_i,
   \bar{\theta}_i) \right\rangle_{\chi} = \left\langle \prod_{i = 1}^{2 n}
   \Phi (x_i, \theta_i, \bar{\theta}_i) \right\rangle_{\chi} . \]
The previous equality implies that
\begin{equation}
  \begin{aligned}
    \left\langle \prod_{i = 1}^n \int P_i (\Phi) \cdot \tau (b, \bar{b}) (F^i)
    \mathd x \mathd \theta \mathd \bar{\theta} \right\rangle_{\chi} = &
    \left\langle \prod_{i = 1}^n \int \tau (b, \bar{b}) (P_i (\Phi)) \cdot F^i
    \mathd x \mathd \theta \mathd \bar{\theta} \right\rangle_{\chi}\\
    = & \left\langle \prod_{i = 1}^n \int P_i (\tau (b, \bar{b}) (\Phi)) \cdot
    F^i \mathd x \mathd \theta \mathd \bar{\theta} \right\rangle_{\chi}\\
    = & \left\langle \prod_{i = 1}^n \int P_i (\Phi) \cdot F^i \mathd x \mathd
    \theta \mathd \bar{\theta} \right\rangle_{\chi},
  \end{aligned}
  \label{equationsupersymmetry5}
\end{equation}
where $P_1, \ldots, P_n$ are arbitrary polynomials and $F^1, \ldots, F^n$
arbitrary functions on superspace.

\begin{lemma}
  \label{lemma_supersymmetry1}Let $F^1, \ldots ., F^n \in \mathcal{S}
  (\mathfrak{S})$ be supersymmetric smooth functions and $P_1, \ldots, P_n$ be
  $n$ polynomials then
  \[ \left\langle \prod^n_{i = 1} \int P_i (\Phi) (x, \theta, \bar{\theta})
     \cdot F^i (x, \theta, \bar{\theta}) \mathd x \mathd \theta \mathd
     \bar{\theta} \right\rangle_{\chi} = (4 \pi)^n \left\langle \prod_{i =
     1}^n f^i_{\emptyset} (0) P_i (\phi (0)) \right\rangle_{\chi} . \]
\end{lemma}

\begin{proof}
  We define the distribution $\mathcal{H}^1 \in \mathcal{S}' (\mathfrak{G})$
  in the following way:
  
  \begin{multline*}
    \left. \mathcal{H}^1 (G) \assign \left\langle \int P_1 (\Phi) (x, \theta,
    \bar{\theta}) \cdot G (x, \theta, \bar{\theta}) \mathd x \mathd \theta
    \mathd \bar{\theta} \right. \right. \left. \left. \times \prod^n_{i = 2} \int P_i (\Phi) (x, \theta,
    \bar{\theta}) \cdot F^i (x, \theta, \bar{\theta}) \mathd x \mathd \theta
    \mathd \bar{\theta} \right\rangle_{\chi} \right.
  \end{multline*}
  
  for any $G \in \mathcal{S} (\mathfrak{G})$. Using the fact that $F^2,
  \ldots, F^n$ are supersymmetric and
  relation~{\eqref{equationsupersymmetry5}} we have that
  \begin{eqnarray}
    & \mathcal{H}^1 (\tau (b, \bar{b}) (G)) = \left\langle \int P_1 (\Phi)
    \cdot \tau (b, \bar{b}) (G) \mathd x \mathd \theta \mathd \bar{\theta}
    \prod^n_{i = 2} \int P_i (\Phi) \cdot F^i \mathd x \mathd \theta \mathd
    \bar{\theta} \right\rangle_{\chi} & _{} \nonumber\\
    & = \left\langle \int P_1 (\Phi) \cdot \tau (b, \bar{b}) (G) \mathd x
    \mathd \theta \mathd \bar{\theta} \prod^n_{i = 2} \int P_i (\Phi) \cdot
    \tau (b, \bar{b}) (F^i) \mathd x \mathd \theta \mathd \bar{\theta}
    \right\rangle_{\chi} = \mathcal{H}^1 (G) . &  \nonumber
  \end{eqnarray}
  This means that $\mathcal{H}^1$ is supersymmetric and since $F^1$ is also
  supersymmetric, by Theorem~\ref{theorem_supersymmetry1} we conclude
  \begin{multline*}
    \mathcal{H}^1 (F^1) = f^1_{\emptyset} (0) \mathcal{H}^1_0 (0) = (4 \pi) \left\langle f^1_{\emptyset} (0) P_i (\phi (0)) \prod^n_{i = 2}
    \int F^i \cdot P_i (\Phi) \mathd x \mathd \theta \mathd \bar{\theta}
    \right\rangle_{\chi} = \mathcal{H}^1 (K)
  \end{multline*}
  where $K \assign (4 \pi) \delta_0 (\mathd x) \theta \bar{\theta}$. Setting
  \begin{multline*}
    \mathcal{H}^2 (G) \assign \left. \left\langle \left( \int P_i (\Phi) K
    \mathd x \mathd \theta \mathd \bar{\theta} \right) \right. \right.  \left. \left. \times \left( \int P_i (\Phi) G \mathd x \mathd \theta
    \mathd \bar{\theta} \right) \prod^n_{i = 3} \int P_i (\Phi) F^i \mathd x
    \mathd \theta \mathd \bar{\theta} \right\rangle \right._{\chi}
  \end{multline*}
  and reasoning similarly we also conclude that $\mathcal{H}^2 (F^2)
  =\mathcal{H}^2 (V)$. Proceeding by transforming each subsequent factor, we
  can deduce that
  \begin{multline*}
    \left\langle \prod^n_{i = 1} \int P_i (\Phi) F^i \mathd x \mathd \theta
    \mathd \bar{\theta} \right\rangle_{\chi} = \left\langle \prod^n_{i = 1} \int P_i (\Phi) K \mathd x \mathd \theta
    \mathd \bar{\theta} \right\rangle_{\chi} = (4 \pi)^n \left\langle \prod_{i
    = 1}^n f^i_{\emptyset} (0) P_i (\phi (0)) \right\rangle_{\chi} .
  \end{multline*}
\end{proof}

\begin{proof*}{Proof of Theorem~\ref{th:pol-eq}}
  It is enough to use Lemma~\ref{lemma_supersymmetry1} with $P_1 = p$, $P_2 =
  \cdots = P_{n + 1} = P$, \ $F_1 = - \theta \bar{\theta} \delta_0 (x)$ and
  $F_2 = \cdots = F_{n + 1} = \tilde{f} (| x |^2 + 4 \theta \bar{\theta})$ to
  conclude.
\end{proof*}

\begin{remark}
  The dimensional reduction proof via supersymmetry is already present
  in~{\cite{Klein1984}} and indeed our result is analogous, under different
  hypotheses, to Theorem~II in~{\cite{Klein1984}}. The proofs of
  Lemma~\ref{theorem_det}, Lemma~\ref{lemma_inequality3} and
  Lemma~\ref{lemma_inequality4} above follows the same ideas of Lemma~3.1,
  Lemma~3.2 and Lemma~3.3 in~{\cite{Klein1984}}. We decided to propose a
  detailed proof of Theorem~\ref{theorem_main1} mainly for two reasons:
  \begin{enumeratenumeric}
    \item The first reason is that the hypotheses on the potential $V$ of
    Theorem~\ref{theorem_main1} and of Theorem~II in~{\cite{Klein1984}} are
    quite different. Indeed in~{\cite{Klein1984}} only polynomial potentials
    are considered while Hypothesis~$V_{\lambda}$ permits to consider
    polynomial of at most fourth degree perturbed by any bounded function. In
    order to prove the boundedness of $\Lambda_U$ in $L^p (\mu)$ under these
    different hypotheses \ we need to prove Lemma~\ref{lemma_inequality2}
    which is a trivial consequence of hypercontractivity when the potential
    $V$ is polynomial but is based on the non-trivial
    inequality~{\eqref{equationinequalitynew}} (proven
    in~{\cite{Ustunel2000}}) for general potentials $V$.
    
    \item The second main reason is the difference in the use of supersymmetry
    and of the supersymmetric representation of the
    integral~{\eqref{equationintegral1}}. Indeed, in our opinion there is a
    little gap in the proof of Theorem~III of {\cite{Klein1984}} that cannot
    be fixed without developing a longer proof. More precisely, in the proof
    of Theorem~III of {\cite{Klein1984}} it is tacitly assumed that the
    expression
    \[ \Psi (F) \assign \left\langle g (\varphi (0)) \exp \left( - \int V
       (\Phi) F \mathd \theta \mathd \bar{\theta} \mathd x \right)
       \right\rangle_{\chi}, \]
    is supersymmetric with respect to the function $F$, i.e. if $F$ is a
    smooth function in $\mathcal{S} (\mathfrak{S})$ and $\tau (b, \bar{b})$ is
    a supersymmetric transformation, then we have that $\Psi (\tau (b,
    \bar{b}) (F)) = G (F)$. In our opinion this fact is non-trivial since the
    law of $\Phi$ is not supersymmetric (see
    Remark~\ref{remark_supersymmetry}). What can be easily proven is only that
    the expressions
    \[ \Psi^n (F) \assign \left\langle g (\varphi (0)) \left( \int V (\Phi) F
       \mathd \theta \mathd \bar{\theta} \mathd x \right)^n
       \right\rangle_{\chi} \]
    are supersymmetric in $F$ (see Theorem~\ref{th:pol-eq} above). This fact
    alone does not easily imply that $\Psi (F)$ is supersymmetric. Indeed for
    the discussion in Section~\ref{sec:super}, we cannot guarantee that the
    series~{\eqref{eq:expansion}}, which is equivalent to $\Psi (F) = \sum_{n
    \geqslant 0} \frac{1}{n!} \Psi^n (F)$, converges absolutely when $V$
    growth at infinity at least as a polynomial of fourth degree (and we do
    not know under which conditions on $V$ and $F$ it converges relatively).
    In order to overcome this problem we propose a proof of
    Theorem~\ref{theorem_main1} which exploits only indirectly the
    supersymmetric representation of the integral {\eqref{equationintegral1}}
    in a way which permits to use only the supersymmetry of the expressions
    $\Psi^n (F)$ and avoiding the proof of the supersymmetry of the expression
    $\Psi (F)$.
  \end{enumeratenumeric}
\end{remark}

\section{Removal of the spatial cut-off}

\label{sec:removal}In this section we prove Theorem~\ref{theorem_cutoff1} on
the removal of the spatial cut-off in the setting of Hypothesis~C. It is
important to note that, differently from Theorem~\ref{theorem_reduction2}, we
explicitly require that the potential $V$ satisfies Hypothesis~C and not only
Hypothesis~QC. This is not due to problems in proving the existence of
solutions to equation {\eqref{equationcutoff1}} or in proving the convergence
of the cut-offed solution to the non-cut-offed one without the Hypothesis~C
(see Lemma~\ref{lemma_existencecutoff}). The main difficulty is instead to
prove the convergence of $\Upsilon_f (\phi) / Z_f$ to $1$. Indeed the previous
factor does not actually converge and what we can reliably expect is that
\begin{equation}
  \lim_{f \rightarrow 1} Z_f^{- 1} \mathbb{E} [\Upsilon_f (\phi_f) | \sigma
  (\phi_f (0))] \rightarrow 1, \label{eq:decorr}
\end{equation}
where hereafter $\phi_f$ denotes the solution to the
equation~{\eqref{equation2d1}} with cut-off $f$, i.e. $\Upsilon_f (\phi_f) /
Z_f$ becomes independent with respect to the $\sigma$-algebra generated by
$\phi_f (0)$.

To prove~{\eqref{eq:decorr}} directly is quite difficult due to the
non-linearity of the equation or equivalently to the presence of the
regularized Fredholm determinant in the
expressions~{\eqref{equationreduction5}} and~{\eqref{equationgaussian1}}
(which is a strongly non-local operator). For this reason we want to exploit a
reasoning similar to the one used in Section~\ref{sec:super}. With this aim we
introduce the equation
\begin{equation}
  (m^2 - \Delta) \phi_{f, t} + t f \partial V (\phi_{f, t}) = \xi
  \label{equationparameter}
\end{equation}
and the functions
\[ F^L_f (t) \assign Z_f^{- 1} \mathbb{E} \left[ L (\phi_{f, t} (0)) e^{4 t
   \int_{\mathbb{R}^2} f' (x) V (\phi_{f, t} (x)) \mathd x} \right], \]
where $t$ is taken such that $t \partial^2 V (y) + m^2$ is positive definite,
and $F^L (t) =\mathbb{E} [L (\phi_t (0))]$ (where $\phi_t$ is the solution to
{\eqref{equationparameter}} with $f \equiv 1$). By Lemma~\ref{lemma_serie1}
(whose proof does not use in any point the cut-off $f$) $F^L (t)$ is real
analytic whenever $V$ is a trigonometric polynomial, $t \partial^2 V (y) +
m^2$ is definite positive for any $y \in \mathbb{R}^n$ and $L$ is an entire
bounded function. Furthermore, by Theorem~\ref{theorem_reduction2}, $F^L_f (t)
= H^L (t)$ (where $H^L (t) = \int L (y) \mathd \kappa_t (y)$, see
Section~\ref{sec:super}) which is real analytic. Thus if we are able to prove
that $\lim_{f \rightarrow 1} \partial^n_t F^L_f (0) = \partial^n_t F^L (0)$ we
have that $H^L (t) = F^L (t)$ whenever $t \partial^2 V + m^2$ is definite
positive proving that Theorem~\ref{theorem_cutoff1} when $V$ is a
trigonometric polynomial satisfying Hypothesis~C. The idea, then, is to apply
a generalization of Lemma~\ref{lemma_reduction1},
Lemma~\ref{lemma_reduction2}, Lemma~\ref{lemma_reduction3} and the reasoning
in the proof of Proposition~\ref{proposition_serie3} and in the proof of
Theorem~\ref{theorem_reduction2} in order to obtain
Theorem~\ref{theorem_cutoff1}.

\begin{remark}
  Hypothesis~C is required in an essential way in the proof of the holomorphy
  of $F^L (t)$, in particular in Lemma~\ref{lemma_serie1}. The fact that the
  cutoff is removed does not allow to reason by approximation as we did in
  Theorem~\ref{theorem_main1}. 
\end{remark}

Since the proof is composed by many steps which are a straightforward
generalization of the results of the previous sections of the paper, we write
here only some details of the parts of the proof of
Theorem~\ref{theorem_cutoff1} which largely differ from what has been obtained
before.

Hereafter we denote by $\omega_{\beta} (x)$ the function
\[ \omega_{\beta} (x) \assign \exp (- \beta \sqrt{(1 + | x |^2)}) \]
and introduce the space $\mathcal{W}_{\beta}$ where $\beta > 0$ in the
following way
\[ \mathcal{W}_{\beta} \assign (- \Delta + 1) C^0_{\exp \beta} (\mathbb{R}^2 ;
   \mathbb{R}^n), \]
where $C^0_{\exp \beta}$ is the space of continuous function with respect to
the weighted $L^{\infty}$ norm
\[ \| g \|_{\infty, \exp \beta} \assign \sup_{x \in \mathbb{R}^2} |
   \omega_{\beta} (x) g (x) | . \]
The triple $(\mathcal{W}_{\beta}, \mathcal{H}, \mu)$ is an abstract Wiener
space. We introduce the obvious generalization of equation
{\eqref{equation2d2}}
\begin{equation}
  (m^2 - \Delta) \bar{\phi} + \partial V (\bar{\phi} +\mathcal{I} \xi) = 0,
  \label{equationcutoff2}
\end{equation}
where $\bar{\phi} = \phi -\mathcal{I} \xi .$

Now we want to prove a result that can replace Lemma~\ref{lemma_bound}.
Indeed Lemma~\ref{lemma_bound} plays a central role in the previous sections
of the paper, allowing to prove the existence of strong solutions to equation
{\eqref{equation2d1}}, the characterization of weak solutions in
Theorem~\ref{theorem_weaksolution2} and Theorem~\ref{theorem_weaksolution} and
finally allowing to show the convergence of weak solutions using the
convergence of potentials in Lemma~\ref{lemma_reduction1},
Lemma~\ref{lemma_reduction2}.

\begin{lemma}
  \label{lemma_existencecutoff}Suppose that $V$ satisfies the Hypothesis~QC
  and suppose that $\bar{\phi}$ is a classical solution to equation
  {\eqref{equationcutoff2}}, then there exists a $\beta_0$ depending only on
  $m^2$ such that, for any $\beta < \beta_0$
  \begin{equation}
    \| \bar{\phi} \|_{\infty, \exp \beta} \lesssim \| \exp (\alpha |
    \mathcal{I} \xi |) \|_{\infty, \exp \beta}, \label{equationweight1}
  \end{equation}
  where $\| \exp (\alpha | \mathcal{I} \xi |) \|_{\infty, \exp \beta}$ is
  almost surely finite and the constants implied by the symbol $\lesssim$
  depend only on $H$ and $m^2$. Furthermore for any $U$ open and bounded we
  have
  \begin{equation}
    \| \bar{\phi} \|_{\mathcal{C}^{2 - \tau} (U)} \lesssim \| \exp (\alpha p |
    \mathcal{I} \xi |) \|_{U_{\epsilon}, \infty} \exp (\alpha p \| \bar{\phi}
    \|_{\infty, \exp \beta} \| \omega_{\beta}^{- 1} \|_{U_{\epsilon}, \infty})
    \label{equationweight2}
  \end{equation}
  where $U_{\epsilon} \assign \{ x \in \mathbb{R}^2 | \exists y \in U, | y - x
  | \leqslant \epsilon \}$ and $\epsilon > 0$.
\end{lemma}

\begin{proof}
  The proof is very similar to the proof of Lemma~\ref{lemma_bound}. We report
  here only the passages having the main differences. For any $\epsilon > 0$
  there is a $\beta_{\epsilon} > 0$ and for any $\beta < \beta_{\epsilon}$ \
  we have
  \[ \left| \frac{\Delta (\omega_{2 \beta} (x))}{\omega_{2 \beta} (x)} -
     \frac{| \nabla \omega_{2 \beta} (x) |^2}{\omega_{4 \beta} (x)} \right| <
     \epsilon, \qquad x \in \mathbb{R}^2 . \]
  Without loss of generality (using the result of Lemma~\ref{lemma_bound}) we
  have that $\lim_{x \rightarrow \infty} | \bar{\phi} (x) |^2 \omega_{2 \beta}
  (x) = 0$ and so $x \mapsto | \bar{\phi} (x) |^2 \omega_{2 \beta} (x)$ has a
  positive maximum at $\bar{x} \in \mathbb{R}^2 .$ This means that $- \Delta
  (| \bar{\phi} |^2 \omega_{2 \beta}) (\bar{x}) \geqslant 0$ and $\nabla
  \bar{\phi} = - \frac{\bar{\phi}}{2 \omega_{2 \beta}} \nabla \omega_{2
  \beta}$ we have that
  \[ (m^2 - \epsilon) | \bar{\phi} (\bar{x}) |^2 \omega_{2 \beta} (\bar{x})
     \leqslant \frac{- \Delta (| \bar{\phi} |^2 \omega_{2 \beta})
     (\bar{x})}{2} + m^2 | \bar{\phi} (\bar{x}) |^2 \omega_{2 \beta} (\bar{x})
  \]
  \[ \leqslant - \omega_{2 \beta} (\bar{x}) (\bar{\phi} (\bar{x}) \cdot
     \partial V (\mathcal{I} \xi (\bar{x}) + \bar{\phi} (\bar{x}))) . \]
  Using a reasoning similar to the one of Lemma~\ref{lemma_bound} the thesis
  follows.
\end{proof}

Since the bounds~{\eqref{equationweight1}} and~{\eqref{equationweight2}} in
$C_{\exp \beta}^0$ and $\mathcal{C}^{2 - \tau}_{\tmop{loc}}$ imply the
compactness in $C^0_{\exp \beta'}$ when $\beta' < \beta$,
Lemma~\ref{lemma_existencecutoff} permits to prove the existence of strong
solutions to equation {\eqref{equationcutoff1}}, their uniqueness when $V$
satisfies Hypothesis~C and the generalization of Lemma~\ref{lemma_reduction1},
Lemma~\ref{lemma_reduction2}, Lemma~\ref{lemma_reduction3},
Proposition~\ref{proposition_serie3} and Theorem~\ref{theorem_reduction2}
needed in order to prove Theorem~\ref{theorem_cutoff1}.

\

At this point the proof of Theorem~\ref{theorem_cutoff1} requires only the
following additional statement.

\begin{theorem}
  \label{theorem_gaussian1}Let $V$ be a trigonometric polynomial, let $L$ be a
  polynomial and let $f_r$ be a sequence of cut-offs satisfying Hypothesis~CO,
  such that $f_r \equiv 1$ on the ball of radius $r \in \mathbb{N}$ and such
  that $f'_r (x) \leqslant C_1 \exp (- C_2 (| x | - r))$ for some positive
  constants $C_{1,} C_2 \in \mathbb{R}_+$ independent on $r$, then
  \[ \partial_t^k H^L (0) = \lim_{r \rightarrow + \infty} \partial^k_t
     F^L_{f_r} (0) = \partial^k_t F^L (0) . \]
\end{theorem}

To make the proof easy to follow we restrict ourselves to the scalar case,
i.e. the case where $n = 1$. The general case is a straightforward
generalization. We will also need certain results about iterated Gaussian
integrals. So let us introduce first some notations.

\

We denote by $\mathcal{T}$ the set of rooted trees with at least a external
vertex which is not the root. We denote by $\tau_0$ the tree with only one
vertex other than the root. In this set we introduce two operations: if $\tau
\in \mathcal{T}$ we denote by $[\tau]$ the tree obtained from $\tau$ by adding
a new vertex at the root which becomes the new root, and if $\tau' \in
\mathcal{T}$ we denote by $\tau \cdot \tau'$ the tree obtained by identifying
the root of $\tau$ and $\tau'$. It is possible to obtain any element of
$\mathcal{T}$ by applying iteratively a finite number of times the previous
operations to $\tau_0$. Furthermore we define $\mathcal{I}^f_{\tau} (x) \in
C^0 (\mathbb{R}^2)$ by induction in the following way
\[ \mathcal{I}^f_{\tau_0} (x) \assign \mathcal{I} \xi, \qquad
   \mathcal{I}^f_{[\tau]} (x) \assign \int_{\mathbb{R}^2} \mathcal{G} (x - y)
   f (y) \mathcal{I}^f_{\tau} (y) \mathd y, \]
\[ \mathcal{I}^f_{\tau \cdot \tau'} (x) \assign \mathcal{I}^f_{\tau} (x) \cdot
   \mathcal{I}^f_{\tau'} (x), \]
where $\mathcal{G} (x)$ is the Green function of the operator $\mathcal{I}=
(m^2 - \Delta)^{- 1}$. We need also to introduce the following notation.
Suppose that $\tau, \tau' \in \mathcal{T}$ and let $\mathcal{P}_{\tau, \tau'}$
be the set of all possible pairing between the external vertices (excepted
their roots) of the forest $\tau \sqcup \tau'$ and let
$\mathcal{P}^{\tmop{int}}_{\tau, \tau'} \subset \mathcal{P}_{\tau, \tau'}$ the
set of all possible pairing involving separately the vertices of $\tau$ and
$\tau'$. If $\pi \in \mathcal{P}$ we write
\[ \mathfrak{I}^{\pi, f}_{\tau, \tau'} (x, y) =\mathbb{E}
   [\widetilde{\mathcal{I}}^{\pi, f}_{\tau} (x) \cdot
   \widetilde{\mathcal{I}}^{\pi, f}_{\tau'} (y)], \]
where $\widetilde{\mathcal{I}}^{\pi, f}_{\tau} (x),
\widetilde{\mathcal{I}}^{\pi, f}_{\tau'} (y)$ are the expression
$\mathcal{I}^f_{\tau} (x)$ where $\xi$ is replaced by some copies of Gaussian
white noises $\xi_V$ one for each \ vertex $V$ of $\tau$ and $\tau'$ which
have correlation $0$ if $(V, V') \not{\in} \pi$ and are identically correlated
otherwise.

\begin{lemma}
  \label{lemma_gaussian1}With the notations and the hypotheses of
  Theorem~\ref{theorem_gaussian1} we have that for any $\tau, \tau' \in
  \mathcal{T}$
  
  \begin{multline*}
    \lim_{r \rightarrow + \infty} \left. \left( \mathbb{E} \left[
    \mathcal{I}^{f_r}_{\tau} (0) \cdot \prod_{i = 1}^p \int f'_r (x)
    \mathcal{I}^{f_r}_{\tau_i} (x) \mathd x \right] \right. \right.    \left. \left. -\mathbb{E} [\mathcal{I}^{f_r}_{\tau} (0)] \cdot \mathbb{E}
    \left[ \prod_{i = 1}^p \int f'_r (x) \mathcal{I}^{f_r}_{\tau_i} (x) \mathd
    x \right] \right) \right. = 0.
  \end{multline*}
\end{lemma}

\begin{proof}
  We present the proof only for the case $p = 1$, since the general case is a
  straightforward generalization. Since $\mathcal{I}^{f_r}_{\tau}$ are
  Gaussian random variables depending polynomially with respect to the white
  noise $\xi$, using the notation previously introduced we have
  
  \begin{multline*}
    \mathbb{E} \left[ \mathcal{I}^{f_r}_{\tau} (0) \cdot \int f'_r (x)
    \mathcal{I}^{f_r}_{\tau'} (x) \mathd x \right] -\mathbb{E}
    [\mathcal{I}^{f_r}_{\tau} (0)] \cdot \mathbb{E} \left[ \int f'_r (x)
    \mathcal{I}^{f_r}_{\tau'} (x) \mathd x \right] =\\
    = \sum_{\pi \in \mathcal{\mathcal{} \mathcal{P}}_{\tau, \tau'} \setminus
    \mathcal{P}^{\tmop{int}}_{\tau, \tau'}} \int_{\mathbb{R}^2}
    \mathfrak{I}^{\pi, f_r}_{\tau, \tau'} (0, x) f'_r (x) \mathd x.
  \end{multline*}
  
  Let us consider the simplest case when $\tau = \tau_k : = [\ldots [\tau_0]
  \ldots]$ $k$ times and $\tau' = \tau_{k'} = [\ldots [\tau_0] \ldots]$ $k'$
  times. In this case we have
  
  \begin{multline*}
    \mathfrak{I}^{\pi, f_r}_{\tau, \tau'} (0, x) = \int \mathcal{G} (0 - y_1)
    f_r (y_1) \ldots \mathcal{G} (y_k - x_1) \times\\
    \times \mathcal{G} (x_1 - x_2) f_r (x_2) \ldots f_r (x_{k'}) \mathcal{G}
    (x_{k'} - x) \mathd y_1 \cdots \mathd y_k \mathd x_1 \cdots \mathd x_k .
  \end{multline*}
  
  In particular, since $\mathcal{C} (x) = \mathcal{G} \ast \mathcal{G}$, which
  is the Green function of $\mathcal{I}^2 = (m^2 - \Delta)^{- 2}$, is bounded
  and positive, and since $\mathcal{G}$ is positive we obtain that
  \[ | \mathfrak{I}^{\pi, f_r}_{\tau, \tau'} (0, x) | \leqslant
     \underbrace{\mathcal{G} \asterisk \ldots \asterisk \mathcal{G}}_{k + k'
     \quad \tmop{times}} (0 - x) = \int_{\mathbb{R}^2} \frac{e^{- i l \cdot
     x}}{(| l |^2 + m^2)^{k + k'}} \mathd l. \]
  Thus we get
  \[ | \mathfrak{I}^{\pi, f_r}_{\tau, \tau'} (0, x) | \cdot (| x |^2 + 1)
     \leqslant \left| \int_{\mathbb{R}^2} (- \Delta_l + 1) \frac{e^{- i l
     \cdot x}}{(| l |^2 + m^2)^{k + k'}} \mathd l. \right| \leqslant C_3, \]
  where $C_3 \in \mathbb{R}_+$. Thus
  \begin{multline*}
    \int_{\mathbb{R}^2} \mathfrak{I}^{\pi, f_r}_{\tau, \tau'} (0, x) f'_r (x)
    \mathd x   \leqslant \int_{B_r^c} \frac{C_3}{(| x |^2 + 1)} C_1 \exp (- C_2 (| x | -
    r)) \mathd x \lesssim_{C_1, C_2, C_3} \frac{1}{r^2 + 1} \rightarrow 0.
  \end{multline*}
  For the general case let us note that $\mathfrak{I}^{\pi, f_r}_{\tau, \tau'}
  (0, x)$ is built by taking the product and the convolution with the
  functions $\mathcal{G}, f_r$ and $\mathcal{C} = \mathcal{G} \asterisk
  \mathcal{G}$. We note that $\mathcal{C}$ appears one time for every pair of
  vertices $(V_1, V_2) \in \pi$. Then, since $\pi \not{\in}
  \mathcal{P}^{\tmop{int}}_{\tau, \tau'}$ there is at least a couple $(V, V')
  \in \pi$ such that $V$ is a vertex of $\tau$ and $V'$ is a vertex of
  $\tau'$. Now we can bound the function $\mathcal{C}$ with a constant $C_4$
  for all pairs of vertices $(V_1, V_2) \not{=} (V, V')$ and $f_r$ by 1
  obtaining, for any $x \in \mathbb{R}^2$, that
  \[ \mathfrak{I}^{\pi, f_r}_{\tau, \tau'} (0, x) \lesssim C_4^{k_1}
     \mathfrak{I}^{f_r}_{\tau_{k_2}, \tau_{k 3}} (0, x) \]
  for some $k_1, k_2, k_3 \in \mathbb{N}$. The thesis follows from the
  previous inequality and the bounds obtained on
  $\mathfrak{I}^{f_r}_{\tau_{k_2}, \tau_{k 3}} (0, x) .$
\end{proof}

\begin{proof*}{Proof of Theorem~\ref{theorem_gaussian1}}
  We write
  \[ \mathfrak{L}_{f_r} (t) \assign L (\phi_{f_r, t} (0)) \qquad
     \mathfrak{E}_{f_r} (t) \assign \exp \left( 4 t \int_{\mathbb{R}^2} f'_r
     (x) V (\phi_{f_r, t} (x)) \mathd x \right) . \]
  We have
  \begin{eqnarray*}
    \partial^k_t F^L_{f_r} (t) & = & \sum_{0 \leqslant l \leqslant k}
    \binom{k}{l} \mathbb{E} \left[ \mathfrak{L}_{f_r}^{(k - l)} (0) \left.
    \partial^l_t \left( \frac{\mathfrak{E}_{f_r} (t)}{\mathbb{E}
    [\mathfrak{E}_{f_r} (t)]} \right) \right|_{t = 0} \right]\\
    & = & \mathbb{E} [\mathfrak{L}_{f_r}^{(k)} (0) \mathfrak{}] + \sum_{1
    \leqslant l \leqslant k} \sum_{0 \leqslant p \leqslant l - 1} \binom{k}{l}
    \binom{l}{p} (\mathbb{E}[\mathfrak{L}^{(k - l)}_{f_r} (0) \cdot
    \mathfrak{E}^{(l - p)}_{f_r} (0)] + \nobracket\\
    &  & \nobracket -\mathbb{E}[\mathfrak{L}^{(k - l)}_{f_r}
    (0)]\mathbb{E}[\mathfrak{E}^{(l - p)}_{f_r} (0)]) \cdot \left.
    \partial^p_t \left( \frac{1}{\mathbb{E} [\mathfrak{E}_{f_r} (t)]} \right)
    \right|_{t = 0},
  \end{eqnarray*}
  where we used the Leibniz rule for the derivative of the product and the
  relation
  \[ \left. \partial^l_t \left( \frac{1}{\mathbb{E} [\mathfrak{E}_{f_r} (t)]}
     \right) \right|_{t = 0} = - \sum_{0 \leqslant p \leqslant l - 1}
     \binom{l}{p} \mathbb{E} [\mathfrak{E}^{(l - p)}_{f_r} (0)] \cdot \left.
     \partial^p_t \left( \frac{1}{\mathbb{E} [\mathfrak{E}_{f_r} (t)]} \right)
     \right|_{t = 0} . \]
  Since $\left. \partial^p_t \left( \frac{1}{\mathbb{E} [\mathfrak{E}_{f_r}
  (t)]} \right) \right|_{t = 0}$ is bounded from above and below when $r
  \rightarrow + \infty$ if we are able to prove that $\mathbb{E}
  [\mathfrak{L}_{f_r}^{(k)} (0) \mathfrak{}] \rightarrow \partial^k_t F^L (0)$
  and $\mathbb{E} [\mathfrak{L}^{(k - l)}_{f_r} (0) \cdot \mathfrak{E}^{(l -
  p)}_{f_r} (0)] -\mathbb{E} [\mathfrak{L}^{(k - l)}_{f_r} (0)] \mathbb{E}
  [\mathfrak{E}^{(l - p)}_{f_r} (0)] \rightarrow 0$ the theorem is proven.
  
  First of all we note that
  \begin{equation}
    (m^2 - \Delta) \partial^k_t \phi_{f_r, t} |_{t = 0} = k f_r \partial^{k -
    1}_t (V (\phi_{f_r, t})) |_{t = 0} \label{equationcutoff4}
  \end{equation}
  for $k > 0$ and $\phi_{f_r, 0} =\mathcal{I} \xi$ for $k = 0$. This means
  that $\mathfrak{L}^{(k - l)}_{f_r} (0), \mathfrak{E}^{(l - p)}_{f_r} (0)$
  are given by a finite combination of convolutions and products between the
  function $\mathcal{G}$ (i.e. the Green function of $\mathcal{I}$), the
  functions of the form $V^{(l)} (\phi_{f_r, 0})$ (where $V^{(l)}$ is the
  $l$-th derivative of $V$), the cut-off $f_r$ and $f'_r$. Since $V$ is a
  trigonometric polynomial, by developing $V$ and its derivative by Taylor
  series, we obtain the following formal expressions
  \begin{equation}
    \begin{aligned}
      \mathfrak{L}^{(k)}_{f_r} (0) = & \sum_{\tau \in \mathcal{T}} A^k_{\tau}
      \mathcal{I}_{\tau}^{f_r} (0),\\
      \mathfrak{E}^{(k)}_{f_r} (0) = & \sum_l \sum_{\tau_1, \ldots, \tau_l \in
      \mathcal{T}} B^{k, l}_{(\tau_1, \ldots, \tau_l)} \prod_{i = 1}^l
      \int_{\mathbb{R}^2} f'_r (x) \mathcal{I}_{\tau_i}^{f_r} (x) \mathd x.
    \end{aligned}
    \label{equationcutoff5}
  \end{equation}
  The previous series are not only formal but they are actually absolutely
  convergent series. Furthermore we can change the integral, the expectation
  and the limit with the series.
  
  In order to prove this we now note that there exist two positive constants
  $C, \alpha > 0$ such that the function $V$ is majorized (in the meaning of
  the majorants method) by $C \exp (\alpha x)$ and let $\tilde{L}$ be the
  polynomial which majorizes the polynomial $L$. We now consider
  $\widetilde{\mathfrak{L}}_{f_r} (t) = \tilde{L} (\phi_{f_r, f} (0))$ and
  $\widetilde{\mathfrak{E}}_{f_r} (t) = \left( t C \int_{\mathbb{R}^2} f'_r
  \exp (\alpha \phi_{f_r, t} (x)) \mathd x \right)$. For what we said,
  $\widetilde{\mathfrak{L}}_{f_r}^{(k)} (0)$ and
  $\widetilde{\mathfrak{E}}_{f_r}^{(p)} (0)$ are a finite combination of
  convolutions and products between $\mathcal{G}$, the functions of the form
  $V^{(l)} (\phi_{f_r, 0})$ (where $V^{(l)}$ is the $l$-th derivative of $V$),
  the cut-off $f_r$ and $f'_r$. Let $\widehat{\mathfrak{L}}^k_{f_r}$ and
  $\widehat{\mathfrak{E}}^k_{f_r}$ be some random variables having the same
  expression of $\widetilde{\mathfrak{L}}_{f_r}^{(k)} (0)$ and
  $\widetilde{\mathfrak{E}}_{f_r}^{(p)} (0)$ where we replace every appearance
  of $V (\phi_{f_r, 0} (x))$ by $C \exp (\alpha | \phi_{f_r, 0} (x) |)$, every
  appearance of $V' (\phi_{f_r, 0} (x))$ with $C \alpha \exp (\alpha |
  \phi_{f_r, 0} (x) |)$ and so on. We introduce the following functions
  dependent on $\tau \in \mathcal{T}$ and defined recursively as follows
  \[ \mathcal{J}_{\tau_0}^{f_r} (x) \assign | \mathcal{I}_{\tau_0}^{f_r} (x) |
     \qquad \mathcal{J}_{[\tau]}^{f_r} (x) \assign \int_{\mathbb{R}^2}
     \mathcal{G} (x - y) f_r (y) \mathcal{J}_{\tau}^{f_r} (y) \mathd y \]
  \[ \mathcal{J}_{\tau \cdot \tau'}^{f_r} (x) \assign \mathcal{J}_{\tau}^{f_r}
     (x) \cdot \mathcal{J}^{f_r}_{\tau'} (x) . \]
  We, then, obtain that
  \[ \mathfrak{\hat{L}}^{(k)}_{f_r} = \sum_{\tau \in \mathcal{T}}
     \hat{A}^k_{\tau} \mathcal{J}_{\tau}^{f_r} (0) \qquad
     \mathfrak{\hat{E}}^{(k)}_{f_r} = \sum_l \sum_{\tau_1, \ldots, \tau_l \in
     \mathcal{T}} \hat{B}^{k, l}_{(\tau_1, \ldots, \tau_l)} \bigsqcap_{i =
     1}^l \int_{\mathbb{R}^2} f'_r (x) \mathcal{J}_{\tau_i}^{f_r} (x) \mathd
     x. \]
  By our construction we have that $\hat{A}^k_{\tau}, \hat{B}^{k, l}_{\tau,
  i}$ are all greater or equal than zero and also the following inequalities
  hold $| A^k_{\tau} | \leqslant \hat{A}^k_{\tau}, | B^{k, l}_{(\tau_1,
  \ldots, \tau_l)} | \leqslant \hat{B}^{k, l}_{(\tau_1, \ldots, \tau_l)}$.
  Furthermore we have $| \mathcal{I}_{\tau}^{f_r} (x) | \leqslant
  \mathcal{J}_{\tau}^{f_r} (x)$. Finally $\mathbb{E} [|
  \mathfrak{\hat{L}}^{(k)}_{f_r} |^p], \mathbb{E} [|
  \mathfrak{\hat{G}}^{(k)}_{f_r} |^p]$ are finite for any $p$, since the $x_1,
  \ldots, x_l$ function
  \[ \mathbb{E} \left[ \exp \left( \beta \alpha \sum_{i = 1}^l | \phi_{f_r,
     0} (x_i) | \right) \right] \leqslant + \infty, \]
  for any $\beta > 0$. Since $\mathcal{G}$ is positive the bounds on
  $\mathbb{E} [| \mathfrak{\hat{L}}^{(k)}_{f_r} |^p], \mathbb{E} [|
  \mathfrak{\hat{G}}^{(k)}_{f_r} |^p]$ can \ be chosen uniformly on $r$. This
  implies that the series {\eqref{equationcutoff5}} are absolutely convergent
  and by Lebesgue's dominated convergence theorem we can exchange the series
  with the summation and the limit. This means that
  \begin{eqnarray*}
    & \lim_{r \rightarrow + \infty} \mathbb{E} [\mathfrak{L}^{(k)}_{f_r} (0)
    \cdot \mathfrak{E}^{(l)}_{f_r} (0)] -\mathbb{E} [\mathfrak{L}^{(k)}_{f_r}
    (0)] \mathbb{E} [\mathfrak{E}^{(l)}_{f_r} (0)] = & \\
    & = \lim_{r \rightarrow + \infty} \sum_{l \in \mathbb{N}} \sum_{\tau,
    \tau_1, \ldots, \tau_l \in \mathcal{T}} A^k_{\tau} B^{k, l}_{(\tau_1,
    \ldots, \tau_l)} \left( \mathbb{E} \left[ \mathcal{I}^{f_r}_{\tau} (0)
    \cdot \prod_{i = 1}^l \int f'_r (x) \mathcal{I}^{f_r}_{\tau_i} (x) \mathd
    x \right] \right. + & \\
    & \left. -\mathbb{E} [\mathcal{I}^{f_r}_{\tau} (0)] \cdot \mathbb{E}
    \left[ \prod_{i = 1}^l \int f'_r (x) \mathcal{I}^{f_r}_{\tau_i} (x) \mathd
    x \right] \right) = & \\
    & = \sum_{l \in \mathbb{N}} \sum_{\tau, \tau_1, \ldots, \tau_l \in
    \mathcal{T}} A^k_{\tau} B^{k, l}_{(\tau_1, \ldots, \tau_l)} \lim_{r
    \rightarrow + \infty} \left( \mathbb{E} \left[ \mathcal{I}^{f_r}_{\tau}
    (0) \cdot \prod_{i = 1}^l \int f'_r (x) \mathcal{I}^{f_r}_{\tau_i} (x)
    \mathd x \right] + \right. & \\
    & \left. -\mathbb{E} [\mathcal{I}^{f_r}_{\tau} (0)] \cdot \mathbb{E}
    \left[ \prod_{i = 1}^l \int f'_r (x) \mathcal{I}^{f_r}_{\tau_i} (x) \mathd
    x \right] \right) = 0, & 
  \end{eqnarray*}
  where in the last line we used Lemma~\ref{lemma_gaussian1}. In a similar way
  it is simple to prove that
  \[ \mathbb{E} [\mathfrak{L}_{f_r}^{(k)} (0) \mathfrak{}] \rightarrow
     \partial^k_t F^L (0), \]
  and this concludes the proof.
\end{proof*}

\begin{proof*}{Proof of Theorem~\ref{theorem_cutoff1}}
  Using the bounds~{\eqref{equationweight1}} and~{\eqref{equationweight2}} we
  can prove the existence of strong solutions to equation
  {\eqref{equationcutoff1}}, and their uniqueness when $V$ satisfies
  Hypothesis~C.
  
  Furthermore using again the bounds~{\eqref{equationweight1}}
  and~{\eqref{equationweight2}} and a suitable generalization of
  Lemma~\ref{lemma_reduction1}, Lemma~\ref{lemma_reduction2},
  Lemma~\ref{lemma_reduction3} we can prove that Theorem~\ref{theorem_cutoff1}
  holds for any potential satisfying Hypothesis~C if and only if
  Theorem~\ref{theorem_cutoff1} holds for trigonometric potentials satisfying
  Hypothesis~C.
  
  \ The fact that Theorem~\ref{theorem_cutoff1} holds for trigonometric
  potentials, satisfying Hypothesis~C, is a consequence of
  Theorem~\ref{theorem_gaussian1}.
\end{proof*}

\appendix\section{Transformations in abstract Wiener
spaces}\label{appendix_wienerspace}

This appendix summarizes the results of~{\cite{Ustunel2000}} which are used in
the paper and establish the related notations. Hereafter we consider an
abstract Wiener space $(W, H, \mu)$ where $W$ is a separable Banach space, $H$
is an Hilbert space densely and continuously embedded in $W$ (with inclusion
map denoted by $i : H \rightarrow W$) called Cameron-Martin space and $\mu$ is
the Gaussian measure on $W$ associated with the Cameron-Martin space, i.e.
$\mu$ is the centered Gaussian measure on $W$ such that for any $w^{\asterisk}
\in W^{\asterisk}$ we have $\hat{\mu} (w^{\asterisk}) = \int_W \exp (i \langle
w^{\asterisk}, w \rangle) \mathd \mu (w) = \exp \left( - \frac{\|
i^{\asterisk} (w^{\asterisk}) \|_H^2}{2} \right)$ where $i^{\asterisk} :
W^{\asterisk} \rightarrow H$ is the dual operator of $i$.

If $u : W \rightarrow \mathbb{R}$ is a measurable non-linear functional we
denote by $\nabla u : W \rightarrow H^{}$ the following linear operator
\[ \nabla u (w) [h] = \langle \nabla u (w), h \rangle_H \assign
   \lim_{\epsilon \rightarrow 0} \frac{u (w + \epsilon h) - u (w)}{\epsilon} .
\]
The operator $\nabla$ is called Malliavin derivative and it is possible to
prove that it is closable (with unique closure) on the set of measurable $L^p
(\mu)$ functions. We denote the domain of $\nabla$ in $L^p (\mu)$ by
$\mathbb{D}_{p, 1}$. The previous operation can be extended for functional
$\mathfrak{u} : W \rightarrow H^{\otimes k}$ where $\nabla \mathfrak{u} : W
\rightarrow H^{\otimes k + 1}$ with its natural topology. Also this extension
of the operator $\nabla$ is closable.

If the measurable non-linear operator $F : W \rightarrow H$, where $| F |_H
\in L^p (\mu)$, is such that $\mathbb{E} [\langle F, \nabla u \rangle_H]
=\mathbb{E} [\tilde{F} u]$ for some $\tilde{F} \in L^p (\mu)$, we say that $F$
is in the domain of the operator $\delta$ and we denote by $\delta (F) =
\tilde{F} \in L^p (\mu)$ the Skorokhod integral of the measurable operator
$F$. The following expression for $\delta (F)$ used in the following holds:
suppose that $F (w) \in i^{\asterisk} (W^{\ast})$ and that $\nabla F (w)$ is a
trace class operator on $H$ for $\mu$ almost every $w \in W$ then
\begin{equation}
  \delta (F) (w) = \langle i^{\asterisk, - 1} (F (w)), w \rangle - \tmop{Tr}
  (\nabla F (w)) . \label{equationwienerspace1}
\end{equation}
We introduce a definition for studying the random transformations defined on
abstract Wiener spaces.

\begin{definition}
  Let $U : W \rightarrow H$ be a measurable map. We say that $U$ is a $H -
  C^1$ map if for $\mu$ almost every $w \in W$ the map $U_w : H \rightarrow
  H$, defined as $h \longmapsto U_w (h) \assign U (w + h)$, is a Fr{\'e}chet
  differentiable function in $H$ and if $\nabla U_w : H \rightarrow H^{\otimes
  2}$, defined as $h \longmapsto \nabla U_w (h) \assign \nabla U (w + h)$
  where $\nabla$ is the Malliavin derivative, is continuous for almost every
  $w \in W$ and with respect to the natural (Hilbert-Schmidt) topology of
  $H^{\otimes 2}$.
\end{definition}

We introduce the shift $T : W \rightarrow W$ associated with $U$, i.e. the map
defined as $T (w) = w + U (w)$, and the non-linear functional $\Lambda_U : W
\rightarrow \mathbb{R}$ as follows
\begin{equation}
  \Lambda_U (w) = \det_2 (I_H + \nabla U (w)) \exp \left( - \delta (U) (w) -
  \frac{1}{2} | U (w) |_H^2 \right), \label{equationwienerspace2}
\end{equation}
where $\det_2 (I_H + K)$ is the regularized Fredholm determinant
(see~{\cite{Simon2005}} Chapter 9) that it is well defined for any
Hilbert-Schmidt operator $K$. In particular if $K$ is self adjoint we have
\[ \det_2 (I + K) = \prod_{i \in \mathbb{N}} (1 + \lambda_i) e^{- \lambda_i},
\]
where $\lambda_i$ are the eigenvalues of the operator $K$.

Suppose that $U (w) \in i^{\asterisk} (W)$ and that $\nabla U (w)$ is a trace
class operator for almost any $w \in W$, then using the expression
{\eqref{equationwienerspace1}} and the properties of $\det_2$ we obtain
\begin{equation}
  \Lambda_U (w) = \det (I_H + \nabla U (w)) \exp \left( - \langle
  i^{\asterisk, - 1} (U (w)), w \rangle_W - \frac{1}{2} | U (w) |^2_H \right)
  \label{equationwienerspace3},
\end{equation}
where $\det (I_H + K)$ is the standard Fredholm determinant. The functional
$\Lambda_U$ is closely related to the transformation of the measure $\mu$ with
respect to the transformation $T$. Indeed suppose that $W$ is finite
dimensional then we have
\[ \mathd \mu = \exp \left( - \frac{1}{2} \langle w, w \rangle_H \right)
   \frac{\mathd x}{Z} = \exp \left( - \frac{1}{2} \langle i^{\asterisk, - 1}
   (w), w \rangle_W \right) \frac{\mathd x}{Z}, \]
where $Z \in \mathbb{R}_+$ is a suitable normalization constant and $\mathd x$
is the Lebesgue measure on $W$. Thus, if $T$ is a diffeomorphism on $W$, we
evidently have, thanks to equation {\eqref{equationwienerspace3}},

\begin{multline*}
  \frac{\mathd T_{\asterisk} (\mu)}{\mathd \mu} = \left. \left. \left| \det (I
  + \nabla U (w)) \exp \left( - \langle i^{\asterisk, - 1} (U (w)), w
  \rangle_W + \phantom{\frac{1}{2}} \right. \right. \right. \right.\\
  \left. \left. \left. \left. - \frac{1}{2} \langle i^{\asterisk, - 1} (U
  (w)), U (w) \rangle_W \right) \right| \right. \right. = | \Lambda_U (w) | .
\end{multline*}

The previous relation can be extended to the case where $W$ and $H$ are
infinite dimensional and the transformation $T$ is not a diffeomorphism but it
is only a $H - C^1$ map.

First of all we need the following generalization to the abstract Wiener
space context of the finite dimensional Sard Lemma.

\begin{proposition}
  \label{proposition_Sard}Let $T (w) = w + U (w)$ be a $H - C^1$ map and let
  $M \subset \mathcal{W}$ be the set of the zeros of $\det_2 (I + \nabla U
  (w))$, then the $\mu$ measure of the set $T (M)$ is zero, i.e. $\mu (T (M))
  = 0$.
\end{proposition}

\begin{proof}
  See Theorem~4.4.1 {\cite{Ustunel2000}}.
\end{proof}

The following is the change of variable theorem for (generally not invertible)
$H - C^1$ maps. \

\begin{theorem}
  \label{theorem_Lambda}Let $T (w) = w + U (w)$ be an $H - C^1$ map and let
  $f, g$ be two positive measurable functions then
  \begin{equation}
    \int_{\mathcal{W}} f \circ T (w) g (w) | \Lambda_U (w) | \mathd \mu (w) =
    \int_{\mathcal{W}} f (w) \left( \sum_{y \in T^{- 1} (w)} g (y) \right)
    \mathd \mu (w) . \label{equationwienerspace4}
  \end{equation}
  In particular if $K \subset W$ is a measurable subset where $\nobracket T
  |_K$ is invertible we
  \[ \int_K f \circ T (w) | \Lambda_U (w) | \mathd \mu (w) = \int_{T (K)} f
     (w) \mathd \mu (w) . \]
\end{theorem}

\begin{proof}
  See Theorem~4.4.1 {\cite{Ustunel2000}}.
\end{proof}

In order to prove Theorem~\ref{theorem_main1}, and so the relationship between
the weak solutions to equation~{\eqref{equation2d1}} and the integrals with
respect to the signed measure $\Lambda_U \mathd \mu$, it is not enough to
consider Theorem~\ref{theorem_Lambda} but we need a relationship analogous
to~{\eqref{equationwienerspace4}} with $| \Lambda_U |$ replaced by
$\Lambda_U$. In order to achieve this result we need some more hypotheses on
the map $U$:
\begin{description}
  \item[Hypothesis~DEG1] The map $U : W \rightarrow H \hookrightarrow W$ is a
  Fr{\'e}chet differentiable map from $W$ into itself and furthermore it is
  $C^1$ with respect to the usual topology of $W$;
  
  \item[Hypothesis~DEG2] The map $T$ is proper (i.e. inverse images of compact
  subsets are compact) and the equation $T^{- 1} (y) = w$ has a finite number
  of solution $y$ for $\mu$ almost every $w \in W$. 
\end{description}
Under the Hypothesis~DEG1 and DEG2 we can define the following number
\[ \tmop{DEG} (w, T) \assign \sum_{y \in T^{- 1} (w)} \tmop{sign} (\det_2 (I_W
   + \nabla U (y))) . \]
This index is a suitable modification of the Leray-Schauder degree of a
Fredholm non-linear operator described, for example, in {\cite{Berger1977}}
Section 5.3C, where the following definition is given: if $B$ is a bounded set
of $W$ such that $T^{- 1} (w) \cap \partial B = \emptyset$ and $\nabla T (y)
\not{=} 0$ for $y \in T^{- 1} (w) \cap B$ we have
\[ \tmop{DEG}_B (w, T) = \sum_{y \in T^{- 1} (w) \cap B} (- 1)^{\left(
   \text{number of negative eigenvalues of $\nabla T (y)$} \right)} . \]
It is evident that under the Hypothesis~DEG2 and, as a consequence of
Proposition~\ref{proposition_Sard}, we have
\[ \lim_{B \rightarrow W} \tmop{DEG}_B (w, T) = \tmop{DEG} (w, T) \]
for almost all $w \in W$.

\begin{theorem}
  \label{theorem_wienerspace3}Under the Hypotheses DEG1 and DEG2 we have that
  $\tmop{DEG} (w, T)$ is $\mu$ almost surely equal to the constant $\tmop{DEG}
  (T) \in \mathbb{Z}$ independent of $w$ and for any bounded function $f$ such
  that $f \circ T \cdot \Lambda_U \in L^1 (\mu)$ we have
  \[ \int_W f \circ T (w) \Lambda_U (w) \mathd \mu (w) = \tmop{DEG} (T) \cdot
     \int_W f (w) \mathd \mu (w) . \]
\end{theorem}

\begin{proof}
  The proof can be found in~{\cite{Ustunel2000}} Theorem~9.4.1 and
  Theorem~9.4.6.
\end{proof}

In general is not simple to compute $\tmop{DEG} (T)$ but this computation
simplified under the following Hypothesis:
\begin{description}
  \item[Hypothesis~DEG3] The map $T_{\epsilon} (w) = w + \epsilon U (w)$ has
  bounded level set uniformly in $\epsilon \in [0, 1]$, i.e. if $B \subset W$
  is bounded $\bigcup_{\epsilon \in [0, 1]} T^{- 1}_{\epsilon} (B)$ is a
  bounded set in $W$.
\end{description}
\begin{theorem}
  \label{theorem_wienerspace4}Under the Hypotheses DEG1, DEG2 and DEG3 we have
  that, for any $\epsilon \in [0, 1]$:
  \[ \tmop{DEG} (T) = \tmop{DEG} (w, T) = \tmop{DEG} (w, T_{\epsilon}) = 1. \]
\end{theorem}

\begin{proof}
  The proposition follows from the invariance of $\tmop{DEG}_B$ under
  homotopies of the operator $T$. In other words for any $B$ such that $T^{-
  1}_{\epsilon} (w) \cap \partial B = \emptyset$ we have $\tmop{DEG}_B (w,
  T_{\epsilon}) = \tmop{DEG}_B (w, T)$. Under the Hypothesis~DEG3 we can
  choose $B$ big enough such that $\tmop{DEG}_B (w, T_{\epsilon}) = \tmop{DEG}
  (w, T_{\epsilon})$ for any~$\epsilon \in [0, 1]$. Since $\tmop{DEG} (w, T_0)
  = \tmop{DEG} (w, \tmop{id}_W) = 1$ the thesis follows.
\end{proof}

\section{Fermionic fields }\label{appendix_grasmannian}

In this appendix we introduce the notion of fermionic fields used in Section
\ref{sec:super} and Section \ref{section:supersymmetry}.

\

We consider a quantum probability space $(\mathfrak{H}, \rho)$, where
$\mathfrak{H}$ is a separable Hilbert space and $\rho$ is a positive trace
class operator. If $K \in \mathcal{B} (\mathfrak{H})$ (where $\mathcal{B}
(\mathfrak{H})$ is the Hilbert space of bounded operators defined on
$\mathfrak{H}$) we define $\langle K \rangle = \tmop{Tr} (K \cdot \rho)$.

Let $H$ be a Hilbert space, we consider two continuous linear maps $\psi,
\bar{\psi} : H \rightarrow \mathcal{B} (\mathfrak{H})$ such that for any $f_1,
f_2 \in H$ we have
\[ \{ \psi (f_1), \psi (f_2) \} = \{ \bar{\psi} (f_1), \bar{\psi} (f_2) \} =
   \{ \psi (f_1), \bar{\psi} (f_2) \} = 0 \]
where $\{ K_1, K_2 \} = K_1 \cdot K_2 + K_2 \cdot K_1$ is the anticommutator
of the operators $K_1, K_2 \in \mathcal{B} (\mathfrak{H})$.

\begin{definition}
  \label{def:psi}Using the previous notations, the two antisymmetric fields
  $\psi, \bar{\psi} : H \rightarrow \mathcal{B} (\mathfrak{H})$ are called
  fermionic fields associated with the Hilbert space $H$ if we have that
  \begin{equation}
    \langle \bar{\psi} (f_1) \psi (g_1) \ldots \bar{\psi} (f_n) \psi (g_n)
    \rangle = \det (\langle f_i, g_j \rangle) . \label{eq:det}
  \end{equation}
\end{definition}

The following theorem ensure the existence of a pair of fermionic fields for
each separable Hilbert space $H$.

\begin{theorem}
  \label{theorem_existence12}For any separable Hilbert space $H$ there exists
  a quantum probability space $(\mathfrak{H}, \rho)$ and two continuous linear
  maps $\psi, \bar{\psi} : H \rightarrow \mathcal{B} (\mathfrak{H})$ such that
  $\psi, \bar{\psi}$ are a pair of fermionic fields associated with $H$.
  Furthermore, we have
  \begin{equation}
    \| \psi (f) \|_{\mathcal{B} (\mathfrak{H})}, \| \bar{\psi} (f)
    \|_{\mathcal{B} (\mathfrak{H})} \leq 2 \| f \|_H . \label{eq:nonfunziona}
  \end{equation}
  (we use the notation $\| \cdot \|_{\mathfrak{K}}$ for the norm in a Hilbert
  space $\mathfrak{K}$).
\end{theorem}

\begin{proof}
  By standard results of quantum fields theory (see, e.g.,~{\cite{Segal1992}}
  Chapter~2), there are four operators $a, a^{\ast}, b, b^{\ast} : H
  \rightarrow \mathcal{B} (\mathfrak{H})$ (formed by two independent pairs of
  anticommuting creation $a, b$ and anticommuting adjoint annihilation
  $a^{\ast}, b^{\ast}$ operators) \ such that
  \begin{eqnarray}
    & \{ a (f), a (g)^{} \} = \{ b (f), b (g)^{} \} = 0 &  \nonumber\\
    & \{ a (f), b (g)^{} \} = \{ a^{\ast} (f), b (g)^{} \} = 0 &  \nonumber\\
    & \{ a^{\ast} (g), a (f) \} = \{ b^{\ast} (g), b (f) \} = \langle f, g
    \rangle_H I_{\mathfrak{H}}, &  \nonumber
  \end{eqnarray}
  and such that
  \[ \langle a (f) K \rangle = \langle K a^{\ast} (f) \rangle = \langle b (f)
     K \rangle = \langle K b^{\ast} (f) \rangle = 0 \]
  for any $f, g \in H$ and any bounded operator $K \in \mathcal{B}
  (\mathfrak{H})$. Consider now
  \[ \psi (f) = a^{\ast} (f) + b (f), \quad \bar{\psi} (f) = b^{\ast} (f) - a
     (f), \]
  where $f \in H$. We want to prove that $\psi, \bar{\psi}$ are the two
  fermionic fields associated with $H$ fields of the thesis of the theorem.
  Obviously $\{ \psi (f), \bar{\psi} (g) \} = \{ \psi (f), \psi (g) \} = \{
  \bar{\psi} (f), \bar{\psi} (g) \} = 0$, so we have only to prove that $\psi,
  \bar{\psi}$ satisfy equality {\eqref{eq:det}} and inequality
  {\eqref{eq:nonfunziona}}.
  
  We prove equality {\eqref{eq:det}} by induction on $n$. By the properties
  of $a, a^{\ast}, b, b^{\ast}$ we have
  
  \begin{align*}
    \langle \bar{\psi} (f_1) \psi (g_1) \rangle = & \langle b^{\ast} (f_1)
    a^{\ast} (g_1) \rangle + \langle b^{\ast} (f_1) b (g_1) \rangle - \langle
    a (f_1) a^{\ast} (g_1) \rangle +\\
    & - \langle a (f_1) b (g_1) \rangle = \langle f_1, g_1 \rangle_H .
  \end{align*}
  
  Suppose that $\langle \bar{\psi} (f_1) \psi (g_1) \ldots \bar{\psi} (f_{n -
  1}) \psi (g_{n - 1}) \rangle = \det (\langle f_i, g_j \rangle_H)$ we want to
  prove the same equality for $n$ operators. We have
  \begin{multline*}
    \langle \bar{\psi} (f_1) \psi (g_1) \ldots \bar{\psi} (f_n) \psi (g_n)
    \rangle = \langle b^{\ast} (f_1) \psi (g_1) \ldots \bar{\psi} (f_n) \psi
    (g_n) \rangle =\\
    = \sum^n_{k = 1} (- 1)^k \langle b^{\ast} (f_1) b (g_k) \rangle \langle
    \bar{\psi} (f_2) \psi (g_1) \ldots \bar{\psi} (f_k) \psi (g_k) \ldots
    \bar{\psi} (f_n) \psi (g_n) \rangle\\
    = \sum^n_{k = 1} (- 1)^k \langle f_1, g_k \rangle_H \det \left( \langle
    f_i, g_j \rangle |_{i \not{=} 1, j \not{=} k} \right) = \det (\langle f_i,
    g_j \rangle)
  \end{multline*}
  where we use the commutation relations of $a^{\ast}$ with $a, b, b^{\ast}$,
  the induction hypothesis and the properties of determinant. Since
  \[ \| a (f) \|_{\mathcal{B} (\mathfrak{H})} = \| a^{\asterisk} (f)
     \|_{\mathcal{B} (\mathfrak{H})} = \| b (f) \|_{\mathcal{B}
     (\mathfrak{H})} = \| b^{\asterisk} (f) \|_{\mathcal{B} (\mathfrak{H})} =
     \| f \|_H, \]
  $\psi, \bar{\psi}$ satisfy inequality~{\eqref{eq:nonfunziona}}.
\end{proof}

Suppose that $i : H \hookrightarrow C^0 (\mathbb{R}^2)$ for some continuous
injection $i$, then by the identification of $H$ with its dual we have that
$i^{\ast} (\delta_x) \in H$, where $\delta_x \in (C^0 (\mathbb{R}^2))^{\ast}$
is the Dirac delta with mass in $x \in \mathbb{R}^2$. In this way we can
define $\psi, \bar{\psi}$ as continuous functions of the point $\mathbb{R}^2$
in the following way
\[ \psi (x) : = \psi (i^{\ast} (\delta_x)) \qquad \bar{\psi} (x) : =
   \bar{\psi} (i^{\ast} (\delta_x)) \]
and the corresponding covariance function as
\[ S (x ; x') = \langle \bar{\psi} (x') \psi (x) \rangle . \]
Hereafter we suppose that $S (x ; x')$ is a continuous function of the form $S
(x ; x') = S (x - x') \geqslant 0$. In this case, if $g \in L^1
(\mathbb{R}^2)$, by Theorem \ref{theorem_existence12} we have $\| \psi (x)
\bar{\psi} (x) \|_{\mathcal{B} (\mathfrak{H})} \leq 2 S (0)$ and thus
$\int_{\mathbb{R}^2} g (x) \bar{\psi} (x) \psi (x) \mathd x$ is a bounded well
defined operator.

Under the previous condition the operator $\mathfrak{K}_g : L^2 (\mathbb{R}^2)
\rightarrow L^2 (\mathbb{R}^2)$, defined as $\mathfrak{K}_g (h) (x) = \int g
(x) S (x - x') h (x') \mathd x'$, is trace class since
\[ \tmop{Tr} (| \mathfrak{K}_g |) \leq \int_{\mathbb{R}^2} | g (x) |
   \tmop{Tr} (| \bar{\psi} (x) \psi (x) \rho |) \mathd x \leq 2 S (0) \| g
   \|_{L^1 (\mathbb{R}^2)} < + \infty . \]
This means that the Fredholm determinant (see {\cite{Simon2005}} Chapter 3)
$\det (I + \mathfrak{K}_g)$ is well defined and finite. Furthermore, we have
the following representation.

\begin{theorem}
  \label{theorem_representation}Under the previous hypotheses and notations we
  have
  \[ \left\langle \exp \left( \int_{\mathbb{R}^2} g (x) \bar{\psi} (x) \psi
     (x) \mathd x \right) \right\rangle = \det (I + \mathfrak{K}_g) . \]
\end{theorem}

\begin{proof}
  By Definition \ref{def:psi} and the definition of the function $S$, we have
  that
  
  \begin{multline*}
    \left\langle \left( \int_{\mathbb{R}^2} g (x) \bar{\psi} (x) \psi (x)
    \mathd x \right)^n \right\rangle =\\
    = \int_{\mathbb{R}^{2 n}} g (x_1) \ldots g (x_n) \det (S (x_i - x_j))
    \mathd x_1 \ldots \mathd x_n =\\
    = \int_{\mathbb{R}^{2 n}} \det \left( \begin{array}{ccc}
      g (x_1) S (x_1 - x_1) & \ldots & g (x_1) S (x_1 - x_n)\\
      \vdots & \ddots & \vdots\\
      g (x_n) S (x_n - x_1) & \ldots & g (x_n) S (x_n - x_n)
    \end{array} \right) \mathd x_1 \ldots \mathd x_n .
  \end{multline*}
  
  On the other hand, when $S$ is continuous, by Theorem~3.10
  of~{\cite{Simon2005}}, we have that
  
  \begin{multline*}
    \det \left( I +\mathfrak{K}_g \right) = \sum_{n = 0}^{+ \infty} \frac{1}{n!} \int_{\mathbb{R}^{2 n}} \det \left(
    \begin{array}{ccc}
      g (x_1) S (x_1 - x_1) & \ldots & g (x_1) S (x_1 - x_n)\\
      \vdots & \ddots & \vdots\\
      g (x_n) S (x_n - x_1) & \ldots & g (x_n) S (x_n - x_n)
    \end{array} \right) \mathd x_1 \ldots \mathd x_n .
  \end{multline*}
  
  The thesis follows.
\end{proof}

\begin{remark}
  \label{remark_representation}The fermionic fields considered in Section
  \ref{sec:super} and Section \ref{section:supersymmetry}, where $S = \varpi
  \mathcal{G}_{1 + 2 \chi} (x - x')$, $H = W^{1 +2 \chi, 2} (\mathbb{R}^2)$
  with norm $\| f \|_H^2 = \int (- \Delta + m^2)^{1 + 2 \chi} (f) (x) f (x)
  \mathd x$, satisfies all the hypotheses of Theorem
  \ref{theorem_representation}.
\end{remark}


\begin{thebibliography}{10}
  \bibitem[1]{Albeverio2017}S.~Albeverio, F.~C.~De Vecchi, and  S.~Ugolini.
  {\newblock}Entropy chaos and Bose-Einstein condensation.
  {\newblock}\tmtextit{J. Stat. Phys.}, 168(3):483--507, 2017.{\newblock}
  
  \bibitem[2]{albeverio_strong_2012}S.~Albeverio, H.~Kawabi, and 
  M.~R{\"o}ckner. {\newblock}Strong uniqueness for both Dirichlet operators
  and stochastic dynamics to Gibbs measures on a path space with exponential
  interactions. {\newblock}\tmtextit{Journal of Functional Analysis},
  262(2):602--638, jan 2012.{\newblock}
  
  \bibitem[3]{albeverio_invariant_2017}S.~Albeverio  and  S.~Kusuoka.
  {\newblock}The invariant measure and the flow associated to the
  $\Phi^4_3$-quantum field model. {\newblock}\tmtextit{ArXiv:1711.07108}, nov
  2017. {\newblock}To appear in Ann. Scuola Normale di Pisa.{\newblock}
  
  \bibitem[4]{albeverio_stochastic_1991}S.~Albeverio  and  M.~R{\"o}ckner.
  {\newblock}Stochastic differential equations in infinite dimensions:
  solutions via Dirichlet forms. {\newblock}\tmtextit{Probability Theory and
  Related Fields}, 89(3):347--386, 1991.{\newblock}
  
  \bibitem[5]{Albeverio2002}S.~Albeverio  and  M.~W Yoshida. {\newblock}$H -
  C^1$ maps and elliptic spdes with polynomial and exponential perturbations
  of Nelson's euclidean free field. {\newblock}\tmtextit{Journal of Functional
  Analysis}, 196(2):265--322, 2002.{\newblock}
  
  \bibitem[6]{Aliprantis2006}C.~D.~Aliprantis  and  Kim~C.~Border.
  {\newblock}\tmtextit{Infinite dimensional analysis}. {\newblock}Springer,
  Berlin, Third  edition, 2006. {\newblock}A hitchhiker's guide.{\newblock}
  
  \bibitem[7]{arai_supersymmetric_1993}A.~Arai. {\newblock}Supersymmetric
  extension of quantum scalar field theories. {\newblock}In \tmtextit{Quantum
  and non-commutative analysis (Kyoto, 1992)},  volume~16  of \tmtextit{Math.
  Phys. Stud.},  pages  73--90. Kluwer Acad. Publ., Dordrecht,
  1993.{\newblock}
  
  \bibitem[8]{Segal1992}John~C.~Baez, Irving~E.~Segal, and  Zheng-Fang Zhou.
  {\newblock}\tmtextit{Introduction to algebraic and constructive quantum
  field theory}. {\newblock}Princeton Series in Physics. Princeton University
  Press, Princeton, NJ, 1992.{\newblock}
  
  \bibitem[9]{Bahouri2011}H.~Bahouri, J.-Y.~Chemin, and  R.~Danchin.
  {\newblock}\tmtextit{Fourier analysis and nonlinear partial differential
  equations},  volume  343  of \tmtextit{Grundlehren der Mathematischen
  Wissenschaften [Fundamental Principles of Mathematical Sciences]}.
  {\newblock}Springer, Heidelberg, 2011.{\newblock}
  
  \bibitem[10]{Gallavotti1980}G.~Benfatto, G.~Gallavotti, and  F.~Nicol{\`o}.
  {\newblock}Elliptic equations and Gaussian processes.
  {\newblock}\tmtextit{J. Funct. Anal.}, 36(3):343--400, 1980.{\newblock}
  
  \bibitem[11]{Berger1977}M.~S.~Berger. {\newblock}\tmtextit{Nonlinearity and
  functional analysis}. {\newblock}Academic Press [Harcourt Brace Jovanovich,
  Publishers], New York-London, 1977. {\newblock}Lectures on nonlinear
  problems in mathematical analysis, Pure and Applied Mathematics.{\newblock}
  
  \bibitem[12]{borkar_stochastic_1988}V.~S.~Borkar, R.~T.~Chari, and 
  S.~K.~Mitter. {\newblock}Stochastic quantization of field theory in finite
  and infinite volume. {\newblock}\tmtextit{Journal of Functional Analysis},
  81(1):184--206, nov 1988.{\newblock}
  
  \bibitem[13]{brydges_dimensional_2003}D.~C.~Brydges  and  J.~Z.~Imbrie.
  {\newblock}Dimensional Reduction Formulas for Branched Polymer Correlation
  Functions. {\newblock}\tmtextit{Journal of Statistical Physics},
  110(3):503--518, mar 2003.{\newblock}
  
  \bibitem[14]{brydges_branched_2003}D.~Brydges  and  J.~Imbrie.
  {\newblock}Branched polymers and dimensional reduction.
  {\newblock}\tmtextit{Annals of Mathematics}, 158(3):1019--1039, nov
  2003.{\newblock}
  
  \bibitem[15]{da_prato_strong_2003}G.~Da Prato  and  A.~Debussche.
  {\newblock}Strong solutions to the stochastic quantization equations.
  {\newblock}\tmtextit{The Annals of Probability}, 31(4):1900--1916,
  2003.{\newblock}
  
  \bibitem[16]{da_prato_stochastic_2008}G.~Da Prato  and  J.~Zabczyk.
  {\newblock}\tmtextit{Stochastic Equations in Infinite Dimensions}.
  {\newblock}Cambridge University Press, feb 2008.{\newblock}
  
  \bibitem[17]{damgaard_stochastic_1987}P.~H.~Damgaard  and  H.~H{\"u}ffel.
  {\newblock}Stochastic quantization. {\newblock}\tmtextit{Physics Reports},
  152(5-6):227--398, aug 1987.{\newblock}
  
  \bibitem[18]{damgaard_stochastic_1988}P.~H.~Damgaard  and  H.~H{\"u}ffel.
  {\newblock}\tmtextit{Stochastic Quantization}. {\newblock}World Scientific,
  1988.{\newblock}
  
  \bibitem[19]{prato_kolmogorov_2005}G.~DaPrato.
  {\newblock}\tmtextit{Kolmogorov Equations for Stochastic PDEs}.
  {\newblock}Birkh{\"a}user, Basel ; Boston, 2004 edition  edition, feb
  2005.{\newblock}
  
  \bibitem[20]{de_goursac_noncommutative_2015}A.~de~Goursac.
  {\newblock}Noncommutative supergeometry and quantum supergroups.
  {\newblock}Volume  597,  page  12028. Eprint: arXiv:1501.06316, apr
  2015.{\newblock}
  
  \bibitem[21]{dewitt_supermanifolds_1992}B.~DeWitt.
  {\newblock}\tmtextit{Supermanifolds}. {\newblock}Cambridge University Press,
  Cambridge ; New York, 2 edition  edition, may 1992.{\newblock}
  
  \bibitem[22]{Evans1998}L.~C.~Evans. {\newblock}\tmtextit{Partial
  differential equations},  volume~19  of \tmtextit{Graduate Studies in
  Mathematics}. {\newblock}American Mathematical Society, Providence, RI,
  1998.{\newblock}
  
  \bibitem[23]{Fetter2012}A.L.~Fetter  and  J.D.~Walecka.
  {\newblock}\tmtextit{Quantum Theory of Many-Particle Systems}.
  {\newblock}Dover Books on Physics. Dover Publications, 2012.{\newblock}
  
  \bibitem[24]{Gilbarg2001}David Gilbarg  and  Neil~S.~Trudinger.
  {\newblock}\tmtextit{Elliptic partial differential equations of second
  order}. {\newblock}Classics in Mathematics. Springer-Verlag, Berlin, 2001.
  {\newblock}Reprint of the 1998 edition.{\newblock}
  
  \bibitem[25]{Glimm1987}James Glimm  and  Arthur Jaffe.
  {\newblock}\tmtextit{Quantum physics}. {\newblock}Springer-Verlag, New York,
  Second  edition, 1987. {\newblock}A functional integral point of
  view.{\newblock}
  
  \bibitem[26]{Gross1967}L.~Gross. {\newblock}Abstract Wiener spaces.
  {\newblock}In \tmtextit{Proc. Fifth Berkeley Sympos. Math. Statist. and
  Probability (Berkeley, Calif., 1965/66), Vol. II: Contributions to
  Probability Theory, Part 1},  pages  31--42. Univ. California Press,
  Berkeley, Calif., 1967.{\newblock}
  
  \bibitem[27]{GH18}M.~Gubinelli  and  M.~Hofmanov{\'a}. {\newblock}Global
  solutions to elliptic and parabolic $\Phi^4$ models in Euclidean space.
  {\newblock}\tmtextit{ArXiv e-prints}, apr 2018.{\newblock}
  
  \bibitem[28]{hairer_discretisations_2018}M.~Hairer  and  K.~Matetski.
  {\newblock}Discretisations of rough stochastic PDEs.
  {\newblock}\tmtextit{The Annals of Probability}, 46(3):1651--1709, may
  2018.{\newblock}
  
  \bibitem[29]{hairer_tightness_2018}Martin Hairer  and  Massimo Iberti.
  {\newblock}Tightness of the Ising-Kac model on the two-dimensional torus.
  {\newblock}\tmtextit{Journal of Statistical Physics}, 171(4):632--655,
  2018.{\newblock}
  
  \bibitem[30]{helmuth_dimensional_2016}T.~Helmuth. {\newblock}Dimensional
  Reduction for Generalized Continuum Polymers. {\newblock}\tmtextit{Journal
  of Statistical Physics}, 165(1):24--43, oct 2016.{\newblock}
  
  \bibitem[31]{imry_random_field_1975}Y.~Imry  and  S.-K.~Ma.
  {\newblock}Random-Field Instability of the Ordered State of Continuous
  Symmetry. {\newblock}\tmtextit{Physical Review Letters}, 35:1399--1401, nov
  1975.{\newblock}
  
  \bibitem[32]{iwata_infinite_1987}K.~Iwata. {\newblock}An infinite
  dimensional stochastic differential equation with state space $C (R)$.
  {\newblock}\tmtextit{Probability Theory and Related Fields}, 74(1):141--159,
  mar 1987.{\newblock}
  
  \bibitem[33]{jona-lasinio_stochastic_1985}G.~Jona-Lasinio  and 
  P.~K.~Mitter. {\newblock}On the stochastic quantization of field theory.
  {\newblock}\tmtextit{Communications in Mathematical Physics (1965-1997)},
  101(3):409--436, 1985.{\newblock}
  
  \bibitem[34]{khasminskii_stochastic_2011}R.~Khasminskii  and 
  G.~N.~Milstein. {\newblock}\tmtextit{Stochastic Stability of Differential
  Equations}. {\newblock}Springer, Heidelberg; New York, 2nd ed. 2012 edition 
  edition, sep 2011.{\newblock}
  
  \bibitem[35]{klein_supersymmetry_1985}A.~Klein. {\newblock}Supersymmetry and
  a two-dimensional reduction in random phenomena. {\newblock}In
  \tmtextit{Quantum probability and applications, II (Heidelberg, 1984)}, 
  volume  1136  of \tmtextit{Lecture Notes in Math.},  pages  306--317.
  Springer, Berlin, 1985.{\newblock}
  
  \bibitem[36]{Klein1984}A.~Klein, L.~J.~Landau, and  J.~F.~Perez.
  {\newblock}Supersymmetry and the Parisi-Sourlas dimensional reduction: a
  rigorous proof. {\newblock}\tmtextit{Comm. Math. Phys.}, 94(4):459--482,
  1984.{\newblock}
  
  \bibitem[37]{klein_supersymmetry_1983}A.~Klein  and  J.~F.~Perez.
  {\newblock}Supersymmetry and dimensional reduction: A non-perturbative
  proof. {\newblock}\tmtextit{Physics Letters B}, 125(6):473--475, jun
  1983.{\newblock}
  
  \bibitem[38]{klein_density_1985}A.~Klein  and  J.~F.~Perez. {\newblock}On
  the density of states for random potentials in the presence of a uniform
  magnetic field. {\newblock}\tmtextit{Nuclear Physics B}, 251:199--211, jan
  1985.{\newblock}
  
  \bibitem[39]{Ugolini2011}Laura~M.~Morato  and  Stefania Ugolini.
  {\newblock}Stochastic description of a Bose-Einstein condensate.
  {\newblock}\tmtextit{Ann. Henri Poincar{\'e}}, 12(8):1601--1612,
  2011.{\newblock}
  
  \bibitem[40]{MW17}J.-C.~Mourrat  and  H.~Weber. {\newblock}Global
  well-posedness of the dynamic $\Phi^4$ model in the plane.
  {\newblock}\tmtextit{The Annals of Probability}, 45(4):2398--2476, jul
  2017.{\newblock}
  
  \bibitem[41]{krylov_stochastic_1999}J.~Zabczyk,~G.~Da~Prato N.V. Krylov, M.
  R{\"o}ckner. {\newblock}\tmtextit{Stochastic PDEs and Kolmogorov equations
  in infinite dimensions: Lectures}. {\newblock}Lecture Notes in Mathematics
  1715. Springer, 1  edition, 1999.{\newblock}
  
  \bibitem[42]{nelson1966}E.~Nelson. {\newblock}Derivation of the
  Schr{\"o}dinger equation from Newtonian mechanics.
  {\newblock}\tmtextit{Phys. Rev.}, (150):1079--1085, 1966.{\newblock}
  
  \bibitem[43]{MR0214150}E.~Nelson. {\newblock}\tmtextit{Dynamical theories of
  Brownian motion}. {\newblock}Princeton University Press, Princeton, N.J.,
  1967.{\newblock}
  
  \bibitem[44]{Nelson1973}E.~Nelson. {\newblock}The free markoff field.
  {\newblock}\tmtextit{Journal of Functional Analysis}, 12(2):211--227,
  1973.{\newblock}
  
  \bibitem[45]{Nualart2006}D.~Nualart. {\newblock}\tmtextit{The Malliavin
  calculus and related topics}. {\newblock}Probability and its Applications
  (New York). Springer-Verlag, Berlin, Second  edition, 2006.{\newblock}
  
  \bibitem[46]{parisi_random_1979}G.~Parisi  and  N.~Sourlas.
  {\newblock}Random Magnetic Fields, Supersymmetry, and Negative Dimensions.
  {\newblock}\tmtextit{Physical Review Letters}, 43(11):744--745, sep
  1979.{\newblock}
  
  \bibitem[47]{parisi_supersymmetric_1982}G.~Parisi  and  N.~Sourlas.
  {\newblock}Supersymmetric field theories and stochastic differential
  equations. {\newblock}\tmtextit{Nuclear Physics B}, 206(2):321--332, oct
  1982.{\newblock}
  
  \bibitem[48]{parisi_perturbation_1981}G.~Parisi  and  Y.~S.~Wu.
  {\newblock}Perturbation theory without gauge fixing.
  {\newblock}\tmtextit{Scientia Sinica. Zhongguo Kexue}, 24(4):483--496,
  1981.{\newblock}
  
  \bibitem[49]{rogers_supermanifolds_2007}A.~Rogers.
  {\newblock}\tmtextit{Supermanifolds: theory and applications}.
  {\newblock}World Scientific, Singapore; Hackensack, NJ, 2007.
  {\newblock}OCLC: 648316841.{\newblock}
  
  \bibitem[50]{Rudin1973}W.~Rudin. {\newblock}\tmtextit{Functional analysis}.
  {\newblock}McGraw-Hill Book Co., New York-D{\"u}sseldorf-Johannesburg, 1973.
  {\newblock}McGraw-Hill Series in Higher Mathematics.{\newblock}
  
  \bibitem[51]{Zaboronsky1997}Albert Schwarz  and  Oleg Zaboronsky.
  {\newblock}Supersymmetry and localization. {\newblock}\tmtextit{Comm. Math.
  Phys.}, 183(2):463--476, 1997.{\newblock}
  
  \bibitem[52]{Simon1974}B.~Simon. {\newblock}\tmtextit{The $P (\phi)_2$
  Euclidean (quantum) field theory}. {\newblock}Princeton University Press,
  Princeton, N.J., 1974. {\newblock}Princeton Series in Physics.{\newblock}
  
  \bibitem[53]{Simon2005}B.~Simon. {\newblock}\tmtextit{Trace ideals and their
  applications},  volume  120  of \tmtextit{Mathematical Surveys and
  Monographs}. {\newblock}American Mathematical Society, Providence, RI,
  Second  edition, 2005.{\newblock}
  
  \bibitem[54]{MR2768734}W.~Stannat. {\newblock}Stochastic partial
  differential equations: Kolmogorov operators and invariant measures.
  {\newblock}\tmtextit{Jahresber. Dtsch. Math.-Ver.}, 113(2):81--109,
  2011.{\newblock}
  
  \bibitem[55]{Trie1983}Hans Triebel. {\newblock}\tmtextit{Theory of function
  spaces},  volume~78  of \tmtextit{Monographs in Mathematics}.
  {\newblock}Birkh{\"a}user Verlag, Basel, 1983.{\newblock}
  
  \bibitem[56]{Ustunel2000}A.~S.~{\"U}st{\"u}nel  and  M.~Zakai.
  {\newblock}\tmtextit{Transformation of measure on Wiener space}.
  {\newblock}Springer Monographs in Mathematics. Springer-Verlag, Berlin,
  2000.{\newblock}
  
  \bibitem[57]{Van2003majorants}J.~van der~Hoeven.
  {\newblock}\tmtextit{Majorants for formal power series}.
  {\newblock}Citeseer, 2003.{\newblock}
  
  \bibitem[58]{Wegner2016}Franz Wegner. {\newblock}\tmtextit{Supermathematics
  and its applications in statistical physics},  volume  920  of
  \tmtextit{Lecture Notes in Physics}. {\newblock}Springer, Heidelberg, 2016.
  {\newblock}Grassmann variables and the method of supersymmetry.{\newblock}
  
  \bibitem[59]{young_lowering_1977}A.~P.~Young. {\newblock}On the lowering of
  dimensionality in phase transitions with random fields.
  {\newblock}\tmtextit{Journal of Physics C: Solid State Physics}, 10(9):0,
  may 1977.{\newblock}
  
  \bibitem[60]{Zinn1993}J.~Zinn-Justin. {\newblock}\tmtextit{Quantum field
  theory and critical phenomena},  volume~85  of \tmtextit{International
  Series of Monographs on Physics}. {\newblock}The Clarendon Press, Oxford
  University Press, New York, Second  edition, 1993. {\newblock}Oxford Science
  Publications.{\newblock}
\end{thebibliography}
\end{document}